\documentclass[10pt,reqno,final]{amsart}
\usepackage{epsfig,amssymb,amsmath,mathtools,version}
\usepackage{amssymb,version,graphicx,fancybox,mathrsfs}
\usepackage{hyperref}
\hypersetup{colorlinks=true,citecolor=blue,linkcolor=red,anchorcolor=blue}
\usepackage[notcite,notref]{showkeys}
\usepackage{subfigure}
\usepackage{color}
\usepackage{stmaryrd}
\usepackage{multirow}
\usepackage{booktabs,siunitx}
\usepackage{multicol}
\usepackage{enumitem}
\usepackage{epstopdf}

\usepackage[marginparwidth=2.5cm, marginparsep=3mm]{geometry}

\usepackage{textgreek}
\usepackage{upgreek}

\usepackage{relsize}

\DeclareTextAccent{\myacc}{T1}{4}

\usepackage{marginnote}

\catcode`\@=11 \theoremstyle{plain}
\@addtoreset{equation}{section}   
\renewcommand{\theequation}{\arabic{section}.\arabic{equation}}
\@addtoreset{figure}{section}
\renewcommand\thefigure{\thesection.\@arabic\c@figure}
\newtheorem{thm}{\bf Theorem}[section]

\newenvironment{theorem}{\begin{thm}} {\end{thm}}
\newtheorem{cor}{\bf Corollary}

\newenvironment{corollary}{\begin{cor}} {\end{cor}}
\newtheorem{lmm}{\bf Lemma}[section]

\newenvironment{lemma}{\begin{lmm}}{\end{lmm}}
\theoremstyle{remark}
\newtheorem{rem}{\bf Remark}[section]

\definecolor{ligreen}{rgb}{0.0, 0.3, 0.0}

\definecolor{darkblue}{rgb}{0.0, 0.0, 0.55}

\definecolor{anti-flashwhite}{rgb}{0.55, 0.57, 0.68}

\newcommand{\refl}[1]{Lemma~{\rm \ref{#1}}}
\newcommand{\reft}[1]{Theorem~{\rm \ref{#1}}}

\def \ri {{\rm i}}
\def \re {{\rm e}}
\def \rd {{\rm d}}

\def \bC {\mathbb{C}}

\allowdisplaybreaks

\graphicspath{{./Figures/} } 

\def\Om{{\mathbb T^d}}
\def\lam{\lambda}

\begin{document}
\bibliographystyle{plain}

\title[Lie-Totter time-splitting method for LogSE] {Low regularity  estimates of the Lie-Totter time-splitting  Fourier spectral method  for the logarithmic Schr\"{o}dinger equation}
\author[X. Zhang\, \&\, L. Wang]{
		\;\; Xiaolong Zhang$^{\dag}$ \; and \;  Li-Lian Wang$^{\ddag}$
		}
	\thanks{${}^{\dag}$MOE-LCSM, CHP-LCOCS, School of Mathematics and Statistics, Hunan Normal University, Changsha, 410081, China.  The research of this author is supported in part by the National Natural Science Foundation of China (No.12101229) and the Hunan Provincial Natural Science Foundation of China (No.2021JJ40331). Email: xlzhang@hunnu.edu.cn. \\
			\indent ${}^{\ddag}$Corresponding author. Division of Mathematical Sciences, School of Physical
		and Mathematical Sciences, Nanyang Technological University,
		637371, Singapore.  Email: lilian@ntu.edu.sg (L. Wang).\\
		\indent The first author would like to acknowledge the support of  China Scholarship Council (CSC, No. 202106720024) for the visit of NTU}


\begin{abstract} In this paper, we conduct rigorous error analysis of the Lie-Totter time-splitting Fourier spectral scheme for the nonlinear Schr\"odinger  equation  with a logarithmic nonlinear term $f(u)=u\ln \!|u|^2$ (LogSE)  and periodic boundary conditions on a $d$-dimensional torus $\mathbb T^d$. Different from  existing works based on regularisation of the nonlinear term $ f(u)\approx f^\varepsilon(u)=u\ln (|u| + \varepsilon )^2,$ we directly discretize the LogSE with the understanding  $f(0)=0.$ Remarkably, in the time-splitting scheme,    the  solution flow map of the nonlinear part:   $g(u)= u\,\re^{-\ri   t \ln\!|u|^{2}}$  has a higher regularity than $f(u)$ (which is not differentiable at $u=0$ but H\"older continuous), where $g(u)$ is Lipschitz continuous and possesses a certain fractional Sobolev regularity with index $0<s<1$.  Accordingly,  we can  derive the $L^2$-error estimate: $\mathcal{O}\big((\tau^{s/2} + N^{-s})\ln\! N\big)$  of the proposed scheme for the LogSE with low regularity solution $u\in C((0,T]; H^s( \mathbb{T}^d)\cap L^\infty( \mathbb{T}^d)).$ Moreover, we can show that the estimate holds for $s=1$ with more delicate analysis of the nonlinear term and the associated solution flow maps. Furthermore, we provide ample numerical results  to demonstrate such a fractional-order convergence for  initial data with low regularity.   This work is the first one devoted to the analysis of splitting scheme for the LogSE without regularisation in the low regularity setting, as far as we can tell.   
\end{abstract}
\maketitle

\section{Introduction}
\setcounter{equation}{0}
\setcounter{lmm}{0}
\setcounter{thm}{0}

 This paper concerns the (dimensionless)
  logarithmic Schr\"odinger equation (LogSE): 
\begin{equation}\label{eq:LogSE}
    \ri\, \partial_t u({x}, t) + \Delta u( {x}, t) = \lam u( x, t) \ln  | u( x, t)|^2; \quad  
    u({x}, 0) = u_0( x ),
\end{equation}
where $x\in \Omega\subseteq {\mathbb R^d}$ with $d\ge 1$ and  
$\lambda\in \mathbb{R}\setminus\{0\},$ and suitable boundary conditions should be imposed if $\Omega$ is bounded. 
Here, our focus is  on the analysis of the time-splitting scheme, so we consider the LogSE on the $d$-dimensional torus  $ \Omega= \mathbb{T}^d=[-\pi,\pi)^d$ with $2\pi$-periodic boundary conditions.    

The  LogSE was first introduced as a model for nonlinear wave mechanics by  Bia{\l}ynicki-Birula and Mycielski
\cite{Bialynicki1976nonlinear} and later found many applications in  physics and engineering (see e.g., \cite{Hefter1985Application,krolikowski2000unified,De2003logarithmic,Buljan2003Incoherent,Hansson2009propagation,Zloshchastiev2010Logarithmic,Avdeenkov2011quantum}). It has the  mass, momentum and energy conservation in common with usual nonlinear Schr\"odinger equations,  but possesses some unusual properties and richer dynamics which the usual ones may not have (e.g.,   ``energy'' with indefinite sign;  existence of exact {\em soliton-like  Gausson}  (standing wave) solutions (when $\lambda<0$); and  tensorisation property among others (see \cite{Bialynicki1976nonlinear,Carles2018Universal})).

One main  challenge of the LogSE is the low regularity of the nonlinear term $f(u)=u \ln|u|^2$. Its mathematical and numerical studies have attracted much recent attention. The investigation of the Cauchy problem dates back to  Cazenave and Haraux \cite{Cazenave1980}. We refer to \cite{Carles2018Universal,carles_logarithmic_2022,hayashi_cauchy_2023,carles_low_2023} (and the references therein)
for comprehensive review for old and new  results on existence, uniqueness and regularity in different senses and settings, together with some open questions. Here, we collect some relevant results:
\begin{itemize}
\item[(i)] As one part of the main results in  \cite[Theorem 1.1]{carles_low_2023}, 
{\em the LogSE \eqref{eq:LogSE} on $\Omega\in \{\mathbb R^d, \mathbb R^d_+,  \mathbb T^d\}$ with 
 $u_0 \in H^s(\Omega)$ for $0<s<1,$ has a unique solution $u \in C(\mathbb{R}; H^s(\Omega)).$ }
 \smallskip
\item[(ii)] According to \cite[Theorem 2.3]{carles_low_2023} (also see \cite{hayashi_cauchy_2023}), 
{\em the LogSE \eqref{eq:LogSE} on $\Omega={\mathbb T^d}$ with 
 $u_0 \in H^1(\mathbb{T}^d),$ has a unique solution $u \in C(\mathbb{R}; H^1(\mathbb{T}^d)).$ }
\end{itemize}
It is noteworthy that the proof of (i) in \cite{carles_low_2023}  is based on 
the analysis of the regularized LogSE (RLogSE):
\begin{equation}\label{eq:RLogSE}
    \ri \, \partial_t u^{\varepsilon} + \Delta u^{\varepsilon} = \lam u^{\varepsilon}\ln (  | u^{\varepsilon}| + \varepsilon    )^2 ;
    \quad u^{\varepsilon}|_{t=0} = u_0,
\end{equation}
and the equivalent representation of the 
fractional Sobolev norm  (see  \cite[Proposition 1.3]{BenyiOH2013fracSoblevTorus} and   \refl{lem:normEquiv} for $\Omega={\mathbb T}^d$). 

The numerical study  of the LogSE has  been a research topic of much recent interest. In general, the existing attempts can be classified into two categories:
\begin{itemize}
    \item[(a)] {\bf Regularised numerical methods.} 
    Most  existing works were on the discretisation of the RLogSE \eqref{eq:RLogSE}, in order to avoid the blowup of $\ln |u|$ at $u=0$ in the original LogSE \eqref{eq:LogSE}. Bao et al  \cite{Bao2019Error} first proposed and analyzed   
    the semi-implicit Crank-Nicolson leap-frog in time and central difference in  space for the RLogSE in a bounded domain $\Omega$.  It was shown therein that
 if $u_0\in H^2(\Omega),$ then
\begin{equation}\label{uepsest}
\|u^{\varepsilon}-u\|_{L^{\infty}(0, T ; L^{2}(\Omega))} \leq C_{1} \varepsilon, \quad\|u^{\varepsilon}-u\|_{L^{\infty}(0, T ; H^{1}(\Omega))} \leq C_{2} \sqrt \varepsilon,
\end{equation}
where the positive constants $C_1=C_1(|\lambda|, T,|\Omega|)$ and $C_2=C_2(|\lambda|, T,|\Omega|,\|u_{0}\|_{H^{2}(\Omega)}).$  However, the error estimate contains an unpleasant  factor ${\rm e}^{C|\!\ln \varepsilon|^2T}$ and the required  regularity was imposed on   $u^{\varepsilon}.$   Bao et al \cite{Bao2019Regularized} further analyzed  the  Lie-Trotter time-splitting scheme and obtained the $L^2$-error bound
\begin{equation}\label{uepsest22}
\|u^{\varepsilon, n}-u^{\varepsilon}(t_n)\|_{L^2(\Omega)} \leq C(T,\|u^{\varepsilon}\|_{L^{\infty}([0, T] ; H^2(\Omega))}) \ln (\varepsilon^{-1})\, \sqrt{\tau},
\end{equation}
for $d=2,3.$ In \cite{Bao2019Regularized}, the Fourier spectral method was employed for spatial discretisation, but the error estimate for the full discretised scheme was not conducted. In  \cite{Carles2022numerical}, the  regularized time-splitting and Fourier spectral method was applied  to numerically solve the LogSE with a harmonic potential to demonstrate the interesting dynamics (see \cite{Carles2021Logarithmic,Carles2021nonuniqueness}). In addition,  Bao et al \cite{Bao2022Error} investigated   a different  regularisation from the energy perspective. 
\vskip 3pt
\item[(b)] {\bf Direct  (non-regularised) numerical methods.} 
 Although  $f(u)=u\ln|u|^2$ is not differentiable at 
  $u=0,$   it is continuous if we define $f(0)=0$.
  It is therefore feasible to directly discretise $f(u)$ in a grid-based method.   Paraschis and  Zouraris  \cite{Paraschis2023FixedpointLogSE}   analysed  the (implicit) Crank-Nicolson scheme
  for  \eqref{eq:LogSE}, where a nonlinear system had to be solved at each time step. Moreover, the non-differentiability of $f(u)$ ruled out the use of Newton-type iterative methods, so 
   the fixed-point iteration was employed therein. 
Wang et al \cite{WangYan2022LogSE} analyzed for the first time the linearized implicit-explicit (IMEX) time-discretisation with finite element approximation in space for the LogSE without regularization. However, it is typical to require strong regularity to achieve the expected order of convergence. 
\end{itemize}


\smallskip
The purpose of this paper is to analyze 
the Lie-Totter time-splitting and Fourier spectral method for the LogSE (without regularisation)  with low regularity initial data and solution as in (i)-(ii) (see  \cite{carles_low_2023,hayashi_cauchy_2023}).  More specifically, we shall show that
under the regularity: $u_0\in L^2(\mathbb T^d)$ and 
$ u \in C((0,T]; H^{s}({\mathbb T^d})\cap L^\infty( \mathbb{T}^d ))$ for some $0 < s < 1,$ the $L^2$-error  is of order $\mathcal{O}\big((\tau^{s/2} + N^{-s})\ln\! N\big)$ (see \reft{THM:errestfrac}). 
Then we shall further prove such an estimate can be extended to $s=1$ (see \reft{THM:errestH1}).
Some critical tools and arguments for the analysis include 
\begin{itemize}
\item Optimal characterisation of the fractional Sobolev $H^s$-regularity of the flow map $\Phi_B^t[u]$ associated with the nonlinear part (see \reft{PhiBbnd}), where the equivalent form of the fractional Sobolev norm in \refl{lem:normEquiv} are crucial.
\smallskip
\item Delicate analysis  of the regularity of the nonlinear term $f(u)$ and its regularized counterpart $f^\varepsilon(u),$ likewise for  
$\Phi_B^t[u]$ and $\Phi_B^{t,\varepsilon}[u]$ (see Sections \ref{sect3:Error}-\ref{sec:H1mainResult}).
\end{itemize} 
We shall provide ample numerical evidences to demonstrate such fractional-order convergence behaviours with given low regularity initial data as in the theorems.  

We regard this paper as the first work devoted to the analysis of time-splitting scheme of the LogSE without regularisation and in low regularity setting. It is in line with the growing recent interest in numerical study of nonlinear Schr\"odinger equations with singular nonlinearity e.g., $f(u)=|u|^{\sigma} u$ with $\sigma<1$ (see 
\cite{bao_error_2023,bao_optimal_2023b}); or  $\sigma=2$ but with  low regularity initial data (see \cite{ostermann_low_2018,ostermann_fourier_2022,cao_new_2023}).

The rest of this paper is organized as follows. In Section \ref{sec:LTS}, we introduce the Lie-Trotter time-splitting scheme for the LogSE \eqref{eq:LogSE} and present the important properties of the two flow maps. In Section \ref{sect3:Error}, we introduce the full-discretisation scheme and 
 derive  its fractional-order convergence result.     In Section \ref{sec:H1mainResult}, we extend the estimate to the case $s=1.$ 
We present 
in Section \ref{sec:numerresults} some  
numerical results and demonstrate the predicted convergence behaviors for the LogSE with  low regularity initial data.

\section{The Lie-Trotter time-splitting scheme for the LogSE}\label{sec:LTS}
In this section, we introduce the Lie-Trotter splitting scheme
for time discretization of the LogSE \eqref{eq:LogSE}, and then derive some important  properties of the solution flow maps associated with the two  sub-problems. Most importantly, we study their stability and lower regularity in the fractional periodic Sobolev space.   

\subsection{Fractional periodic  Sobolev space}\label{subsection:fracSob}
Throughout this paper,  let $\mathbb{C},$ $\mathbb{R},$ $\mathbb{Z}$,  $\mathbb{N}$, and $\mathbb{N}^{+}$  be the set of
all complex numbers, real numbers,  integers, non-negative integers, and positive integers, respectively. In particular, we denote ${\mathbb Z}_0={\mathbb Z}\setminus\{0\}.$
Further, let $x = (x_1,\cdots,x_d)\in {\mathbb R}^d$  and  $|x|$ be its length, likewise for $k\in {\mathbb Z}^d$ and $|k|.$ 

  We are concerned with $d$-dimensional $2\pi$-periodic functions  defined on the $d$-dimensional torus: $\mathbb{T}^d=[-\pi,\pi)^d$.  For any $u\in L^2(\mathbb T^d),$ we expand 
\begin{equation}\label{FS}
u(x) = \sum_{ k \in \mathbb{Z}^{d} }  \hat{u}_{k}  \re^{ \ri k \cdot x}, \quad \hat{u}_k= \frac{1}{|\mathbb T|^d}\int_{\mathbb{T}^d} u(x) \re^{-\ri  k \cdot x}\, \rd x,
\end{equation}
where $|\mathbb T|=2\pi.$ We define the usual fractional periodic Sobolev space
\begin{equation}\label{def:FSF}
	H^{s}({\mathbb T^d}) := \big\{ u\in L^2( {\mathbb T^d}) \; : \;  \| u \|_{H^{s}({\mathbb T^d}) }< \infty \big\},\quad s\ge 0,
\end{equation}
endowed with the norm and semi-norm
\begin{equation}\label{FSnormsemi} 
\| u \|_{ H^{s}( {\mathbb T^d} )}: = |\mathbb T|^{\frac d 2}\bigg(\sum_{k  \in \mathbb{Z}^d }   (1 + |k|^2)^{s}\,  |\hat{u}_k |^2 \bigg)^{\frac 1 2}, \quad
|u|_{H^{s}({\mathbb T^d}) }:=  |\mathbb T|^{\frac d 2} \bigg(  \sum_{ k \in \mathbb{Z}_0^d } | k |^{2s} \, |\hat{u}_k |^2 \bigg)^{\frac 12}.
\end{equation}
In particular, if $s=0,$ we denote the $L^2$-norm simply by $\|\cdot\|.$ 


Remarkably,  if $0<s<1,$  the above semi-norm expressed in frequency space has the following equivalent characterization in physical space. Such an equivalence plays an exceedingly important role in the forthcoming error analysis.  
\begin{lemma}\label{lem:normEquiv} 
If $u \in {H}^{s} ( \mathbb{T}^d)$ with  $0< s < 1$, we have 
\begin{equation}\label{equivdefn}
{\mathcal C}_{1}\, |u |_{ {H}^s ( \mathbb{T}^d ) }\le  \Big(\int_{  \mathbb{T}^d } \int_{ \mathbb{T}^d } \frac{|u(x+y) - u(y)|^2}{|x|^{d+2s}} \rd x \rd y\Big)^{\frac 1 2}\le {\mathcal C}_{2}\, |u |_{{H}^s ( \mathbb{T}^d )},
\end{equation}
where ${\mathcal C}_1,{\mathcal C}_2$ are positive constants independent of $u.$
\end{lemma}
\begin{proof} The proof  follows from B$\acute{e}$nyi and Oh  \cite[Proposition 1.3]{BenyiOH2013fracSoblevTorus} directly with a change of  periodicity from $1$ to $2\pi,$ 
i.e., the domain from   $[-\frac 1 2, \frac 1 2)^d$ to $\mathbb T^d=[-\pi, \pi)^d.$  
\end{proof}

It is also important to remark that the situation is reminiscent of the  frequency-physical equivalence in the fractional Sobolev space  $W^{s,2}(\mathbb R^d)$ 
(see e.g., \cite{Demengel2012FracSobolev}).  
Let $\Omega$ be a general, possibly nonsmooth, open domain in ${\mathbb R}^d,$ and recall the  Gagliardo semi-norm of the fractional Sobolev space $W^{s,2}(\Omega):$  
  \begin{equation}\label{defnGag}
[u]_{W^{s,2}(\Omega)}:=\Big(\int_{\Omega} \int_{\Omega} \frac{ |u(x) - u(y)|^2}{|x-y|^{d+2s}} \rd x \rd y\Big)^{\frac 12},\quad 0<s<1.
\end{equation}
In particular,  when $\Omega={\mathbb R}^d,$  the  
Gagliardo semi-norm of $W^{s,2}(\mathbb R^d)$ has the equivalent forms
\begin{equation}\label{defnGag2}
\begin{split}
[u]_{W^{s,2}(\mathbb R^d)}^2 & = \int_{\mathbb R^d} \int_{\mathbb R^d} \frac{ |u(x) - u(y)|^2}{|x-y|^{d+2s}} \rd x \rd y= \int_{ \mathbb{R}^d } \int_{ \mathbb{R}^d } \frac{ |u(x + y ) - u(y)|^2}{|x|^{d+2s}} \rd x \rd y
 \\
 &\cong 
 \int_{\mathbb{R}^d}\big(1+|\xi|^{2 s}\big)|\mathscr{F} [u](\xi)|^2 {\rm d} \xi,
\end{split}
\end{equation}
where $\mathscr{F} [u](\xi)$ is the Fourier transform of $u.$ However, the identity in the first line of \eqref{defnGag2} does not hold for 
$u\in H^s(\mathbb T^d),$ so $W^{s,2}(\mathbb T^d)\not= H^s(\mathbb T^d).$

It is also noteworthy that the same setting in \eqref{FS}-\eqref{FSnormsemi} can be applied to tensorial Legendre polynomial expansions to define the corresponding fractional Sobolev space $W^{s,2}(\Omega)$ with $\Omega=(-1,1)^d$ (see e.g.,
\cite{Shen2011Book}). However, it is still unknown if there exists a similar equivalent form of the semi-norm in physical space as in \refl{lem:normEquiv}. 

\subsection{Lie-Trotter time-splitting scheme} 
We  formulate the LogSE \eqref{eq:LogSE} as
\begin{equation}\label{LogSE:split}
\ri \,\partial_t u  = Au + Bu,
\end{equation}
with
\begin{equation*}
Au:= - \Delta u, \quad Bu:=  \lambda f(u)=\lambda u\ln |u|^2,
\end{equation*}
and then split \eqref{LogSE:split} into two separated flows:
\begin{subequations}\label{ABproblems}
    \begin{equation}\label{eq:lieeq1}
\hspace*{-45pt}{\rm (A)}: \quad
	\begin{dcases}
		\ri\, \partial_t v(x, t) =-\Delta v(x, t),    \\
		  v(x, 0) = v_0(x),  
	\end{dcases}
\end{equation}
\vskip 1pt
\begin{equation}\label{eq:lieeq3}
	{\rm (B)}: \quad \begin{dcases}
		\ri\, \partial_t w({x}, t) =  \lam w( {x},t)\ln| w( x, t)|^{2},\\ 
	w({x}, 0) = w_0({x}).
	\end{dcases}
\end{equation}
\end{subequations}
Formally, we can express their exact solutions in terms  of the flow maps:
\begin{subequations}\label{ABFlow-maps}
\begin{alignat}{4}
& {\rm (A)}: \quad  \Phi_A^t [v_0](x) \;:= \re^{\ri t \Delta} v_0(x)\,=v(x,t), \label{PhiA} \\
&{\rm (B)}: \quad \Phi_{B}^{t} [ w_0](x) := \re^{-\ri \lam  t \ln|w_0(x)|^{2}} w_0(x)\,=w(x,  t). \label{PhiB}
\end{alignat}
\end{subequations}
As the usual cubic Schr\"odinger equation, the solution formula  \eqref{PhiB}
can be  obtained by multiplying \eqref{eq:lieeq3} by the conjugate  $\bar w$ and taking the imaginary part: 
\begin{equation*}\label{Msubeqn2}
\ri\, \Re  \{\partial_{t} w \cdot \bar w\}= \frac{\ri}{2} {\partial_t} |w|^2  = \lambda \Im \{ |w|^2 \ln |w|^2 \, \} = 0,
\end{equation*}
which implies $|w(x,t)| = |w_0(x)|.$ Replacing $\ln |w|^2$ by $\ln |w_0|^2$ in  \eqref{eq:lieeq3}, we
 can solve  \eqref{eq:lieeq3} straightforwardly and derive \eqref{PhiB}.


Let $\tau>0$ be the time-stepping size, and  $t_m=m\tau$ for $ 0 \le m \le M:=[T/ \tau]$ and given final time $T>0.$ 
The Lie-Trotter time-splitting for the LogSE  \eqref{LogSE:split} is to find the approximations $\big\{U^m(x)\approx u(x,m\tau)\big\}_{m=1}^M$ recursively through 
\begin{equation}\label{eq:splititer}
\begin{dcases}
 U^{m} (x) =  \Phi_{A}^{\tau} \, \Phi_{B}^{\tau}[ U^{m-1} ](x),\;\;\; m=1,\ldots,M, \\
 U^0(x)=u_0(x),
 \end{dcases}
\end{equation}
or equivalently,
\begin{equation}\label{eq:logSplitSol}
 U^{m} (x)=( \Phi_{A}^{\tau}\,\Phi_{B}^{\tau})^{m} [ U^0](x). 
\end{equation}

It is worth mentioning that the above derivation and scheme  work for several typical types of   domains and boundary conditions, for example, 
(i) $\Omega=\mathbb T^d$ with periodic boundary conditions; (ii) $\Omega={\mathbb R}^d$ with far-field decaying boundary conditions; and (iii) $\Omega$ being a bounded domain with homogeneous Dirichlet or Neumann boundary conditions. Although our focus is on the periodic case (i), we shall indicate if a property is valid for other domains and boundary conditions.    For example, one verifies readily that it preserves the mass: 
\begin{equation}\label{eq:mass}
\| U^{m+1}\|= \| U^{m}\|=\cdots= \| u_0\|.
\end{equation}

\begin{rem}\label{Rmk:Diff-Bao} {\em In 
Bao et al \cite{Bao2019Regularized}, the time-splitting scheme was implemented on the regularized LogSE with the flow map $\Phi_B^{t,\varepsilon}[w_0]= \re^{-\ri \lam  t \ln( |w_0| + \varepsilon )^{2}} w_0$ to avoid the blowup of $\ln |w_0|$ at $|w_0|=0.$  In contrast, we directly work at the LogSE and  understand that $\Phi_B^t[w_0]= \re^{-\ri \lam  t \ln|w_0|^{2}} w_0=0,$ when $|w_0|=0.$ \qed} 
\end{rem}

\subsection{Analysis of the flow maps}  

As we shall see in the forthcoming section, the  following $L^2$-estimate
of the  flow map $\Phi_{A}^{t}$ of the linear Schr\"odinger equation \eqref{eq:lieeq1} plays a crucial role in the error analysis.
 Importantly, we can justify the order in $t$ is optimal, so the initial data should have $H^2$-regularity when we say it's a first-order scheme.
\begin{theorem}\label{PhiA:Projerr}
	If $v_0 \in {H}^{r}( {\mathbb T^d} ) $ with $0\le r \le 2$ and $d\ge 1$, then for $t>0,$
	\begin{equation}\label{PhiA:projerrine}
		 \big\|   \Phi_A^{t}  [  v_0  ] -v_0 \big\|  \le 2^{ 1 -\frac r 2 }  \,t^{ \frac r2}\, | v_0 |_{ {H}^{r}(\Om) },
	\end{equation}
 and for $0<t<1,$
 \begin{equation}\label{PhiA:projerrineB}
		 \big\|   \Phi_A^{t}  [  v_0  ] -v_0 \big\|  \ge 2^{ 1 -\frac r 2 } |\mathbb{T}|^{\frac d2} \, c_0 \,t^{ \frac r2}\, \Big(\sum_{ k \in \mathbb{K}_{c_0} }  |k|^{2r}  | \hat{v}_{0,k} |^2\Big)^{\frac 12}, 
	\end{equation}
	where $\{\hat{v}_{0,k}\}$ are the Fourier expansion coefficients of $v_0(x),$ and 
 \begin{equation}\label{Kc0}
\mathbb{K}_{c_0}:= \Big\{k\in \mathbb{Z}^d : \frac 2 t \sin^{-1}(c_0)  \le |k|^2 \le  \frac 2  t \,{\rm sinc}^{-1}(c_0) \Big\},
\end{equation}
\end{theorem}
\begin{proof} We expand $v_0$ in its Fourier series and find from  \eqref{PhiA} readily that
\begin{equation*}
v_0(x) = \sum_{ k \in \mathbb{Z}^d}   \hat{v}_{0,k} \,  \re^{\ri k \cdot x},\quad    \Phi_A^{t} [v_0] = \sum_{ k \in \mathbb{Z}^d}  \re^{-\ri | k |^2 t} \hat{v}_{0,k} \,  \re^{\ri k \cdot x},
\end{equation*}
so by the Parsval's identity, 
	\begin{equation}\label{eq:pfProjerr00}
	\begin{split}
			\big\|   \Phi_A^{t}  [ v_0 ] -v_0 \big\|^2  &  =  |\mathbb{T}|^d  \sum_{ k \in \mathbb{Z}^{d} } \big| 1- \re^{-\ri | k |^2 t} \big|^2 \, | \hat{v}_{0,k} |^2= 4  |\mathbb{T}|^d \sum_{ k \in \mathbb{Z}_0^{d} } \sin^2 \Big(\frac{1}{2} | k |^2 t \Big) | \hat{v}_{0,k} |^2.
			\end{split}
	\end{equation}
One verifies readily that 
\begin{equation}\label{sinxrineq}
|\sin \xi|\le \xi^{\frac r 2}, \quad  \xi \ge 0,\;\; 0\le r\le 2. 
\end{equation}
Indeed, it is obviously true for  $\xi\ge 1,$ while for $\xi\in [0,1),$  it is a direct consequence of the fact: $\sin \xi \le \xi\le \xi^{\frac r2}.$ Thus, 
the estimate \eqref{PhiA:projerrine} follows from 
 \eqref{eq:pfProjerr00}-\eqref{sinxrineq} and the definition of the semi-norm in  \eqref{FSnormsemi}. 

To derive the lower bound, 
we study  the sinc-type kernel function:
\begin{equation}\label{sincK}
g(\xi):=g(\xi;s):=\frac{\sin \xi } {\xi^s},\quad \xi\in (0,\pi),\;\; s=\frac r 2\in [0,1],
\end{equation}
where for $s=1,$ $g(\xi)={\rm sinc}(\xi).$ One verifies readily that  $g(\xi)$
  attains its unique maximum value at $\xi=\xi_*=\xi_*(s),$ where $\xi_*\in (0,\pi/2)$ satisfies  $\xi_*=s \tan \xi_*.$  Moreover, for $s\in [0,1],$ we have
  \begin{equation}
  \nonumber
g(\xi_*)=\max_{\xi\in [0,\pi]} g(\xi)\ge g(1)=\sin 1\approx 0.84147.
\end{equation}
On the other hand, for fixed $\xi\in (0,1]$ (resp.  $\xi\in (1,\pi]$), $\big(\frac 1 {\xi}\big)^s$ is monotonically increasing (resp. decreasing) in $s,$ so we have 
\begin{equation}\label{gboundA}
g(\xi;s)\ge G(\xi):=\begin{dcases}
g(\xi,0)=\sin \xi, & {\rm for}\;\; 0\le \xi\le 1,\\[2pt]
g(\xi,1)=\frac{\sin \xi}\xi, & {\rm for}\;\;  1<\xi \le \pi,  
\end{dcases}
\end{equation}
and  $G(\xi)\in [0,\sin 1]$ (see Figure \ref{fig:sincmaximum}).  Thus we deduce from the intermediate value theorem that for any $c_0\in (0,\sin 1),$ 
 \begin{equation}\label{gboundB}
g(\xi;s)\ge G(\xi)\ge  c_0,\quad \forall\, \xi \in \big[\sin^{-1}(c_0),\, {\rm sinc}^{-1}(c_0)\big].   
\end{equation}
Using  \eqref{gboundB} with $\xi =\frac 1 2 |k|^2 t$ and $s= r/2,$ we find from  \eqref{Kc0},  \eqref{eq:pfProjerr00} and \eqref{sincK}    that
\begin{equation}\label{eq:pfProjerr11}
	\begin{split}
			\big\|   \Phi_A^{t}  [ v_0 ] -v_0 \big\|^2  &  \ge  4 |\mathbb{T}|^d \sum_{ k \in \mathbb{K}_{c_0}} \sin^2 \Big(\frac{1}{2} | k |^2 t \Big) | \hat{v}_{0,k} |^2\\
   & \ge 2^{2-r}\, |\mathbb{T}|^d \,c_0^2 \, t^r\,  \sum_{ k \in \mathbb{K}_{c_0} }  |k|^{2r}  | \hat{v}_{0,k} |^2,
			\end{split}
	\end{equation}
 which completes the proof.  
\end{proof}

\begin{rem}\label{optimal-order}
{\em The estimates in Theorem \ref{PhiA:Projerr} imply that the order ${\mathcal O}(t^{\frac r 2})$
is optimal. Indeed, for small $t>0$ and fixed $c_0\in (0,\sin 1),$ the set  $\mathbb{K}_{c_0}$ is non-empty. For example,  if  $c_0=0.1$, then the length of the interval in \eqref{gboundB}: 
${\rm sinc}^{-1} (c_0)-\sin^{-1} (c_0) \approx 2.75,$ as demonstrated in Figure \ref{fig:sincmaximum}. 
Then for  any $v_0\in {\mathcal V}_{c_0}:={\rm span}\{{\rm e}^{\ri k\cdot x}\,:\, k\in \mathbb{K}_{c_0}\}$ and small $t>0,$ we infer from Theorem \ref{PhiA:Projerr} that
$$2^{ 1 -\frac r 2 } c_0 \,t^{ \frac r2}\, | v_0 |_{ {H}^{r}(\Om) } \le \big\|   \Phi_A^{t}  [  v_0  ] -v_0 \big\|  \le 2^{ 1 -\frac r 2 } \,t^{ \frac r2}\, | v_0 |_{ {H}^{r}(\Om) }, $$
which validates the optimal order ${\mathcal O}(t^{\frac r 2}).$
} \qed
\end{rem}

\begin{figure}
\centering
\includegraphics[width=0.4\textwidth,height=0.35\textwidth]{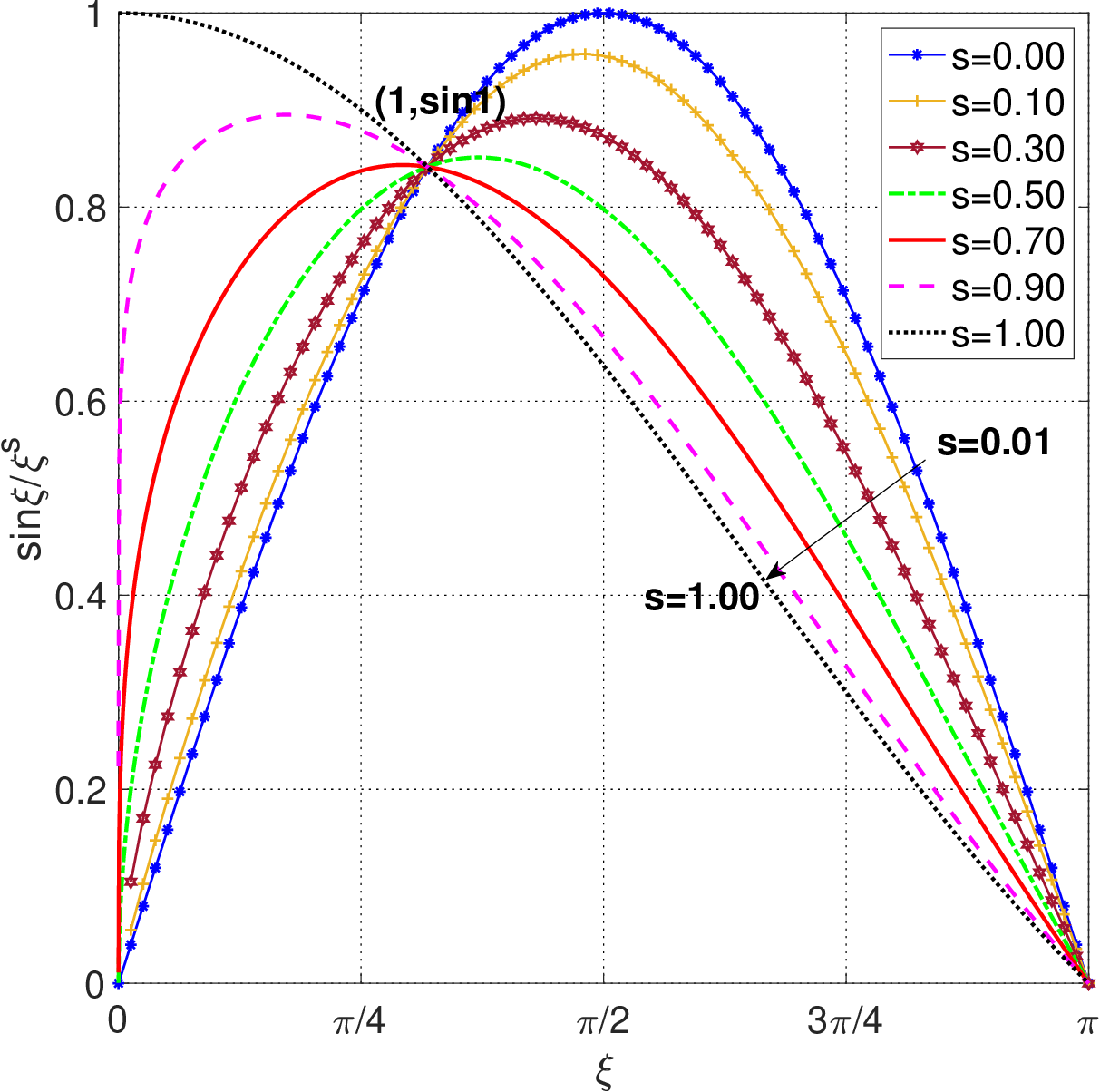} \quad \qquad
\includegraphics[width=0.4\textwidth,height=0.35\textwidth]{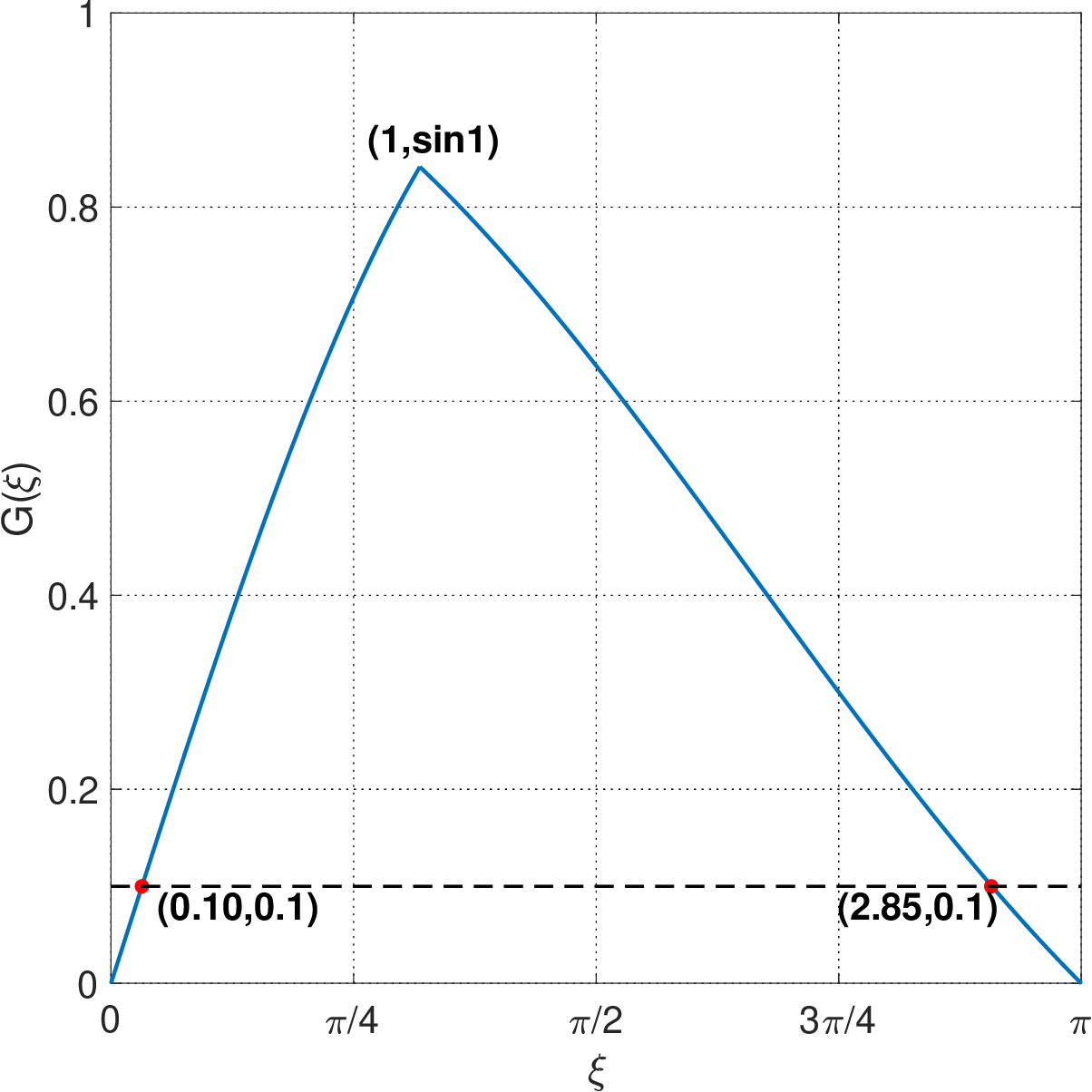}
\caption{Left: Graphs of $g(\xi;s)$ for various $s\in [0,1]$.  Right: Graph of $G(\xi)$ defined in \eqref{gboundA}.}\label{fig:sincmaximum}
\end{figure}

In what follows, we analyze the stability and regularity  of the flow map $\Phi_{B}^{t}:$ 
\begin{equation}\label{PhiB00}
 \Phi_{B}^{t} [ w_0]= \re^{-\ri \lam  t \ln |w_0|^{2}} w_0 = \big\{\!\cos(\lam  t \ln |w_0|^{2})-\ri \sin(\lam  t \ln |w_0|^{2})\big\} w_0.
\end{equation}
It is noteworthy that in the splitting scheme, 
the logarithmic function is composited with trigonometric functions, and interestingly, $\Phi_{B}^{t} [ w_0]$ has a higher regularity than the H\"older continuity of $f(w_0)=w_0\ln|w_0|^2$ in the (non-regularised) implicit or semi-implicit schemes (see  \cite{Paraschis2023FixedpointLogSE, WangYan2022LogSE}). 
\begin{theorem}\label{PhiBbnd} We have the following properties of the flow map $\Phi_B^t.$
\begin{itemize}
\item[{\rm (i)}] For any $x,y\in {\mathbb T^d},$ we have
\begin{equation}\label{diffbnd}
\big|\Phi_{B}^{t}[w_0](x)- \Phi_{B}^{t}[w_0](y)\big|\le (1 + 2|\lambda| t)  |w_0(x) - w_0(y)|,
\end{equation}
which implies if $w_0(x)$ is Lipschitz continuous, so is  $\Phi_{B}^{t}[w_0](x).$ 
\smallskip 
\item[{\rm (ii)}] For any   $w_0\in {H}^{s} (\Om)$ with $0< s <1$, we have
	\begin{equation}\label{eq:logPhiBc}
		| \Phi_{B}^{t} [ w_0] |_{ {H}^{s}(\Om) } \le {\mathcal C}_1^{-1} \mathcal C_2 (1 + 2|\lam| t  )\, | w_0 |_{ {H}^{s}(\Om) },
	\end{equation}
 where $\mathcal C_i, i=1,2$ are the same constants as in \refl{lem:normEquiv}. 
 \smallskip 
 \item[\rm{(iii)}] For any $u, v \in L^2({\mathbb T^d}),$ we have 
	\begin{equation}\label{PhiBL2-Lip}
		\big\| \Phi^t_B  [u]  - \Phi^t_B [v] \big\| \le (1 + 2|\lam| t ) \| u - v \|.
	\end{equation}
 \end{itemize}
\end{theorem}
\begin{proof}   We first prove \eqref{diffbnd}. In fact,
it  suffices  to consider  $|w_0(x)|\ge |w_0(y)|.$ If $|w_0(y)|=0,$  we have $ \Phi_{B}^{t}[w_0](y)=0,$ and
  derive from  \eqref{PhiB} that $|\Phi_{B}^{t}[w_0](x)|=|w_0(x)|,$  so \eqref{diffbnd} is valid.
  Now, assuming $|w_0(x)|\ge |w_0(y)|>0,$ we obtain
\begin{equation}\label{eq:insertest}
\begin{split}
&\big|\Phi_{B}^{t}[w_0](x)- \Phi_{B}^{t}[w_0](y)\big|=\big|\re^{-2\ri \lambda t \ln |w_0(x)| } w_0(x)- \re^{-2\ri \lambda t \ln  |w_0(y)| } w_0(y) \big|
\\ & =  \big| \re^{-2\ri \lambda t \ln |w_0(x)| } \big(w_0(x) - w_0(y) \big) + \big(\re^{-2\ri \lambda t \ln |w_0(x)| } - \re^{-2\ri \lambda t \ln |w_0(y)| } \big) w_0(y)\big|
\\ & \le  \big | w_0(x) - w_0(y) \big| + |w_0(y)|\Big| 1 - \re^{-2\ri \lambda t \ln \frac{ |w_0(x)|}{|w_0(y)| }}  \Big|
\\ & \le  \big | w_0(x) - w_0(y) \big| + 2 |w_0(y)| \Big| \sin  \Big( \lambda t \ln \Big(1 + \frac{|w_0(x)| - |w_0(y)|}{|w_0(y)|} \Big) \Big)\Big|
\\ & \le  (1 + 2|\lambda| t ) | w_0(x) - w_0(y) |,
\end{split}
\end{equation}
where we used the basic inequalities: $\sin x\le x$ and $\ln (1+x) \le x$ for $x\ge 0.$

We now turn to the proof of \eqref{eq:logPhiBc}.  Using the equivalence \refl{lem:normEquiv} and \eqref{eq:insertest}, we obtain
\begin{equation}\label{temproof}
\begin{aligned}
		\big| \Phi_{B}^{t} [w_0] \big|^2_{ {H}^{s} (\Om)} & \le  \frac 1 {{\mathcal C}_1^2} \int_{{\mathbb T^d}} \int_{{\mathbb T^d}}  \frac{ \big| \Phi_{B}^{t} [w_0]( x+ y) - \Phi_{B}^{t} [w_0]( y) \big|^2  }  {  |  x |^{d + 2s} } \rd  x \rd y
		\\&\le (1 + 2|\lambda| t)^2 \frac 1 {{\mathcal C}_1^2}   \int_{{\mathbb T^d}}\int_{\Om} \frac{ | w_0( x+y) -w_0( y)  |^2 }{ |  x |^{d + 2s} } \rd   x \rd y
		\\&\le 
  ({\mathcal C}_1^{-1} \mathcal C_2)^2\,(1 + 2 |\lam|  t)^2\,  | w_0 |_{{H}^{s}(\Om) }^2.
		\end{aligned}
\end{equation}

Finally,  
  replacing $w_0(x),w_0(y)$ in \eqref{eq:insertest}  by 
 $u(x),v(x)$, respectively, we get   
\begin{equation}\label{pstab010}
		\begin{aligned}
			| \Phi_{B}^{t}[u]( x )& - \Phi_{B}^{t} [v] ( x) | \le   \big(1 + 2|\lam| t \big)\, | u( x) - v( x)|,
		\end{aligned}
	\end{equation}
which implies \eqref{PhiBL2-Lip}.
\end{proof}

\begin{rem}\label{L2Hs-Stab}{\em Apparently, from \eqref{PhiB00}, we have the conservation: $\|\Phi_{B}^{t} [ w_0] \|=\|w_0\|.$ 
Thus, the stability \eqref{eq:logPhiBc} holds with the norm $\|\cdot\|_{ {H}^{s}(\Om)}$ in place of  the semi-norm $|\cdot|_{ {H}^{s}(\Om)}$ 
but with the constant factor $1+{\mathcal C}_1^{-1} \mathcal C_2\,(1 + 2 |\lam|  t)$. However, $\Phi_{B}^{t} [ w_0]$ is not differentiable even for smooth $w_0,$ and note that
\begin{equation}\label{eq:loggradPhiB}
	\nabla \Phi_{B}^{t} [w_0] = \re^{-\ri \lam t \ln \!| w_0 |^2  } \Big( \nabla w_0 -2\ri \lam t \frac{w_0 }{| w_0|  } \nabla |w_0| \Big),
\end{equation}
if $|w_0|\not=0.$ 
} \qed
\end{rem}
\begin{rem}\label{non-diff-A} {\em  We find from 
\eqref{diffbnd} and the derivation \eqref{temproof} that the stability result \eqref{eq:logPhiBc} is 
 valid for general non-periodic functions in 
 ${W^{s,2}(\Omega)}$ with the
Gagliardo semi-norm \eqref{defnGag},  that is, 
$$[\Phi_{B}^{t} [w_0]]_{W^{s,2}(\Omega)}\le C(1 + 2\,|\lam|\, t  ) [w_0]_{W^{s,2}(\Omega)},$$
where $C$ is a positive constant independent of $t$ and $w_0.$ 
}\qed
\end{rem}

	


\section{Low regularity analysis of the full-discretization scheme}\label{sect3:Error}
\setcounter{equation}{0}
\setcounter{lmm}{0}
\setcounter{thm}{0}

In this section, we conduct rigorous error analysis for the Lie-Trotter time-splitting scheme \eqref{eq:splititer} with spatial discretisation by the Fourier spectral method for the LogSE with low regularity solution in the fractional Sobolev space.  

 \subsection{Full-discretization scheme}
Notice from \eqref{PhiB}  that if $u_0(x)$ is periodic, then $ \Phi_{B}^{\tau}[u_0](x)$ is periodic as well. We discretize  \eqref{eq:splititer} in space by the Fourier spectral method. Define 
\begin{equation}\label{XnFourier}
    X_N^d={\rm span}\big\{\re^{\ri k\cdot x}\,:\, k\in {\mathbb K}_N^d\big\}, \quad 
    \mathbb{K}^d_N:=\big\{k \in \mathbb{Z}^{d}\,:\,  |k_i|\le N,\;  1\le i \le d\big\}.
\end{equation}
Let $\Pi_{N}^d u$ be the $L^2$-orthogonal projection  of $u$ on $X_N^d,$  that is,
\begin{equation}\label{def:Proj}
\Pi_{N}^d u (x):= \sum_{ k \in \mathbb{K}^d_N} \hat{u}_{k} \,\re^{\ri k\cdot x},\quad \hat{u}_k= \frac{1}{ |\mathbb{T}|^d}\int_{\mathbb{T}^d} u(x) \re^{-\ri  k \cdot x} \rd x.
\end{equation}


The Lie-Trotter time-splitting Fourier spectral (LTSFS) scheme for the LogSE \eqref{eq:LogSE} is to find $\big\{u_N^{m+1}(x)\in X_N^d\big\}_{m=0}^{M-1}$ recursively via
\begin{equation}\label{FourierFullS}
\begin{dcases}
u_{N}^{m+1}(x)= \Phi^{\tau}[u_{N}^{m}](x)=\Phi_A^\tau\, \Pi_{N}^d\,  \Phi_B^\tau [u_{N}^{m}](x),  & m=0,1,\ldots, M-1,\\
 u_{N}^0(x) = \Pi_{N}^d  u_0(x), & x\in {\mathbb T^d}. 
\end{dcases}
\end{equation}
 More precisely, we can obtain $u_{N}^{m+1}(x)$ from $u_{N}^{m}(x)$ by
\begin{equation}\label{NumSolComp}
u_{N}^{m+1}(x) = \sum_{k \in \mathbb{K}^d_N} \re^{-\ri |k|^2\tau}\, (\widehat{\Phi_{B}^{\tau} [u_N^m] })_k \, \re^{\ri k\cdot x},
\end{equation}
where 
\begin{equation}\label{PhiBcoef}
( \widehat{\Phi_{B}^{\tau} [u_N^m]} )_k  = \frac{1}{ |\mathbb{T}|^{d} }\int_{{\mathbb T^d}} \re^{-\ri \lam \tau \ln|u_N^m(x)|^{2}} u_N^m(x) \re^{-\ri k\cdot x}\rd x.
\end{equation}
Here the expansion coefficients in  \eqref{PhiBcoef} will be computed very accurately via numerical quadrature with negligible quadrature errors.  In fact, we find the use of over-quadrature (i.e., with the number of nodes more than $N^d$) appears necessary
for such a  nonlinear term with low regularity.  Moreover, this is a relatively simpler setting for the clarity of analysis focusing on the logarithmic  nonlinearity.

\subsection{Useful lemmas}  We make necessary preparations for the error analysis. 
We first derive some estimates for the nonlinear functions 
\begin{equation}\label{fzfepszB}
 f(z)=z\ln |z|^2,\quad  f^{\varepsilon}(z)=z\ln(|z| + \varepsilon)^2,\quad z\in \mathbb C, \;\; \varepsilon >0,  
\end{equation}
which are useful to deal with the logarithmic nonlinear term.  Such results will  play an important  role similar to  the Cazenave-Haraux (CH) property of $f$ (cf. \cite[Lemma 1.1.1] {Cazenave1980}):
\begin{equation}\label{f:CH}
		\big|  \Im \big\{( f (z_1) - f(z_2) ) (\bar{z}_1 - \bar{z}_2 )\big\} \big|  \le  2\,|z_1 - z_2|^2,\quad \forall\, z_1, z_2 \in \bC, 
\end{equation}
for the error analysis in e.g.,  \cite{Bao2019Regularized,Paraschis2023FixedpointLogSE}.

\begin{lemma}\label{Lem:fRfdiff} For any $z, z_1,z_2\in {\mathbb C}$ and $\varepsilon >0,$  we have  
\begin{subequations}\label{fzfepsz}
\begin{gather}
| f^{\varepsilon} (z) - f(z) | \le 2\, \varepsilon, 
\label{eq:ffepsdiff}\\[6pt]
 |f^{\varepsilon}(z_2) - f^{\varepsilon}(z_1)| \le 2\{ |\! \ln( \zeta + \varepsilon) | + 1 \}  |z_1 - z_2|, \label{eq:fepsstab}\\[6pt]
 | f(z_1) - f(z_2) | \le 4\,\varepsilon +  2\{ |\! \ln( \zeta + \varepsilon) | + 1 \} |z_1 - z_2|,\label{eq:fuvdiff}
\end{gather}
\end{subequations}
where $\zeta:=\max\{ |z_1|, |z_2|\}.$
\end{lemma}
\begin{proof} In fact, we can find \eqref{eq:ffepsdiff} from \cite[Lemma 2.1]{Paraschis2023FixedpointLogSE}. For completeness, we provide a slightly different proof. 
 It is trivial for  $|z|=0$. For $|z|>0$, using the basic inequality: $\ln(1+ x) \le x,\, x\ge0$,  we derive
\begin{equation*}
| f^{\varepsilon}(z) - f(z) | = 2 |z| \, \big| \ln ( |z| + \varepsilon) - \ln |z| \big| = 2 |z| \ln \Big(1 + \frac{\varepsilon}{|z|} \Big) \le 2\,{\varepsilon}.
\end{equation*}

We  follow  Bao et al \cite[(2.12)]{Bao2019Regularized} to derive \eqref{eq:fepsstab}, but formulate it as a more general and tighter bound. 
In view of the symmetry, we only need to consider $|z_2|\ge |z_1|$. Direct calculation leads to 
\begin{equation}\label{fepsdiff}
\begin{split}
| f^{\varepsilon}( z_2 ) - f^{\varepsilon}( z_1 ) | 
&= 2\big|\ln(|z_2| + \varepsilon) (z_2-z_1)+ z_1 \big\{\!\ln( | z_2 | + \varepsilon) - \ln (| z_1 | + \varepsilon)\big\} \big|\\
&\le 2 |\!\ln(| z_2 |+ \varepsilon)|\, | z_2 - z_1  | + 2| z_1 | \, \big|\!\ln( | z_2 | + \varepsilon) - \ln (| z_1 | + \varepsilon) \big| \\
& = 2 |\!\ln(| z_2 |+ \varepsilon)|\, | z_2 - z_1  | +2 | z_1 |\,\Big|\!\ln\Big(1 + \frac{| z_2 | - | z_1 |}{|  z_1 | + \varepsilon} \Big) \Big| \\
& \le 2 |\!\ln(| z_2 |+ \varepsilon)|\, | z_2  - z_1  |  + 2| z_1 | \, \frac{| z_2 | - | z_1 |}{| z_1 | + \varepsilon}  \\
& \le 2 \{  |\!\ln(| z_2 |+ \varepsilon)| + 1 \} | z_2 - z_1 | =2\{ |\! \ln( \zeta + \varepsilon) | + 1 \}  |z_1 - z_2|,
\end{split}
\end{equation}
where we used the above basic inequality again, and the fact $| z_2 | - | z_1 | \le | z_2 - z_1 |$.


Finally, from the triangular inequality and \eqref{eq:ffepsdiff}--\eqref{eq:fepsstab}, we obtain 
\begin{equation*}
\begin{aligned}
|f( z_2 ) - f( z_1 )| &\le |f(z_2) - f^{\varepsilon}(z_2) | + | f^{\varepsilon}(z_2) - f^{\varepsilon} (z_1) | + |f^{\varepsilon}(z_1) - f(z_1) | \\
&\le 4\,\varepsilon +  | f^{\varepsilon}(z_2) - f^{\varepsilon} (z_1) | \\
&\le 4\, \varepsilon + 2 \{  |\!\ln( \zeta + \varepsilon)| + 1 \} \, | z_2 - z_1 |.
\end{aligned}
\end{equation*}
This ends the proof.
\end{proof}


 With the aid of Lemma \ref{Lem:fRfdiff}, we can derive the following important results valid 
 on general bounded domain $\Omega\subset {\mathbb R}^d$ including  $\Omega={\mathbb T}^d$. 
\begin{lemma}\label{fufvL2-est}
For any  $u,v\in L^\infty(\Omega)$ and $0<\varepsilon<1,$ we have
\begin{subequations}\label{L2-fzfepsz}
\begin{gather}
 \|f^{\varepsilon}(u) - f^{\varepsilon}(v)\| \le 2\Upsilon(\varepsilon) \|u - v\|, \label{eq:L2fepsstab}\\[4pt]
 \|f(u) - f(v)\| \le 4|\Omega|^{\frac 12}\varepsilon +2 \Upsilon(\varepsilon) \|u - v\|,
 \label{eq:L2fuvdiff}
\end{gather}
\end{subequations}
where 
\begin{equation}\label{UpsilonDef}
    \Upsilon(\varepsilon):= \Upsilon(\varepsilon; \|u\|_\infty, \|v\|_\infty):=
    \max\big\{|\!\ln \varepsilon|,\, \ln(\|u\|_\infty+1), \,\ln(\|v\|_\infty+1)\big\}+1. 
\end{equation}
\end{lemma}
\begin{proof}  We first prove \eqref{eq:L2fepsstab}. By \eqref{eq:fepsstab}, 
 \begin{equation}\label{L2-dev1}
\|f^{\varepsilon}(u) - f^{\varepsilon}(v)\| \le  2 \big\{ \sup_{x\in \Omega}  
 |\! \ln( \zeta(x) + \varepsilon) | + 1 \big\}  \|u - v\|,
 \end{equation}
where $\zeta(x)=\max\{ |u(x)|, |v(x)|\}.$  We can show that for any $w\in L^\infty (\Omega)$ and $0<\varepsilon<1,$
\begin{equation}\label{logbnd}
 |\! \ln( |w(x)| + \varepsilon)|\le\max\!\big\{|\!\ln \varepsilon|, |\!\ln(\|w\|_\infty+\varepsilon) |\big\} \le  \max\!\big\{|\!\ln \varepsilon|, |\!\ln(\|w\|_\infty+1) |\big\}.
\end{equation}
Indeed, it is clear that if $|w(x)| + \varepsilon<1,$ then the modulus of the logarithmic function is decreasing, so we have $|\! \ln( |w(x)| + \varepsilon)|\le |\!\ln \varepsilon|.$ On the other hand, if $|w(x)| + \varepsilon\ge 1,$ it is increasing, so $|\! \ln(|w(x)| + \varepsilon)|\le |\!\ln(\|w\|_\infty+\varepsilon) |.$ Thus, the first inequality in 
\eqref{logbnd} holds. Using the same argument, we deduce that 
 \begin{equation*}\label{logbnd00}
  |\!\ln(\|w\|_\infty+\varepsilon) |\le 
 \max\big\{|\!\ln \varepsilon|, \ln(\|w\|_\infty+1) \big\}.
 \end{equation*}
Therefore,  \eqref{logbnd} is verified.   Then we can derive \eqref{eq:L2fepsstab} from \eqref{L2-dev1}-\eqref{logbnd} straightforwardly. 

Using \eqref{eq:fuvdiff} and following the same derivation for \eqref{eq:L2fepsstab} as above, we can obtain the estimate \eqref{eq:L2fuvdiff}.
\end{proof}


In the error analysis, we also need to use the following stability result in the fractional periodic Sobolev space. 
\begin{lemma}\label{Lem:RfHs}
 If $u \in L^\infty(\mathbb{T}^d )\cap H^s(\mathbb{T}^d)$ with $0\le s \le 1$, then for $0<\varepsilon<1,$
\begin{equation}\label{eq:RfHs}
| f^{\varepsilon} (u)|_{ {H}^{s}(\mathbb{T}^d ) } \le C_s
\widehat \Upsilon(\varepsilon) \, |u|_{H^s(\mathbb{T}^d)},
\end{equation}
where 
\begin{equation}\label{widehatUpsilonDef}
    \widehat \Upsilon(\varepsilon):= \widehat \Upsilon(\varepsilon; \|u\|_\infty):=
    \max\big\{|\!\ln \varepsilon|,\, \ln(\|u\|_\infty+1)\big\}+1,
\end{equation}
and  $C_s=2$ for $s=0,1$ and $C_s=
2 {\mathcal C}_1^{-1} {\mathcal C_2}$ for $s\in (0,1),$ respectively. 
\end{lemma}
\begin{proof}
For $ s=0$,  we derive from  \eqref{logbnd} directly that   
\begin{equation*}
\begin{aligned}
\| f^{\varepsilon} (u) \| &\le 2 \| \! \ln (|u| + \varepsilon) \|_\infty \, \| u \|  \le 2\max\{ |\! \ln \varepsilon |, \; \ln (\| u \|_\infty + 1) \} \| u \|. 
\end{aligned}
\end{equation*}

For $0< s < 1$, we find  from \eqref{eq:fepsstab} in \refl{Lem:fRfdiff} that 
        \begin{equation*}
         |f^{\varepsilon}(u(x+y) ) - f^{\varepsilon} (u(y)) | \le 2\{ |\! \ln(\zeta + \varepsilon ) | + 1\} | u(x+y) - u(y) |,
        \end{equation*}
where $\zeta(x,y)=\max\{|u(x)|, |u(x+y)|\}$. Then using \refl{lem:normEquiv} and \eqref{logbnd}, we derive that 
\begin{equation}\label{feps-stabA}
\begin{split}
| f^{\varepsilon} (u) |_{ {H}^{s}(\mathbb{T}^d ) } &\le \frac 1{\mathcal C_1} \Big( \int_{\mathbb{T}^d } \int_{\mathbb{T}^d } \frac{ | f^{\varepsilon} (u(x+y) )- f^{\varepsilon}(u(y) ) |^2  }{|x|^{d+2s}} \rd x \rd y\Big)^{\frac 1 2}\\
& \le 2\, {\mathcal C}_1^{-1} {\mathcal C_2} \Big\{ \sup_{x,y\in {\mathbb T^d}}| \!\ln( \zeta(x,y) +\varepsilon)|  + 1 \Big\} | u |_{ {H}^{s}( \mathbb{T}^d ) } \\
& \le 2\, {\mathcal C}_1^{-1} {\mathcal C_2} \big\{\max\{ |\! \ln \varepsilon|  , \ln(\| u \|_\infty +1)  \}+1\big\} | u |_{ {H}^{s}( \mathbb{T}^d ) }.
\end{split}
\end{equation}

Finally, for $s=1$, we obtain from the direct calculation that 
\begin{equation*}
\nabla f^{\varepsilon} (u) = 2 \ln ( |u| + \varepsilon ) \nabla u +  \frac{2 u} {| u | + \epsilon} \nabla | u |,\quad \nabla |u| = \frac{u \nabla \bar u + \bar u \nabla u}{2|u|} = \frac{ \Re\{u\nabla \bar u\} } {|u|}.
\end{equation*}
Since $| \nabla |u| | \le | \nabla u |$,  we get from \eqref{logbnd} that
\begin{equation*}
\begin{aligned}
\|  \nabla f^{\varepsilon} (u) \| &\le 2 \Big\{ \sup_{x\in \mathbb T^d}| \! \ln( | u(x) | + \varepsilon ) | + 1  \Big\}  \| \nabla u \|  \\
& \le  2 \big\{\max\{ |\! \ln \varepsilon|, \ln(\| u \|_\infty +1)\}+1\big\}  \| \nabla u \| .
\end{aligned}
\end{equation*}
This  completes the proof.
\end{proof}


Finally, we present an inverse inequality and a relevant  approximation result on Fourier expansions.  It is noteworthy that they can be obtained from the the corresponding results in one dimension. Here we sketch the derivations in the Appendix for the readers' reference, and the focus is placed on deriving sharper constants. Indeed,  the $d$-dimensional Fourier approximation can be founded in e.g., \cite[Theorem 3.1]{Pareschi2022FourierSpectralBoltzman} but with an implicit constant. 

\begin{lemma}\label{lem:PiNapprox}
 For any  $\phi \in X_N^d$ with $d\ge 1$, we have 
\begin{equation}\label{XNinftyto2}
\| \phi \|_\infty \le \Big(\frac{2N+1}{|\mathbb{T}|} \Big)^{\frac d 2} \| \phi \|.
\end{equation}
\end{lemma}

\begin{lemma}\label{lem:errProj}
For any  $u \in H^s({\mathbb T^d})$ with $ 0 \le \mu \le s,$ we have
\begin{subequations}
\begin{gather}
    | u   - \Pi_N^d u |_{H^{\mu}({\mathbb T^d})}  \le   N^{\mu-s} | u |_{ H^{s}({\mathbb T^d}) }, \label{err:P_N} \\[4pt]
    | \Pi_N^d u |_{H^{s}({\mathbb T^d}) } \le | u |_{ H^{s}({\mathbb T^d}) },\label{XNprojless}
\end{gather}   
\end{subequations}
where for $\mu=0$ or $s=0,$ we understand  $|\cdot|_{H^{0}({\mathbb T^d})}=\|\cdot\|.$ 
\end{lemma}

\subsection{Main result on error estimate}\label{FracMainResult}
With the above preparations, we are now ready to carry out the convergence analysis of the  Lie-Trotter splitting Fourier spectral method \eqref{FourierFullS}. For notational convenience, we denote
\begin{equation}\label{globalnotation}
\begin{split}
&u^k(x) := u(x,t_k), \quad u^k_*(x):=\Pi_N^d u^k(x), \quad 
e_*^k(x):= u^k_*(x) - \Phi^\tau [u_*^{k-1}](x),
\end{split}
\end{equation}
and 
\begin{equation}\label{globalnotation2}
|u|_{C(H^s)}:=\sup_{t\in (0,T]} |u(\cdot, t)|_{H^{s}({\mathbb T^d})},  \quad \| u \|_{C(L^\infty) }:= \sup_{t\in (0,T]} \| u(\cdot, t) \|_{ L^\infty(\mathbb{T}^d ) }.
\end{equation}
The main result is stated as follows. 
\begin{theorem}\label{THM:errestfrac}
Assume that $u_0\in L^2(\mathbb T^d)$ and   the solution $u$ of \eqref{eq:LogSE} has the regularity 
\begin{equation}\label{errest:cons}
 u \in C((0,T]; H^{s}({\mathbb T^d})\cap L^\infty( \mathbb{T}^d ) ), \quad 0 < s < 1,
 \end{equation}
 and let  $u_N^{m+1}$ be the numerical solution obtained by  \eqref{FourierFullS}. 
  Then we have the $L^2$-error estimate 
\begin{equation}\label{mainresultA}
\| u^{m+1} - u_N^{m+1} \| \le C {\re}^{2|\lambda| T}  \big({\mathcal C}(u_0,u)+\ln N\big)\,(\tau^{\frac s 2} + N^{-s})\, | u |_{ C( H^{s} ) },
\end{equation}
for $m=0,1,\ldots, M-1,$
where $C$ is a positive constant independent of $\tau,N$ and any function, and 
$${\mathcal C}(u_0,u):=\ln (\|u_0\|+1)+\ln(\|u\|_{C(L^\infty) }+1). $$
\end{theorem}
\begin{proof} 
Using the triangle inequality and \refl{lem:errProj}, we deduce that
	\begin{equation}\label{errest:ineq1}
	\begin{split}
	\| u^{m+1} - u_{N}^{m+1}\| & \le \| u^{m+1} - u^{m+1}_* \| + \| u^{m+1}_* - u_{N}^{m+1} \| \\
 & \le N^{-s} | u^{m+1}|_{H^{s}({\mathbb T^d})} + \| u^{m+1}_* - u_{N}^{m+1} \| \\
	& \le N^{-s} |u|_{C(H^s)} + \| u^{m+1}_* - u_{N}^{m+1} \|.
       \end{split}
	\end{equation}
 Thus, it suffices to estimate $\| u^{m+1}_* - u_{N}^{m+1} \|.$ From 
 \eqref{FourierFullS}, we know that $u_{N}^{m+1}=\Phi^{\tau}[u_N^{m}],$ so 
\begin{equation*}\label{errest:ineq20}
		\begin{aligned}
		          \| u^{m+1}_* - u_{N}^{m+1} \|  & \le   \| u^{m+1}_*  - \Phi^{\tau}[u^{m}_*] \| + \|  \Phi^{\tau}[u^{m}_*] - \Phi^{\tau}[u_N^{m}] \| \\
		          & =\| e_*^{m+1} \| +\big \| \Phi_A^\tau\, \Pi_{N}^d\,(  \Phi_B^{\tau}[u^{m}_*] - \Phi^{\tau}_B[u_N^{m}])\big\|.
		\end{aligned}
	\end{equation*}
As the flow map of the linear Schr\"odinger's operator satisfies $\| \Phi_A^\tau[v_0]\|=\|v_0\|$ and  $\|\Pi_{N}^d  w \|\le \|w\|,$ we derive from Theorem \ref{PhiBbnd}-(iii) the recurrence relation
\begin{equation*}\label{err21}
		        \| u^{m+1}_* - u_{N}^{m+1} \| 
		           \le \| e_*^{m+1} \| + \alpha 
            \| u^{m}_* - u_{N}^{m} \|,\quad \alpha:=1+ 2|\lambda|\tau,
	\end{equation*}
which  implies 
 	\begin{equation}\label{errest:ineq2}
		\begin{aligned}
		          \| u^{m+1}_* - u_{N}^{m+1} \|  
		          & \le \| e_*^{m+1} \| + \alpha\big(\|  e_*^{m} \|+\alpha \| u^{m-1}_* - u_{N}^{m-1} \| \big) \\
		          & \;\;   \vdots  \\
		          & \le \| e_*^{m+1} \| + \alpha \| e_*^m \| +\cdots + \alpha^m \| e_*^1\| +  \alpha^{m+1} \| u^{0}_* - u_{N}^{0} \| \\
		          & = \sum_{k=0}^{m} \alpha^{m-k}\| e_*^{k+1}\|, 
		\end{aligned}
	\end{equation}
 where we have noted $ u^{0}_* =u_{N}^{0}=\Pi_N^d u_0.$
Thus, by  \eqref{errest:ineq1}-\eqref{errest:ineq2}, 
\begin{equation}\label{errest:maininequ}
\| u^{m+1} - u_{N}^{m+1}\| \le  N^{-s} |u|_{ C(H^s)} + \sum_{k=0}^{m} \alpha^{m-k}\| e_*^{k+1}\|.
\end{equation}
\smallskip
\emph{The rest of the proof is to show that for  $k=0,\ldots, M-1,$}
\begin{equation}\label{mainestimate}
\| e_*^{k+1}\| \le   C \tau \, {\rm e}^{2|\lambda|\tau} \big\{  {\mathcal C }(u_0,u)+\ln N\big\}\, (\tau^{\frac s 2} + N^{-s} ) \,   |u|_{ C(H^s) }.
\end{equation}
 Indeed, if \eqref{mainestimate} 
were proved, we obtain from \eqref{errest:maininequ} straightforwardly that 
\begin{equation*}\label{sumlocerr}
\begin{split}
\| u^{m+1} - u_{N}^{m+1}\|  & \le N^{-s} |u|_{ C(H^s) }  + C \tau \, 
\Big(\sum_{k=0}^m \alpha^{m-k}\Big) \\
& \quad \times {\rm e}^{2|\lambda|\tau}  \big\{  {\mathcal C }(u_0,u)+\ln N\big\}\, (\tau^{\frac s 2} + N^{-s} ) \,   |u|_{ C(H^s) }. 
\end{split}
\end{equation*}
Since 
$$
\sum_{k=0}^m \alpha^{m-k}=\sum_{k=0}^m \alpha^{k} = 
\frac{1- \alpha^{m+1} }{ 1-\alpha}=
\frac{ (1 + 2|\lambda| \tau )^{m+1}-1 }{ 2|\lambda | \tau} \le 
\frac{{\rm e}^{2(m+1)|\lambda|\tau}-1 }{ 2|\lambda |\tau},
$$
we can derive  the desired estimate  \eqref{mainresultA}, given \eqref{mainestimate}. 

\bigskip 

\noindent\underline{\bf Proof of \eqref{mainestimate}.}\,   Let $\Psi^{t}[u^k]$ be the  flow map of \eqref{eq:LogSE} with the initial value $u^k$, that is, 
\begin{equation}\label{loc:00}
\begin{dcases}
 \ri \, \partial_{t} \Psi^{t} [ u^k]  + \Delta  \Psi^{t} [ u^k ]  = \lambda\, f(\Psi^{t} [ u^k ]), \quad t\in (0,\tau], \\
 \Psi^0[u^k] = u^k.
 \end{dcases}
\end{equation}
Note that at $t=\tau,$ $\Psi^\tau [u^k] = u^{k+1}=u(\cdot, t_{k+1})$. 
 For simplicity of notation, we further introduce
\begin{equation}\label{PsiStart}
   \Psi^{t}_*[\cdot]:=\Pi_N^d \Psi^{t} [\cdot], \quad  f_*(\cdot):=\Pi_N^d f(\cdot),\quad \mathcal{E}_k^t := \Psi^t_*[ u^k]- \Phi^t[ u^k_*].
\end{equation}
As   $\Pi_N^d \Delta=\Delta \Pi_N^d,$ a direct projection of \eqref{loc:00}  leads to  
\begin{equation}\label{loc:01} 
\begin{dcases}
\ri \, \partial_{t}  \Psi^{t}_* [ u^k]  + \Delta  \Psi^{t}_* [ u^k ]  = \lambda\, f_*(\Psi^{t} [ u^k ]), \quad t\in (0,\tau], \\
 \Psi^0_*[u^k] = u^k_*.
\end{dcases}
\end{equation}
 
We next derive the  error equation of $\mathcal{E}_k^t.$  Using the definitions of 
$\Phi^t$ in 
the scheme \eqref{FourierFullS} and 
 $\Phi_A^t,\,\Phi_B^t$  in \eqref{PhiA}-\eqref{PhiB}, direct calculation leads to  
\begin{equation*}
\begin{split}
\partial_t \Phi^t[u^k_*] &= \partial_t \big\{\Phi_A^t \Pi_N^d \Phi_B^t [u^k_*] \big\}= \partial_t \big\{ {\rm e}^{\ri t \Delta } \Pi_N^d\, {\rm e}^{-\ri \lambda t \ln | u^k_* |^2 } u^k_*\big\}\\
&= \ri \Delta \Phi^t[u^k_*] - \ri \lambda\, {\rm e}^{\ri t \Delta } \Pi_N^d  \big\{  {\rm e}^{-\ri \lambda t \ln | u^k_* |^2 }  u^k_*\, \ln | u^k_*|^2  \big\}\\
&= \ri \Delta \Phi^t[u^k_*] - \ri \lambda\, {\rm e}^{\ri t \Delta } \Pi_N^d  
\big\{  \Phi_B^t[u^k_*]\,\ln | u^k_*|^2   \big\}\\
&= \ri \Delta \Phi^t[u^k_*] - \ri \lambda\, {\rm e}^{\ri t \Delta } \Pi_N^d  
\big\{  \Phi_B^t[u^k_*] \,\ln | \Phi_B^t[u^k_*] |^2  \big\}\\
& = \ri \Delta \Phi^t[u^k_*] - \ri \lambda\, \Phi_A^t \Pi_N^d f(\Phi_B^t[u^k_*]),
\end{split}
\end{equation*}
where have used  the simple fact $|u^k_*|=|\Phi_B^t[u^k_*]|.$ It immediately implies
\begin{equation}\label{loc:02}
\ri \, \partial_{t}  \Phi^{t} [ u^k_* ]  + \Delta \Phi^{t} [ u^k_*]  = \lambda\,\Phi_{A}^{t}  \Pi_N^d f(\Phi_{B}^{t} [ u^k_*] ),\quad t \in (0, \tau],
\end{equation}
with $\Phi^0[u^k_*] = u^k_*$.  Accordingly, we 
derive from \eqref{loc:01}-\eqref{loc:02} the error equation:
\begin{equation}\label{loc:errequ}
\begin{dcases}
\ri\, \partial_t \mathcal{E}_k^{t} + \Delta  \mathcal{E}_k^t = \lambda \big\{  f_*(\Psi^{t} [u^k ]) - \Phi_{A}^{t}    f_*(\Phi_{B}^{t} [ u^k_*]) \big\}, \quad t\in (0, \tau],\\
\mathcal{E}_k^0=0.
\end{dcases}
\end{equation}
In view of $\Psi^\tau [u^k] = u^{k+1},$  we notice that 
\begin{equation}\label{errest:et0}
 \mathcal{E}_k^\tau=\Psi_*^\tau[u^k] -\Phi^\tau[u^k_*] = u_*^{k+1}- \Phi^\tau[u^k_*] = e_*^{k+1},
\end{equation}  
  so we next prove $ \mathcal{E}_k^\tau $ satisfying the estimate \eqref{mainestimate}.  For this purpose, 
taking the inner product with $\mathcal{E}_k^t$ on the first equation of \eqref{loc:errequ}, the imaginary part of the resulting equation reads
\begin{equation}\label{et1AA}
\begin{split}
			\frac{1}{2} \frac{\rd} {\rd t} \| \mathcal{E}_k^t \|^2 & = \,  \lambda\, \Im \big(f_*(\Psi^{t} [ u^k ] )  -   \Phi_{A}^{t}   f_*(\Phi_{B}^{t} [ u^k_*]), {\mathcal{E}_k^t } \big) \\
			&=\, \lambda \Im\,\Big\{\big( f_*(\Psi^{t} [u^k])  - f(\Psi^{t}_* [u^k]) , \mathcal{E}_k^t ) +  
			 \big(  f (\Psi^{t}_* [ u^k] )  - f ( \Phi^t [ u^k_*] ), \mathcal{E}_k^t  \big)  \\
			  &\quad +\big(f ( \Phi^t [ u^k_* ] ) - \Phi_{A}^{t}    f_*( \Phi_{B}^{t} [u^k_*] ), \mathcal{E}_k^t  \big)\Big\}.		
\end{split}
\end{equation}
As $\mathcal{E}_k^t := \Psi^t_*[ u^k]- \Phi^t[ u^k_*],$ we use \eqref{f:CH} to deal with the second term and get  
\begin{equation*}\label{second-termA}
         \big|\Im \big\{( f (\Psi^{t}_* [ u^k] )  - f ( \Phi^t [ u^k_*] ))\cdot \bar{\mathcal{E}}_k^t\big\} \big| \le 2 |\mathcal{E}_k^t|^2,
\end{equation*}
which immediately implies 
\begin{equation}\label{second-termB}
\big|\Im\, \big(  f (\Psi^{t}_* [ u^k] )  - f ( \Phi^t [ u^k_*] ), \mathcal{E}_k^t  \big)\big|\le 2 \|\mathcal{E}_k^t\|^2.
\end{equation}
Using  the Cauchy-Schwarz inequality, we obtain from 
\eqref{et1AA}-\eqref{second-termB} that 
\begin{equation}\label{loc:03}
\begin{split}
\frac{\rd}{\rd t} \| \mathcal{E}_k^{t} \|  & \le  2 |\lambda| \,  \|\mathcal{E}_k^{t}  \| +  |\lambda|\,  \big\{\| f(\Phi^{t} [ u^k_* ]) - \Phi_{A}^{t}   f_*(\Phi_{B}^{t} [ u^k_* ]) \|   +    \|  f_*(\Psi^{t} [ u^k ])  - f( \Psi^{t}_* [ u^k ]) \|\big\}\\
&= 2 |\lambda| \,  \|\mathcal{E}_k^{t}  \|+|\lambda |\,\big\{\mathcal T_1^t+\mathcal T_2^t\big\},
\end{split}
\end{equation}
so it is necessary to estimate the two ``error terms'': $\mathcal T_1^t,\mathcal T_2^t$ for $t\in (0,\tau)$, before we apply the Gr\"onwall's inequality to  \eqref{loc:03}.

	\medskip  
	
\noindent{\bf (i)\, Estimate $\mathcal T_1^t$.}\;  Using  the triangle inequality, we further split $\mathcal T_1^t$ into the following four terms: 
\begin{equation}\label{e1toe4}
 \mathcal T_1^t=\| f(\Phi^{t} [ u^k_*]) - \Phi_{A}^{t}    f_*(\Phi_{B}^{t} [ u^k_* ]) \|  \le \mathcal T_{1,1}^t +\mathcal T_{1,2}^t+\mathcal T_{1,3}^t+\mathcal T_{1,4}^t,
\end{equation}
where
\begin{equation}\label{T1234}
\begin{split}
& \mathcal{T}^t_{1,1} := \| f(\Phi^t[ u^k_*] ) - f(\Pi_N^d \Phi_B^t [ u^k_*] ) \|,\quad \mathcal{T}^t_{1,2}:= \| f(\Pi_N^d \Phi_B^t [ u^k_* ]) - f(\Phi_B^t [ u^k_* ]) \|, \\[4pt]
& \mathcal{T}^t_{1,3}:= \| f(\Phi_B^t [ u^k_* ]) - f_*(\Phi_B^t [ u^k_* ]) \|, \quad\;\; \mathcal{T}^t_{1,4}:= \|  f_*(\Phi_B^t [ u^k_* ])  - \Phi_{A}^{t}   f_*(\Phi_{B}^{t} [ u^k_* ]) \|.
\end{split}
\end{equation}
We first  estimate $\mathcal{T}^t_{1,1}$.  Let ${\mathbb I}$ be the identity operator.  Using the  theorems and lemmas stated previously, we can obtain 
\begin{equation}\label{Pest:e1}
 \begin{aligned}
    \mathcal{T}^t_{1,1} & \le 4|{\mathbb T}|^{\frac d 2} \varepsilon + 2  \Upsilon_{\!1}^t(\varepsilon)  \, \| (  \Phi_A^t - {\mathbb I}) \Pi_N^d \Phi_B^t  [ u^k_* ] \| && {\rm (\refl{fufvL2-est})} \\
    & \le 4|{\mathbb T}|^{\frac d 2} \varepsilon+
    2^{2-\frac s 2}  t^{\frac s 2} \Upsilon_{\!1}^t(\varepsilon)  \, |\Pi_N^d \Phi_B^t[u^k_*]\big|_{H^s(\mathbb{T}^d)} && {\rm (\reft{PhiA:Projerr})}\\
    & \le 4|{\mathbb T}|^{\frac d 2} \varepsilon+
    2^{2-\frac s 2}  t^{\frac s 2} \Upsilon_{\!1}^t(\varepsilon)  \, |\Phi_B^t[u^k_*]|_{H^s(\mathbb{T}^d)}  && {\rm (\refl{lem:errProj})} \\
   & \le 4|{\mathbb T}|^{\frac d 2} \varepsilon+
    2^{2-\frac s 2}  t^{\frac s 2} \Upsilon_{\!1}^t(\varepsilon)  \,
    \mathcal{C}_1^{-1}\mathcal{C}_2\, \tilde \alpha (t)\, |u^k_*|_{H^s(\mathbb{T}^d)} && {\rm (\reft{PhiBbnd})} \\
    & \le 4|{\mathbb T}|^{\frac d 2} \varepsilon+
    2^{2-\frac s 2} \mathcal{C}_1^{-1}\mathcal{C}_2 \Upsilon_{\!1}^t(\varepsilon) \, \tilde \alpha(t)  \, t^{\frac s 2}  \,|u^k|_{H^s(\mathbb{T}^d)},  && {\rm (\refl{lem:errProj})}
   \end{aligned}
\end{equation}
where $\tilde \alpha(t):=1+2|\lambda|t$ and
\begin{equation*}
\Upsilon_{\!1}^t(\varepsilon) := 
\max\big\{|\!\ln \varepsilon|,\, \ln(\| \Phi^t[ u^k_* ] \|_{\infty}+1),\, \ln(\| \Pi_N^d \Phi_B^t [ u^k_* ]\|_{\infty}+1) \big\}+1.
\end{equation*}
We now estimate $\Upsilon_{\!1}^t(\varepsilon).$
We observe from \eqref{FourierFullS}-\eqref{NumSolComp} that $\Phi^t[u^k_*]\in X_N^d$. Using the inverse inequality \eqref{XNinftyto2} and the facts:
$\|\Phi_A^t[w]\|=\|w\|$ and $\|\Phi_B^t[w]\|=\|w\|,$ we obtain from  \refl{lem:errProj} and 
the conservation of mass: $\| u^k \| = \| u_0 \|$
that
\begin{equation}\label{PhiBLinfty1}
\begin{aligned}
\| \Phi^t[u^k_*] \|_\infty &\le C_d N^{\frac d 2}  \| \Phi^t_A\Pi_N^d \Phi_B^t[u_*^k]\| = C_d N^{\frac d 2}\| \Pi_N^d \Phi_B^t[u_*^k]\| \le C_d N^{\frac d 2} \| \Phi_B^t[u^k_*] \| \\
& = C_d N^{\frac d 2} \| u_*^k \| \le C_d N^{\frac d 2} \| u^k\|=C_d N^{\frac d 2} \| u_0\|,
\end{aligned}
\end{equation}
and similarly, 
\begin{equation}\label{PhiBLinfty2}
\begin{aligned}
\| \Pi_N^d \Phi_B^t[u^k_*] \|_\infty 
\le C_d N^{\frac d 2}  \| u_0 \|,\quad C_d:=\Big(\frac{2+N^{-1}}{|\mathbb T|}\Big)^{\frac  d 2}.
\end{aligned}
\end{equation}
Then we derive from the above that
\begin{equation}\label{Pupsilon1}
\begin{split}
\Upsilon_{\!1}^t(\varepsilon)  & \le |\! \ln \varepsilon | +   \ln (C_d N^{\frac d 2} \| u_0\| + 1 ) +   1 \\
& \le |\! \ln \varepsilon | +  \ln \big\{N^{\frac d 2}(C_d\| u_0 \|+ N^{-\frac d 2})\big\}  + 1\\
& \le |\! \ln \varepsilon | + (d/2) \ln N + \ln (C_d \| u_0 \| + 1) + 1.
\end{split}
\end{equation} 
Thus, we obtain from \eqref{Pest:e1} and the above estimates that for $0<\varepsilon<1,$
\begin{equation}\label{Pest:e1f}
\begin{aligned}
    \mathcal{T}^t_{1,1} & \le 4 | {\mathbb T} |^{\frac d 2} \varepsilon + C \tilde \alpha(t)   \big( | \! \ln \varepsilon| + (d/2) \ln N+\ln (C_d \| u_0 \| + 1) + 1\big)\,  t^{\frac s2}  | u^k |_{ H^s ({\mathbb T^d})}\\
    &\le C \big( | \! \ln \varepsilon| + \ln N+ \ln (\| u_0 \| + 1)\big)\,  t^{\frac s2}  
    |u|_{ C(H^s)}.
\end{aligned}
\end{equation}
Here, we use $C$ to denote a generic positive constant  independent of $\varepsilon, \tau, t, N$ and any function, and its dependence on other parameters (e.g., $T,d, s, |\lambda|$) can be tracked if necessary. Note that we do not carry the factor $\tilde \alpha(t)=1+2|\lambda| t,$ which actually can be bounded by $\alpha=1+2|\lambda| \tau,$ when we integrate $t\in (0,\tau).$ 

\medskip
We now deal with $\mathcal{T}^t_{1,2}$ given in \eqref{T1234}. Following the derivations in \eqref{Pest:e1}, we can obtain 
\begin{equation}\label{Pest:e2}
\begin{aligned}
\mathcal{T}^t_{1,2} & \le 4|{\mathbb T}|^{\frac d 2} \varepsilon + 2 \Upsilon_{\!2}^t(\varepsilon)  \| ( \Pi_N^d - {\mathbb I}) \Phi_B^t[ u^k_* ] \|  && {\rm (\refl{fufvL2-est})} \\ 
 & \le 4|{\mathbb T}|^{\frac d 2} \varepsilon + 2 \Upsilon_{\!2}^t(\varepsilon)  N^{-s} | \Phi_B^t [ u^k_* ] |_{H^s ({\mathbb T^d}) }  && {\rm (\refl{lem:errProj})} \\ 
& \le 4|{\mathbb T}|^{\frac d 2} \varepsilon + 2\,   \mathcal{C}_1^{-1} \mathcal{C}_2 \tilde \alpha(t)\, \Upsilon_{\!2}^t(\varepsilon)  N^{-s} 
   | u^k  |_{H^s ({\mathbb T^d}) },  && {\rm (\reft{PhiBbnd})} \\ 
\end{aligned}
\end{equation}
where 
\begin{equation}\label{Upsilon2}
\Upsilon_{\!2}^t(\varepsilon) := 
\max\big\{|\!\ln \varepsilon|,\,  \ln(\| \Pi_N^d \Phi_B^t [ u^k_* ]\|_{\infty}+1),\, \ln(\| \Phi_B^t [ u^k_* ]\|_{\infty}+1) \big\}+1.
\end{equation}
From the definition \eqref{PhiB} and   the inverse inequality \eqref{XNinftyto2}, we find readily that 
\begin{equation}\label{PhiBInftyBnd}
\| \Phi_B^t[u^k_*] \|_\infty = \| u^k_*\|_\infty \le C_d N^{\frac d 2} \| u^k\|= C_d N^{\frac d 2} \| u_0\|.
\end{equation}
Thus using \eqref{PhiBLinfty2}, we can bound  
$\Upsilon_{\!2}^t(\varepsilon)$ as with \eqref{Pupsilon1}, and derive the following bound similar to \eqref{Pest:e1f}:
\begin{equation}\label{Pest:e2f}
\mathcal{T}^t_{1,2} \le 
 C \big( | \! \ln \varepsilon| + \ln N+ \ln (\| u_0 \| + 1)\big)\, N^{-s} |u|_{C(H^s)}.
\end{equation}

\medskip

We next turn to  estimate $\mathcal{T}^t_{1,3}$ in \eqref{T1234}.  
Using the triangle inequality and aforementioned lemmas and theorem, leads to 
\begin{equation}\label{ek3ineq0}
\begin{split}
\mathcal{T}^t_{1,3} 
&\le  \| f(\Phi_B^t [ u^k_* ]) - f^{\varepsilon}(\Phi_B^t [ u^k_* ]) \| + \|   f^{\varepsilon}(\Phi_B^t [ u^k_* ]) - \Pi_N^d f^{\varepsilon}(\Phi_B^t [ u^k_* ]) \| \\
&\quad + \| \Pi_N^d (f^{\varepsilon}(\Phi_B^t [ u^k_* ]) - f(\Phi_B^t [ u^k_* ]))\|\\
& \le \| f(\Phi_B^t [ u^k_* ]) - f^{\varepsilon}(\Phi_B^t [ u^k_* ]) \| + \|   f^{\varepsilon}(\Phi_B^t [ u^k_* ]) - \Pi_N^d f^{\varepsilon}(\Phi_B^t [ u^k_* ]) \| \\
&\quad + \| f^{\varepsilon}(\Phi_B^t [ u^k_* ]) - f(\Phi_B^t [ u^k_* ])\| \qquad \qquad \qquad \qquad \quad\;  { (\rm \refl{lem:errProj} )}\\
& = 2\| f(\Phi_B^t [ u^k_* ]) - f^{\varepsilon}(\Phi_B^t [ u^k_* ]) \| + \| (  \Pi_N^d - {\mathbb I} )f^\varepsilon(\Phi_B^t [ u^k_* ]) \| \\
& \le 4 |{\mathbb T}|^{\frac d 2} \varepsilon + \| ( \Pi_N^d - {\mathbb I} )f^{\varepsilon} (\Phi_B^t[u_*^k]) \| \qquad \qquad \qquad\qquad  {( \rm \refl{Lem:fRfdiff} )}\\
& \le 4 |{\mathbb T}|^{\frac d 2} \varepsilon +  N^{-s} | f^\varepsilon(\Phi_B^t [ u^k_* ])  |_{ H^s(\mathbb{T}^d)}  \qquad \qquad \qquad\qquad \!\! {(\rm \refl{lem:errProj} )}\\
& \le 4 |{\mathbb T}|^{\frac d 2} \varepsilon+    C_s\,  \Upsilon_{\!3}^t(\varepsilon)\, N^{-s}\, |\Phi_B^t [ u^k_* ]|_{ H^s ({\mathbb T^d}) } \qquad \quad\qquad  \;\;  {(\rm \refl{Lem:RfHs})} \\
& \le 4 |{\mathbb T}|^{\frac d 2} \varepsilon +   C_s \mathcal{C}_1^{-1} \mathcal{C}_2\, \tilde \alpha(t) \,\Upsilon_{\!3}^t(\varepsilon)\, N^{-s}\, | u^k |_{ H^s ({\mathbb T^d})}, \quad\quad\;\,    {\rm (\reft{PhiBbnd})} 
\end{split}
\end{equation}
where by \eqref{PhiBInftyBnd},
\begin{equation}\label{Upsion3}
\begin{aligned}
\Upsilon_{\!3}^t(\varepsilon)&:=\max\big\{|\!\ln \varepsilon|,\, \ln(\| \Phi_B^t[u^k_*] \|_\infty+1)\big\}+1\\
&\le  |\! \ln \varepsilon| + (d/2)\ln N  + \! \ln ( C_d\| u_0\| + 1)  + 1. 
\end{aligned}
\end{equation}
Thus, similar to \eqref{Pest:e2f}, we have 
\begin{equation}\label{Pest:e3f}
\mathcal{T}^t_{1,3}  
 \le   C \big(  |\! \ln \varepsilon| + \ln N  + \ln (\| u_0 \| + 1)\big)  N^{-s}| u |_{ C(H^s)}.
\end{equation}

Finally, we  estimate $\mathcal{T}^t_{1,4}$ in \eqref{T1234}. From the property: $\|\Phi_A^t[w]\|=\|w\|,$ we immediately derive 
$\|\Phi_A^t[w]-w\|\le 2\|w\|.$ Thus, we have 
\begin{equation}\label{ek4ineq0}
\begin{aligned}
\mathcal{T}^t_{1,4} & = \| ( \Phi_A^t - {\mathbb I} ) f_*(\Phi_B^t [ u^k_* ])\|\\
& \le \| ( \Phi_A^t - {\mathbb I} )  \{ f_*(\Phi_B^t [ u^k_*]) -  f_*^{\varepsilon}(\Phi_B^t [ u^k_* ]) \} \| + \| ( \Phi_A^t - {\mathbb I} )   f_*^{\varepsilon}(\Phi_B^t [ u^k_* ]) \|  \\
& \le 2  \| f_*(\Phi_B^t [ u^k_*]) -  f_*^{\varepsilon}(\Phi_B^t [ u^k_* ]) \|  + \| ( \Phi_A^t - {\mathbb I} )   f_*^{\varepsilon}(\Phi_B^t [ u^k_* ]) \| \\
& \le 4|{\mathbb T}|^{\frac d 2}\varepsilon + \|({\mathbb I} - \Phi_A^t)   f_*^{\varepsilon}(\Phi_B^t [ u^k_* ]) \|,
\end{aligned}
\end{equation}
where we denoted $f_*^{\varepsilon}=\Pi_N^d f^{\varepsilon},$ and in the last step, we used \eqref{eq:ffepsdiff} in \refl{Lem:fRfdiff}.
Using \reft{PhiA:Projerr} and \refl{lem:errProj} and following the last three steps in \eqref{ek3ineq0}, we obtain
\begin{equation*}
\begin{aligned}
\| ( \Phi_A^t -& {\mathbb I} )   f_*^{\varepsilon}(\Phi_B^t [ u^k_* ]) \|  \le 2^{1-\frac s 2}  t^{\frac s 2} |\Pi_N^d  f^{\varepsilon}(\Phi_B^t [ u^k_*])  |_{ H^s ({\mathbb T^d}) } \\
&\le 2^{1-\frac s 2}  t^{\frac s 2} | f^{\varepsilon}(\Phi_B^t [ u^k_*]) |_{ H^s ({\mathbb T^d}) } \\
&\le 2^{1-\frac s 2}  C_s \mathcal{C}_1^{-1} \mathcal{C}_2  \tilde{\alpha}(t) \{  |\! \ln \varepsilon| + (d/2) \ln N  +  \ln (C_d\| u_0\| + 1) + 1 \}  t^{\frac s2} | u^k |_{ H^s ({\mathbb T^d}) }.
\end{aligned}
\end{equation*}
Consequently, we have 
\begin{equation}\label{Pest:e4f}
\mathcal{T}_{1,4}^t \le C \big(  |\! \ln \varepsilon| + \ln N  + \ln (\| u_0 \| + 1) \big) t^{\frac s 2} | u |_{C(H^s)}.
\end{equation}

A combination of the estimates of ${\mathcal T}^t_{1,k}$ for $k=1,2,3,4$ in \eqref{Pest:e1f}, \eqref{Pest:e2f}, \eqref{Pest:e3f} and \eqref{Pest:e4f},   leads to the bound for  $\mathcal{T}^t_1$
in \eqref{e1toe4}, that is, 
\begin{equation}\label{T1t_est}
{\mathcal T}_1^t \le  C \big(  |\! \ln \varepsilon| + \ln N  + \ln (\| u_0 \| + 1) \big) \big(t^{\frac s 2} + N^{-s}\big) | u |_{C(H^s)}. 
\end{equation}

\medskip 
	
\noindent{\bf (ii) Estimate $\mathcal{T}^t_2.$} \, 
By the triangle inequality,
\begin{equation*} \label{e0orig}
\begin{aligned}
\mathcal{T}^t_{2} & =  \|   f_*(\Psi^{t} [ u^k ])  - f( \Psi_*^{t} [ u^k ])\|  = \|  \Pi_N^d f(\Psi^t[u^k])  - f( \Psi_*^t[u^k] ) \| \\
         &\le   \| (\Pi_N^d - {\mathbb I}) \{ f(\Psi^t[u^k]) - f^{\varepsilon}(\Psi^t[u^k]) \} \| + \| (\Pi_N^d - {\mathbb I}) f^{\varepsilon}(\Psi^t[u^k]) \| 
         \\
         &\quad + \| f(\Psi^t[u^k])  - f(\Psi^t_*[u_k])\|,
         \end{aligned}
\end{equation*}
where $ \Psi^{t}_*=\Pi_N^d \Psi^{t}$ as defined in \eqref{PsiStart}. 
Using the property $\|(\Pi_N^d - {\mathbb I})w\|\le 2\|w\|$  and \refl{Lem:fRfdiff}, we can bound the first term by 
\begin{equation*}
\| (\Pi_N^d - {\mathbb I}) \{ f(\Psi^t[u^k]) - f^{\varepsilon}(\Psi^t[u^k]) \} \|  \le 2 \|  f(\Psi^t[u^k]) - f^{\varepsilon}(\Psi[u^k])  \| \le 4|\mathbb{T}|^{\frac d 2} \varepsilon. 
\end{equation*}
We further bound the second term by using \refl{Lem:RfHs} and \refl{lem:errProj} as follows
\begin{equation*}
\begin{aligned}
\| (\Pi_N^d - {\mathbb I}) f^{\varepsilon}(\Psi^t[u^k]) \| &\le  N^{-s}  | f^{\varepsilon} (\Psi^t[u^k])  |_{H^{s}( \mathbb{T}^d ) }\\
&\le  C_s \{ |\! \ln \varepsilon | + \!\ln (\| \Psi^t[u^k]\|_\infty +1 ) + 1\} N^{-s} | \Psi^t[u^k] |_{H^{s}( \mathbb{T}^d )}\\
& \le  C_s  \{ |\! \ln \varepsilon | + \!\ln ( \| u \|_{C(L^\infty)} +1 ) + 1\} N^{-s} | u |_{C(H^{s})},
\end{aligned}
\end{equation*}
where we noticed from \eqref{loc:00} that 
$\Psi^t[u^k](x) = u(t_k+t, x),$ so 
$$ \| \Psi^t[u^k] \|_\infty = \| u(t_k+t,\cdot)\|_\infty \le \| u \|_{C(L^\infty)}, 
\quad |\Psi^t[u^k]|_{H^s(\mathbb{T}^d)} \le |u|_{C( H^s)}.
$$
The third term can be bounded by  \refl{fufvL2-est} and \refl{lem:errProj}:  
\begin{equation*} 
\begin{aligned}
 \| f(\Psi^t[u^k])  - f(\Psi_*^t[u^k])\| &\le   4|\mathbb{T} |^{\frac d 2} \varepsilon +  2 \Upsilon_{\!4}^t(\varepsilon) \| \Psi^t[u^k] - \Psi^t_*[ u^k ]\|  \\
 &\le 4|\mathbb{T} |^{\frac d 2} \varepsilon + 2 \Upsilon_{\!4}^t(\varepsilon)  N^{-s}  | u |_{ C(H^s) }, 
 \end{aligned}
\end{equation*}
where by the inverse inequality \eqref{XNinftyto2}, \eqref{logbnd} and $\| \Psi^t[ u^k ] \| = \| u_0 \| $,
\begin{equation*}
\begin{aligned}
    \Upsilon_4^t(\varepsilon) &:= \max\big\{|\!\ln \varepsilon|,\, \ln(\|\Psi^t[u^k]\|_\infty+\varepsilon), \,\ln(\|\Psi_*^t[u^k]\|_\infty+\varepsilon)\big\}+1 \\
    & \le |\! \ln \varepsilon |  + \ln (\| u \|_{C(L^\infty)} + 1) + (d/2) \ln N + \ln(C_d \| u_0 \| + 1) + 1.
\end{aligned}
\end{equation*}
Thus, we deduce from the above that
\begin{equation}\label{T2t_est}
 \mathcal{T}_2^t \le  C \{ |\! \ln \varepsilon | + \ln N + \!\ln (\| u \|_{C(L^\infty)} +1) 
   + \ln (\| u_0 \| +1 )\}N^{-s}  | u  |_{ C(H^s) }.
\end{equation}

\smallskip
Finally, with the estimates \eqref{T1t_est}-\eqref{T2t_est} at our disposal, we   return to \eqref{loc:03}, and take $\varepsilon = N^{-s}$, leading to 
\begin{equation}\label{loc:0300}
\begin{split}
\frac{\rd}{\rd t} \| \mathcal{E}_k^{t} \|  
&\le  2 |\lambda| \,  \|\mathcal{E}_k^{t}  \|+|\lambda |\,\big\{\mathcal T_1^t+\mathcal T_2^t\big\}\le  2 |\lambda| \,  \|\mathcal{E}_k^{t}  \|\\
&\quad + C \{ \ln N + \ln (\| u \|_{C(L^\infty)} +1) 
   + \ln (\| u_0 \| +1 )\} \big(t^{\frac s 2}+N^{-s}\big)  | u  |_{ C(H^s) }.
\end{split}
\end{equation}
As $\mathcal{E}_k^{0}=0$ (see \eqref{loc:errequ}), we use the Gr\"onwall's inequality on $(0,\tau)$ and note $\mathcal{E}_k^{\tau}=e_*^{k+1}$ (see \eqref{errest:et0}) to obtain
$$
\| e_*^{k+1}\| \le   C\, \tau\, {\rm e}^{2|\lambda|\tau} \big\{\ln (\| u_0 \| +1 )+  \ln (\| u \|_{C(L^\infty)} +1) 
   + \ln N \big\} 
  (\tau^{\frac s 2} + N^{-s} )    \,  |u|_{ C(H^s) },
$$
which implies \eqref{mainestimate}.  This completes the proof. 
\end{proof}

As a consequence of the regularity result in \cite[Theorem 1.1]{carles_low_2023} and \reft{THM:errestfrac}, we have the following estimate for one-dimensional case. 
\begin{corollary}\label{1DHs-est} If $u_0\in H^s(\mathbb T)$ for some $\frac 1 2<s<1,$ then we have the $L^2$-estimate
\begin{equation}\label{main1D}
\| u^{m+1} - u_N^{m+1} \| \le \widetilde {\mathcal C}(\tau^{\frac s 2} + N^{-s}) \ln\! N,
\end{equation}
for $m=0,\ldots,M-1,$ where 
$\widetilde {\mathcal C}=\widetilde {\mathcal C}\,\big(|\lambda|, T, \|u_0\|_{H^s(\mathbb T)}\big)$ is a positive constant.  For $d=2,3,$ we have a similar estimate if the solution $u\in C((0,T]; L^\infty(\mathbb T)).$
\end{corollary}
\begin{proof}  From  \cite{carles_low_2023}, we know that if $u_0\in H^s(\mathbb T^d),$ then the LogSE has a unique solution 
$u \in C(\mathbb{R}; H^s(\mathbb T^d))$ such that  
$\|u\|_{C(H^s)}\le C \|u_0\|_{H^s(\mathbb T^d)}.$  We know from the standard imbedding property that
 $H^s(\mathbb T^d)\hookrightarrow L^\infty (\mathbb T^d)$ if $s>\frac d 2.$  Then the estimate \eqref{main1D} follows from \eqref{mainresultA} immediately.  
 For $d=2,3,$ we require 
 $u\in C((0,T]; L^\infty(\mathbb T))$ to ensure the constant in the logarithmic factor to be bounded. 
\end{proof}

\section{Error estimate for  $s=1$}\label{sec:H1mainResult}
\setcounter{equation}{0}
\setcounter{lmm}{0}
\setcounter{thm}{0}
\setcounter{cor}{0}

In \reft{THM:errestfrac}, we analyzed the convergence of the   LTSFS scheme \eqref{FourierFullS} when $u_0 \in L^2(\mathbb{T}^d)$ and the solution $u$ to \eqref{eq:LogSE} has certain fractional Sobolev regularity.  Remarkably, it was shown that if  $u_0 \in H^1( \mathbb{T}^d ),$  the solution $ u\in C(\mathbb{R}, H^1( \mathbb{T}^d ))$ 
(see \cite[Theorem 2.3]{carles_low_2023} and (ii) in the introductory section).  A natural question that arises is whether we can improve the $L^2$-estimate  \eqref{mainresultA} in \reft{THM:errestfrac} for $0<s<1$
to $s=1$ given the higher regularity. However, this cannot be obtained from the limiting process $s\to 1^-$ and  the main reason that the limit cannot pass is 
$\Phi_B^t[w_0]\notin H^1(\mathbb{T}^d)$ 
(see Remark \ref{L2Hs-Stab}) and $f\notin H^1(\mathbb{T}^d)$. Correspondingly, some results used for the proof of \reft{THM:errestfrac} are not valid for $s=1$
(e.g., \eqref{eq:logPhiBc}). We next take an alternative path to bypass the non-differentiability of  $\Phi_B^t[w_0]$ and $f.$

\begin{theorem}\label{THM:errestH1}
Let $u_0 \in L^2(\mathbb{T}^d)$ and  assume  the solution $u$ of \eqref{eq:LogSE} has the regularity \eqref{errest:cons} with $s=1.$
Then the error estimate \eqref{mainresultA} holds for $s=1,$ that is, 
\begin{equation}\label{errestH1}
\| u^{m+1} - u_N^{m+1} \| \le C {\re^{2|\lambda|T}}  \big( \mathcal{C}(u_0, u)+ \ln N  \big) (\sqrt{\tau} + N^{-1})\, | u |_{ C(H^1) },
\end{equation}
for $m=0,1,\cdots, M-1$, where $ \mathcal{C}(u_0,u)$ is the same as in \reft{THM:errestfrac}. 
\end{theorem}

\begin{proof} The proof follows the same line as that of  \reft{THM:errestfrac}, but needs care to deal with the derivations involving 
$|\Phi_B^t[\cdot]|_{H^s(\mathbb{T}^d)}$ and 
$|f(\cdot)|_{H^s(\mathbb{T}^d)}.$ Indeed, following the steps in the proof of \reft{THM:errestfrac}, we find the first 
one is in the second line of \eqref{Pest:e1}, so we bound the following term in $\mathcal{T}_{1,1}^t$  differently  as follows 
\begin{equation}\label{newdev1}
\begin{aligned}
    \| ( \Phi_A^t - &\mathbb{I}) \Pi_N^d \Phi_B^t  [ u^k_* ] \| \le  \| ( \Phi_A^t - \mathbb{I} ) \Pi_N^d \big( \Phi_B^t  [ u^k_* ] ) - \Phi_B^{t,\varepsilon}  [ u^k_* ]  \big) \| + \|  (\Phi_A^t - \mathbb{I} ) \Pi_N^d \Phi_B^{t,\varepsilon}  [ u^k_* ] \|, 
\end{aligned}
\end{equation}
where the regularised 
$
\Phi_B^{t,\varepsilon}[w_0] := \re^{-\ri \lambda t  \ln( |w_0| + \varepsilon)^2 } w_0.
$
We can verify readily that 
\begin{equation} \label{RPhiBdiff}
\begin{split}
  |  \Phi_{B}^{t,\varepsilon}  [ w_0 ] - \Phi_{B}^{t}  [ w_0 ]  | 
  \le 2 |\lambda | t \,\varepsilon,
  \end{split}
\end{equation}
which is trivial for $|w_0|=0,$ and for $|w_0|\not =0,$
\begin{equation*}
\begin{split}
| \Phi_{B}^{t,\varepsilon} [w_0] - \Phi_{B}^{t} [w_0] | 
& =2 | w_0 | \Big|\sin \Big(\lambda t \ln \Big(\frac{ |w_0| + \varepsilon}{ |w_0|} \Big) \Big) \Big|  \le 2 | \lambda |\,  t\, | w_0 |\, \ln \Big( 1 + \frac{\varepsilon}{|w_0|} \Big),
\end{split}
\end{equation*}
so we can get \eqref{RPhiBdiff} using $\ln(1+x)\le x$ for $x\ge 0$. Moreover, we need to use  the  property: 
\begin{equation}\label{rPhiBH1}
\|\nabla \Phi_B^{t,\varepsilon}[w_0]\| \le  (1+ 2|\lambda|t )  \|\nabla w_0 \|,
\end{equation}
which follows directly from (see \eqref{eq:loggradPhiB})
\begin{equation*}
\nabla \Phi_B^{t,\varepsilon}[w_0] = {\rm e}^{-\ri \lambda t \ln ( | w_0 | + \varepsilon)^2}\Big(\nabla w_0  - 2\ri \lambda t  \frac{ w_0}{ | w_0 | + \varepsilon } \nabla |w_0|\Big), 
\end{equation*}
and the fact $|\nabla | w_0| | \le  |\nabla w_0 |.$ 
With these two intermediate results, we can now estimate  the two terms in \eqref{newdev1}: 
\begin{equation*}
\begin{aligned}
    \| ( \Phi_A^t - &\mathbb{I}) \Pi_N^d \Phi_B^t  [ u^k_* ] \| \le  \| ( \Phi_A^t - \mathbb{I} ) \Pi_N^d \big( \Phi_B^t  [ u^k_* ] ) - \Phi_B^{t,\varepsilon}  [ u^k_* ]  \big) \| + \|  (\Phi_A^t - \mathbb{I} ) \Pi_N^d \Phi_B^{t,\varepsilon}  [ u^k_* ] \|  \\
    & \le 2 \| \Pi_N^d \big( \Phi_B^t  [ u^k_* ]  - \Phi_B^{t,\varepsilon}  [ u^k_* ]  \big) \| + \sqrt{2} \, t^{\frac12} \big| \Pi_N^d \Phi_B^{t,\varepsilon}  [ u^k_* ] \big|_{H^1(\mathbb{T}^d) } \qquad {(\rm \reft{PhiA:Projerr}) } \\
    & \le 2\| \Phi_B^t  [ u^k_* ]  - \Phi_B^{t,\varepsilon}  [ u^k_* ]  \| + 
    \sqrt{2} \, t^{\frac12} \big|  \Phi_B^{t,\varepsilon}  [ u^k_* ] \big|_{H^1(\mathbb{T}^d) } \qquad\qquad\qquad { (\rm \refl{lem:errProj}) }\\
    & \le  4 |\mathbb{T}|^{ \frac{d}{2} }|\lambda| t \varepsilon + \sqrt{2} \tilde{\alpha}(t) t^{\frac12} |u_*^k|_{H^1( \mathbb{T}^d )} \qquad \qquad \qquad\qquad\qquad\quad\,\,  (\eqref{RPhiBdiff}\; \&\; \eqref{rPhiBH1})  \\
    & \le  4 |\mathbb{T}|^{ \frac{d}{2} } |\lambda| t \varepsilon + \sqrt{2} \tilde{\alpha}(t) t^{\frac12} |u^k|_{H^1( \mathbb{T}^d )}. \qquad \qquad \,\qquad\qquad\qquad \quad { (\rm \refl{lem:errProj} ) }
\end{aligned}
\end{equation*}
It is seen that the key is to shift the differentiation to the regularized $\Phi_B^{t,\varepsilon}  [ u^k_* ]$ with an extra $\varepsilon$-term. 
Then the estimate of $\mathcal{T}_{1,1}^t$ in \eqref{Pest:e1f} becomes 
\begin{equation}\label{H1:T11}
\begin{split}
\mathcal{T}_{1,1}^t
& \le C \{ |\! \ln \varepsilon| + \ln N + \ln(\| u_0 \| +1)  \}  t^{\frac 12}\, |u|_{C(H^1)}.
\end{split}
\end{equation}

The same issue happens to  $\mathcal{T}_{1,2}^t$, i.e.,  the second line of \eqref{Pest:e2}. Using \refl{lem:errProj} and \eqref{RPhiBdiff}-\eqref{rPhiBH1} again, we can obtain  
\begin{equation*}
\begin{aligned}
    \| ( \Pi_N^d - \mathbb{I}  ) \Phi_B^t [u^k_*] \| &\le \| (\Pi_N^d - \mathbb{I} )( \Phi_B^t[  u^k_*  ] -  \Phi_B^{t,\varepsilon}[  u^k_* ] )   \| + \| ( \Pi_N^d - \mathbb{I} ) \Phi_B^{t,\varepsilon}[ u^k_* ] \| \\
& \le  4|\mathbb{T}|^{\frac d 2} t \varepsilon +   N^{-1} | \Phi_B^{t,\varepsilon} [u^k_*] |_{ H^1( \mathbb{T}^d ) } \\
& \le 4|\mathbb{T}|^{\frac d 2} t \varepsilon +  \tilde{\alpha}(t) N^{-1} | u^k |_{H^1(\mathbb{T}^d)}.
\end{aligned}
\end{equation*}
With this change, we can update the bound of $\mathcal{T}_{1,2}^t$ in  \eqref{Pest:e2f}  as  
\begin{equation}\label{H1:T12}
\begin{split}
\mathcal{T}_{1,2}^t 
& \le  C\{ |\! \ln \varepsilon | + \ln N + \ln (\| u_0 \| + 1) \}  N^{-1}  | u |_{C(H^1)}.
\end{split}
\end{equation}

The situation with $\mathcal{T}_{1,3}^t, \mathcal{T}_{1,4}^t$ is slightly different, where $\Phi_B^t[u^k_*]$ appears in $f^\varepsilon(\Phi_B^{t} [u^k_*]),$ so we follow the same argument by inserting  $f^\varepsilon(\Phi_B^{t,\varepsilon} [u^k_*]).$ 
More precisely, we extract from \eqref{ek3ineq0} the following term and use  \refl{lem:errProj}, \refl{fufvL2-est} and \eqref{RPhiBdiff}-\eqref{rPhiBH1} to derive  
\begin{equation*}
\begin{aligned}
     \| ( \Pi_N^d - \mathbb{I} ) f^\varepsilon(\Phi_B^{t} [u^k_*] ) \| &\le 
     \| ( \Pi_N^d - \mathbb{I}  ) \{ f^\varepsilon(\Phi_B^t [ u^k_* ]) - f^\varepsilon(\Phi_B^{t,\varepsilon} [ u^k_* ] ) \} \|  +  \| ( \Pi_N^d - \mathbb{I} ) f^\varepsilon(\Phi_B^{t,\varepsilon} [u^k_*] ) \|  \\
     &\le 4 \Upsilon_3^t(\varepsilon) \| \Phi_B^{t} [u^k_*] - \Phi_B^{t,\varepsilon} [u^k_*] \| + 2 \Upsilon_3^t(\varepsilon) N^{-1} | \Phi_B^{t,\varepsilon}[u_*^k] |_{H^1(\mathbb{T}^d)}  \\
     & \le 2  \Upsilon_3^t(\varepsilon)\big\{ 4 |\mathbb{T}|^{\frac d2}|\lambda| t \varepsilon + \tilde{\alpha}(t) N^{-1} |u^k|_{H^1(\mathbb{T}^d)}\big\},
\end{aligned}
\end{equation*}
where we noted  $\| \Phi_B^{t,\varepsilon}[u_*^k] \|_\infty = \| \Phi_B^t[u_*^k] \|_\infty$.
Thus,  \eqref{Pest:e3f} is valid for $s=1,$ that is,  
\begin{equation}\label{H1:T13}
\begin{aligned}
\mathcal{T}_{1,3}^t 
& \le C\{ |\ln \varepsilon| +  \ln N + \ln( \| u_0 \| + 1)\} N^{-1} | u |_{C(H^1)}.
\end{aligned}
\end{equation}
Similarly, we re-estimate the following term in  $\mathcal{T}_{1,4}^t$ (see \eqref{ek4ineq0}) by using \reft{PhiA:Projerr}, \refl{fufvL2-est} and \eqref{RPhiBdiff}-\eqref{rPhiBH1} as follows
\begin{equation*} 
\begin{aligned}
   \| ( \Phi_A^t - \mathbb{I} ) \Pi_N^d  f^{\varepsilon}(\Phi_B^t [ u^k_*]) \| &\le \| ( \Phi_A^t -\mathbb{I} ) \Pi_N^d \{ f^{\varepsilon}(\Phi_B^t [ u^k_*] )  - f^{\varepsilon}(\Phi_B^{t,\varepsilon} [ u^k_*])  \}\|  \\
&\quad + \| (\Phi_A^t - \mathbb{I} ) \Pi_N^d f^{\varepsilon}(\Phi_B^{t,\varepsilon} [ u^k_*])  \| \\
& \le 4 \Upsilon_3^{t}(\varepsilon)  \| \Phi_B^t [ u^k_*] - \Phi_B^{t,\varepsilon} [ u^k_*] \| + \sqrt{2}\,  t^{\frac 12}  | f^{\varepsilon}(\Phi_B^{t,\varepsilon} [ u^k_*])|_{H^1(\mathbb{T}^d)} \\
& \le 2 \Upsilon_3^{t}(\varepsilon) \big\{ 4 |\mathbb{T}|^{\frac d2} |\lambda| t \varepsilon + \sqrt{2} \tilde{\alpha}(t) |u^k|_{ H^1({\mathbb T^d}) } \big\},
\end{aligned}
\end{equation*}
so \eqref{Pest:e4f} holds for $s=1,$ namely, 
\begin{equation}\label{H1:T14}
\begin{aligned}
\mathcal{T}_{1,4}^t 
& \le C \{ |\! \ln \varepsilon| + \ln N + \ln(\| u_0 \| +1 ) \} t^{\frac 12} | u^k |_{C(H^1)}.
\end{aligned}
\end{equation}
Consequently, we obtain from \eqref{H1:T11}-\eqref{H1:T14} that 
\begin{equation}
    \mathcal{T}_1^t \le C \{ |\! \ln \varepsilon| + \ln N + \ln(\| u_0 \| +1 ) \} (t^{\frac 12} + N^{-1}) | u |_{C(H^1)}.
\end{equation}



It is important to notice that the estimate of $\mathcal{T}_2^t$ only involves the  regularity of the solution $u.$ In other words, the bound of $\mathcal{T}_2^t$ in  \eqref{T2t_est} is  valid for $s=1$.

With the updated bounds of  $\mathcal{T}_1^t+\mathcal{T}_2^t$ with $s=1$ in  \eqref{loc:03} and  \eqref{loc:0300}, we follow the same process to obtain the desired estimate. 
\end{proof}

Thanks to \reft{THM:errestH1} and the regularity result in  \cite[Theorem 2.3]{carles_low_2023} (also see (ii) in the introductory section), we have the following one-dimensional estimate similar to that in  Corollary \ref{1DHs-est}. 
\begin{corollary}\label{1DH1-est} If $u_0\in H^1(\mathbb T),$  then we have the $L^2$-estimate
\begin{equation}\label{main1D-H1}
\| u^{m+1} - u_N^{m+1} \| \le \widehat {\mathcal C}\, (\sqrt{\tau} + N^{-1}) \ln\! N,
\end{equation}
for $m=0,\ldots,M-1,$ where 
$\widehat {\mathcal C}=\widehat {\mathcal C}\,\big(|\lambda|, T, \|u_0\|_{H^1(\mathbb T)}\big)$ is a positive constant.  For $d=2,3,$ we have a similar estimate if the solution $u\in C((0,T]; L^\infty(\mathbb T)).$
\end{corollary}

\section{Numerical results}\label{sec:numerresults}
In this section, we provide ample numerical results to validate the convergence of 
the LTSFS scheme \eqref{FourierFullS} 
for solving the LogSE \eqref{eq:LogSE} with low regularity initial data. 
  Given the nature of singularity, we find it is without loss of generality to  test on the method  in one spatial   dimension.



\subsection{Accuracy test on exact Gausson solutions} 
It is known that the LogSE has exact soliton-like Gausson solutions  (see  \cite{Bialynicki1979Gaussons} and \cite[(1.7)]{Bao2019Error}): 
\begin{equation}\label{ex-uxt}
  u(x, t) = b \,\exp\big\{\ri \big(x\cdot \zeta - (a+|\zeta|^2)t \big) + (\lambda/2)|x-2\zeta t|^2 \big\},\quad  x \in \mathbb R^d,\;  t \ge 0,
\end{equation}
where $a = -\lambda (d-\ln|b|^2)$, and $b, \lambda \in \mathbb R \setminus \{0\}$, $\zeta \in \mathbb R^d$ are free  to choose.
Then we take the initial input to be 
\begin{equation}\label{ex:oneu0}
    u_0(x)=u(x,0) = b \,\exp\big(\ri x\cdot \zeta   + (\lambda/2)|x|^2 \big),\quad  x \in \mathbb R^d.
\end{equation}

 In the following tests, we take  $d = 1,\lambda = -16$, $b =1, \zeta=0$, and set the interval of computation to be $[-\pi,\pi)$ so that $|u_0(x)|_{x= \pm\pi}\approx 5.12\times 10^{-35}$ to enforce the periodic conditions.  
 To demonstrate the convergence order in time, we use $N=200$ Fourier modes so that the error is dominated in time. 
 In Figure \ref{fig:exacttimeorder}, we plot on the left the $L^2$-errors at $t=0.4, 0.7, 1$ against the time-stepping sizes of $\tau$  from $10^{-4}$ to $10^{-2}$, and depict on the right time evolution of the mass. 
We observe from Figure \ref{fig:exacttimeorder} a perfect first-order convergence and a good conservation of the initial mass.  Note that  the solution \eqref{ex-uxt} is sufficiently smooth, and 
$$|u(x,t)|=|b| e^{\lambda|x-2\zeta t|^2/2}>0,\quad x\in {\mathbb R^d},\; t\ge 0,$$
which implies $f(u)=u\ln |u|^2$ and $\Phi_B^t[u]= u\,\re^{-\ri \lam  t \ln |u|^{2}}$ are smooth.  On the other hand,  we infer from  \reft{PhiA:Projerr} that the splitting scheme is first-order under $H^2$-regularity.  We believe the first-order convergence is attributed to these two facts, though the rigorous justification appears subtle.  In addition, under the $H^2$-regularity, 
 Bao et al \cite[Remark 2]{Bao2019Regularized} showed 
 the convergence  order
 ${\mathcal O}(\tau |\ln \varepsilon|/\varepsilon)$  of the time-splitting scheme  
for the regularised LogSE. Then taking $\varepsilon=\tau^\delta$ leads to 
the convergence  order
 ${\mathcal O}(\tau^{1-\delta} |\ln \tau|)$ for any $0<\delta<1.$ However, a similar  estimate for the non-regularised case is not available which appears open.  

 

\begin{figure}[!th]
    \centering
    \hspace{-0.2in}
    \includegraphics[width=0.42\textwidth,height=0.36\textwidth]{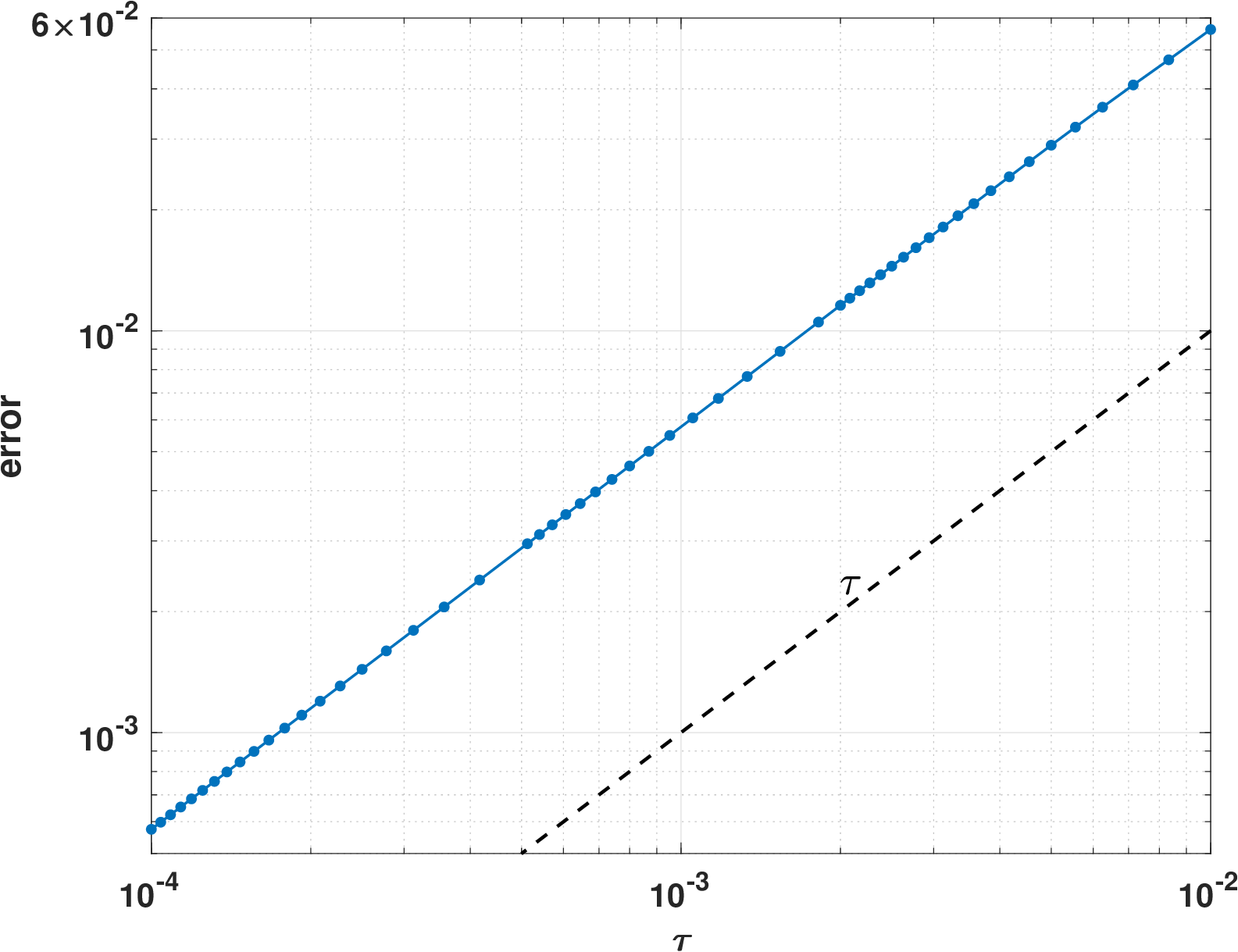}~~\quad \qquad
    \includegraphics[width=0.42\textwidth,height=0.36\textwidth]{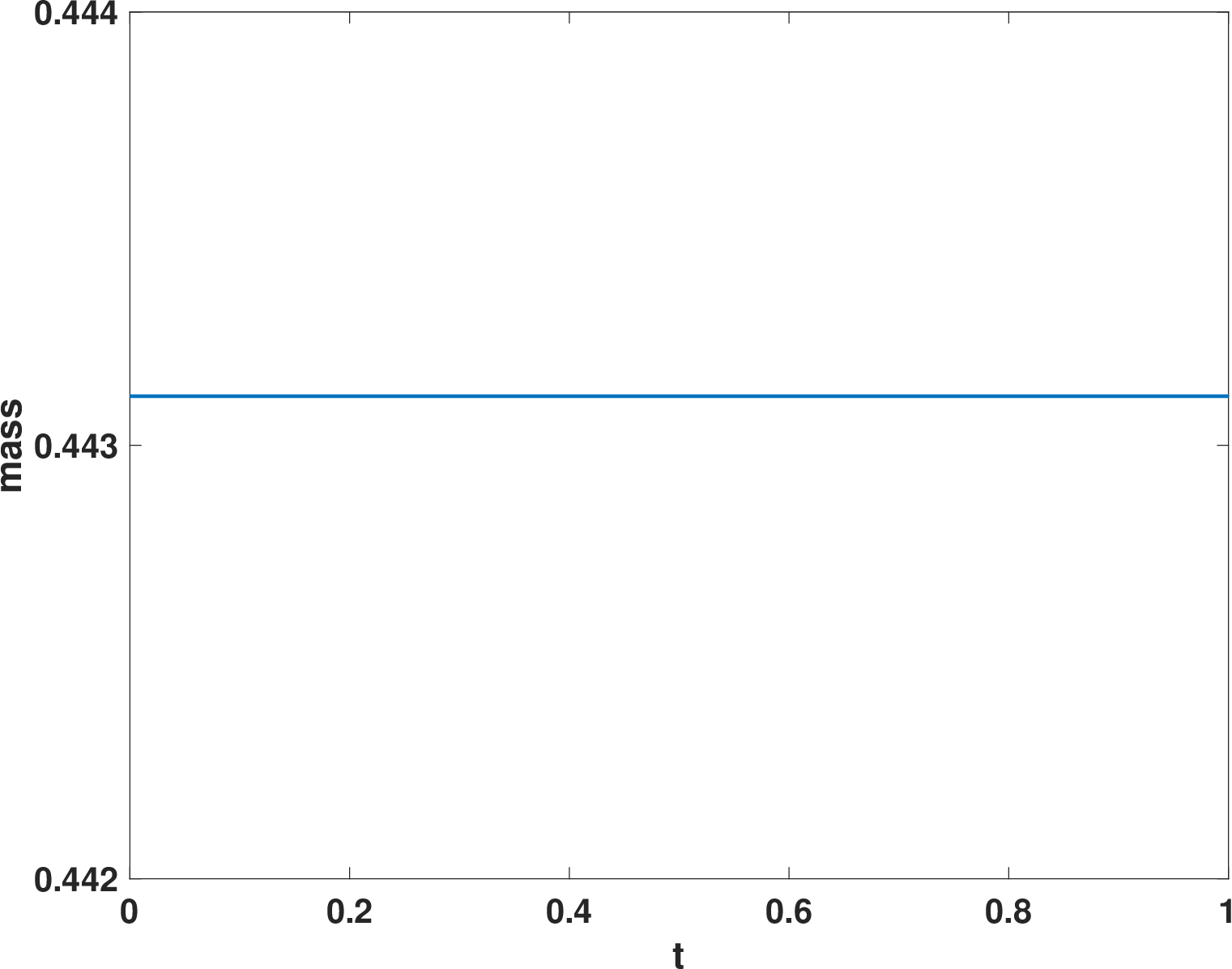} 
    \caption{Left: Convergence of the scheme \eqref{FourierFullS} in time with the initial value given in \eqref{ex:oneu0}. Left:  $L^2$-error against $\tau$ in the log-log scale. Right:  Evolution of mass  for $t\in [0,1]$. }
    \label{fig:exacttimeorder}
\end{figure}

\subsection{Low regularity initial data $u_0\in H^s(\mathbb T)$ for some $s\in (1/2, 1)$}\label{subsect:numHs} To verify the fractional-order convergence behaviours, we consider the following  two examples.    
\medskip 

\noindent{\bf Example 1: $u_0\in H^s(\mathbb T)$ generated by  decaying random  Fourier coefficients.}\, 
We construct $u_0(x)$ through properly decaying Fourier expansion coefficients $\{\hat u_{0,k}\}$ given by 
\begin{equation}\label{ex2:u01A}  
\hat u_{0,0}=0,\quad 
\hat u_{0,k}= \frac{a_k}{|k|^{s+\beta}}, \quad |k|\ge 1,
\end{equation}
where  $a_k\in {\mathbb C}$ with real and imaginary parts being  uniformly distributed random numbers on $[-1,1].$ We find from  the definition \eqref{FSnormsemi}
that if $\beta>1/2,$ then $u_0\in H^s(\mathbb{T}),$ since   
\begin{equation*}\label{u0Hs-ver}
\begin{aligned}
 \|u_0 \|^2_{H^s(\mathbb{T})} &= \sum_{|k|=0}^\infty (|k|^2+1)^s |\hat u_{0,k}|^2= \sum_{|k|=1}^\infty  |a_k|^2 \frac{(|k|^2+1)^s}{|k|^{2s+2\beta}}
 \le 2^{s+1} \sum_{|k|=1}^\infty \frac{1}{|k|^{2\beta}} 
 \\ & \le 2^{ s+1} \int_1^\infty \frac{1}{x^{2\beta}} \rd x 
 =\frac{ 2^{ s+1}}{2\beta -1}<\infty.  
\end{aligned}
\end{equation*}
In real implementation,  we truncate the infinite series with a cut-off number 
$K=10^6$ and randomly generate 
$$
\Re (a_k)=2*\text{\tt rand}(K,1)-1,\quad \Im(a_k)=2*\text{\tt rand}(K,1)-1,
$$
where  ``{\tt rand}(\;)'' is the Matlab routine  to produce 
a random vector drawn from the uniform distribution in the interval $[0,1].$

In this case, the exact solution is unavailable, so  we use the LTSFS scheme 
with 
\begin{equation}\label{ref-solu-tau}
\tau=2^{-18}\approx 3.8\times 10^{-6}, \quad N=[1/\sqrt{\tau}],
\end{equation}
to compute a reference ``exact'' solution, denoted by $u_{\rm ref}(x,t)$ at $t=m\tau.$
In Figure \ref{fig:Hsu00} (a)-(b), we plot the real and imaginary parts of $u_0(x)$ and 
$u_{\rm ref}(x,t)$ at $t=0.4, 0.7,1,$ on the left and right, respectively, for 
$s=0.8$ and $\beta=0.51.$ Observe that  the initial value and the reference solutions at different time are apparently  continuous but not differentiable.



We now  examine the error 
\begin{equation}\label{Esm-tau}
{\mathbb E}_s^m(\tau):= \frac{\|u_{\rm ref}(\cdot, m\tau)-u_N^m(\cdot)\|}
{\|u_0\|_{H^s(\mathbb T)}} \le {\mathcal C}(|\lambda|,T,\|u_0\|_\infty)\; \tau^{\frac s 2}|\ln \tau|,
\end{equation}
with the error bound predicted by \reft{THM:errestfrac} and Corollary \ref{1DHs-est},
where we take $N=[1/\sqrt{\tau}]$ and the constant $\mathcal C$ depends   on $\ln (\|u_0\|_\infty+1).$ 
In Figure \ref{fig:Hsu00} (c), we plot the convergence in $\tau$ which roughly indicates the order ${\mathcal O}(\tau^{\frac s 2})$ as expected.  It also shows the scheme is stable as the errors do not increase with  time. We record the time evolution of  mass in Figure \ref{fig:Hsu00} (d), which shows a good preservation of this quantity. 

\begin{figure}[!th]
		\centering
		\subfigure[$\Re\{u_{\rm ref}(x,t)\}$]
{\includegraphics[width=0.40\textwidth,height=0.35\textwidth]{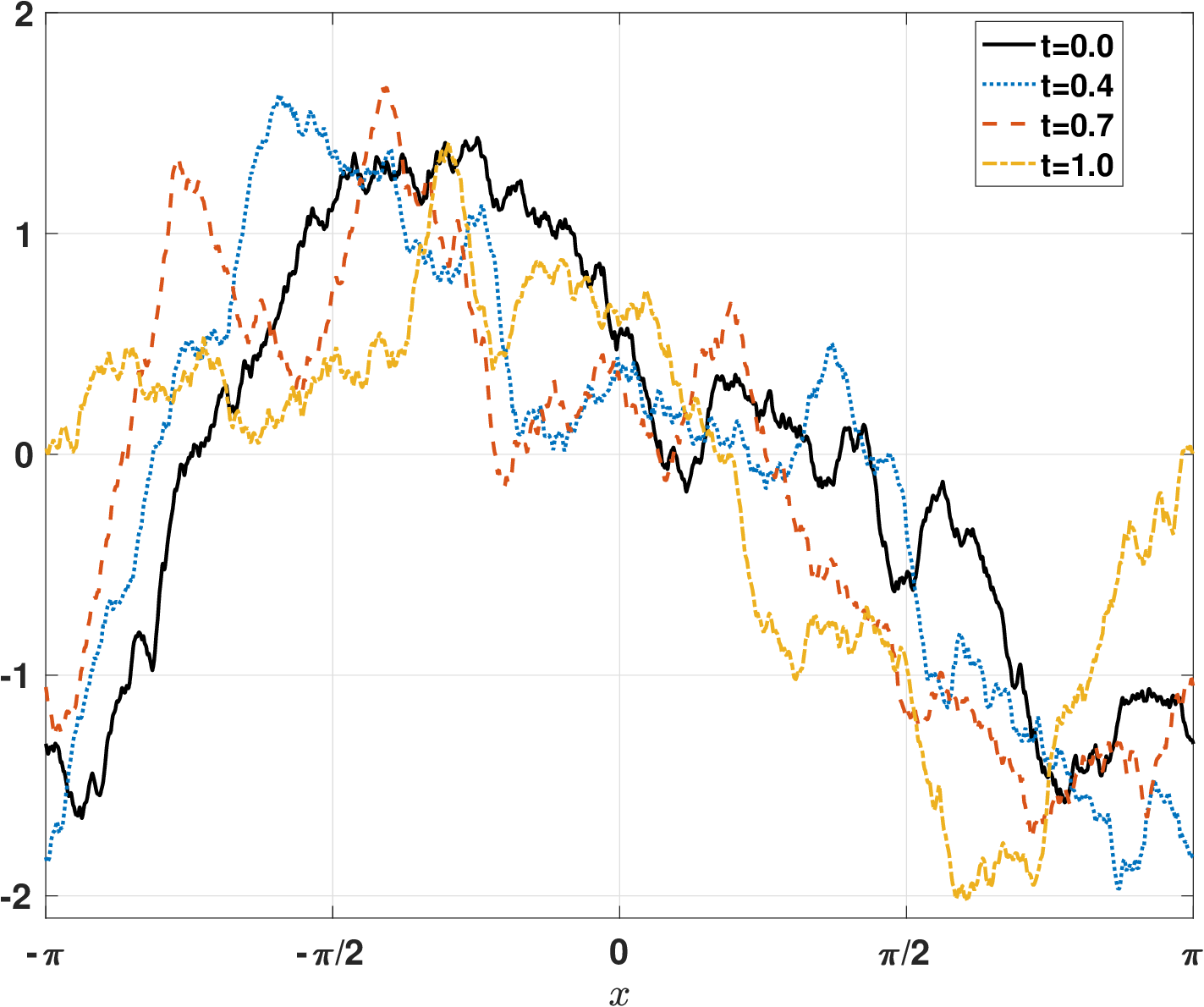}} \quad\qquad 
		\subfigure[$\Im\{u_{\rm ref}(x,t)\}$]
{\includegraphics[width=0.40\textwidth,height=0.35\textwidth]{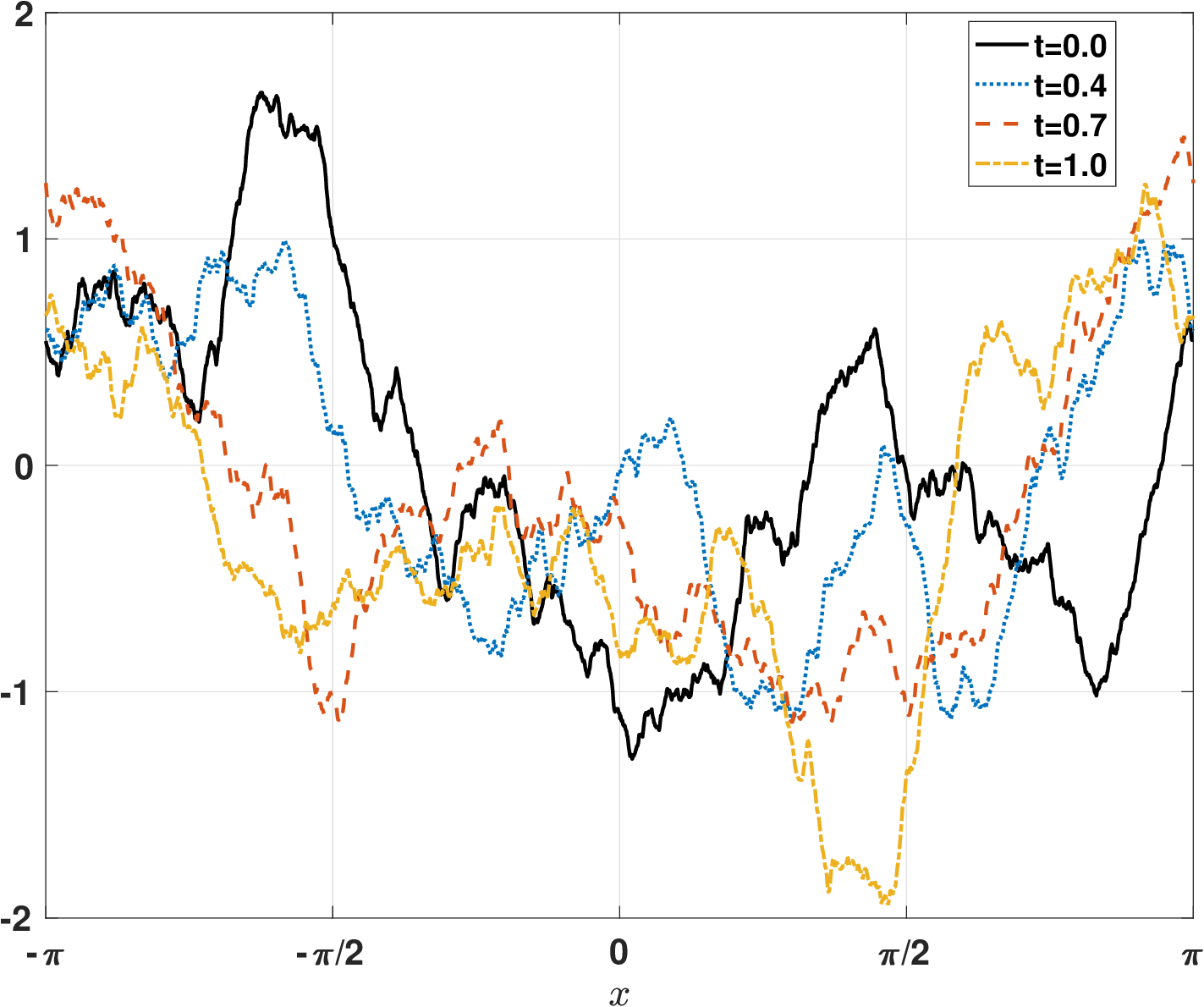}}
		\subfigure[Error ${\mathbb E}_s^m(\tau)$]
{\includegraphics[width=0.42\textwidth,height=0.36\textwidth]{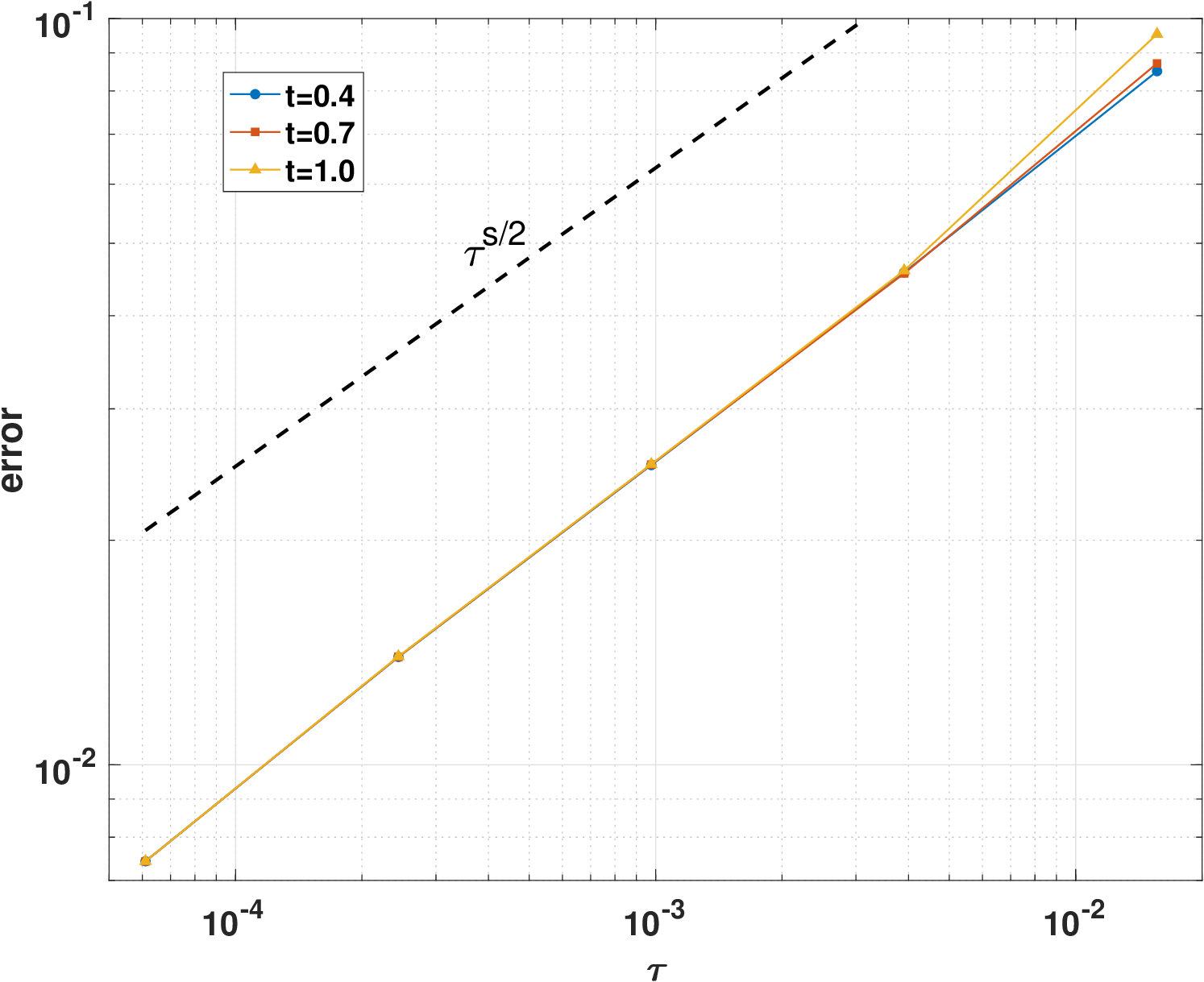}}\qquad 
		\subfigure[Evolution of mass]{
\includegraphics[width=0.42\textwidth,height=0.36\textwidth]{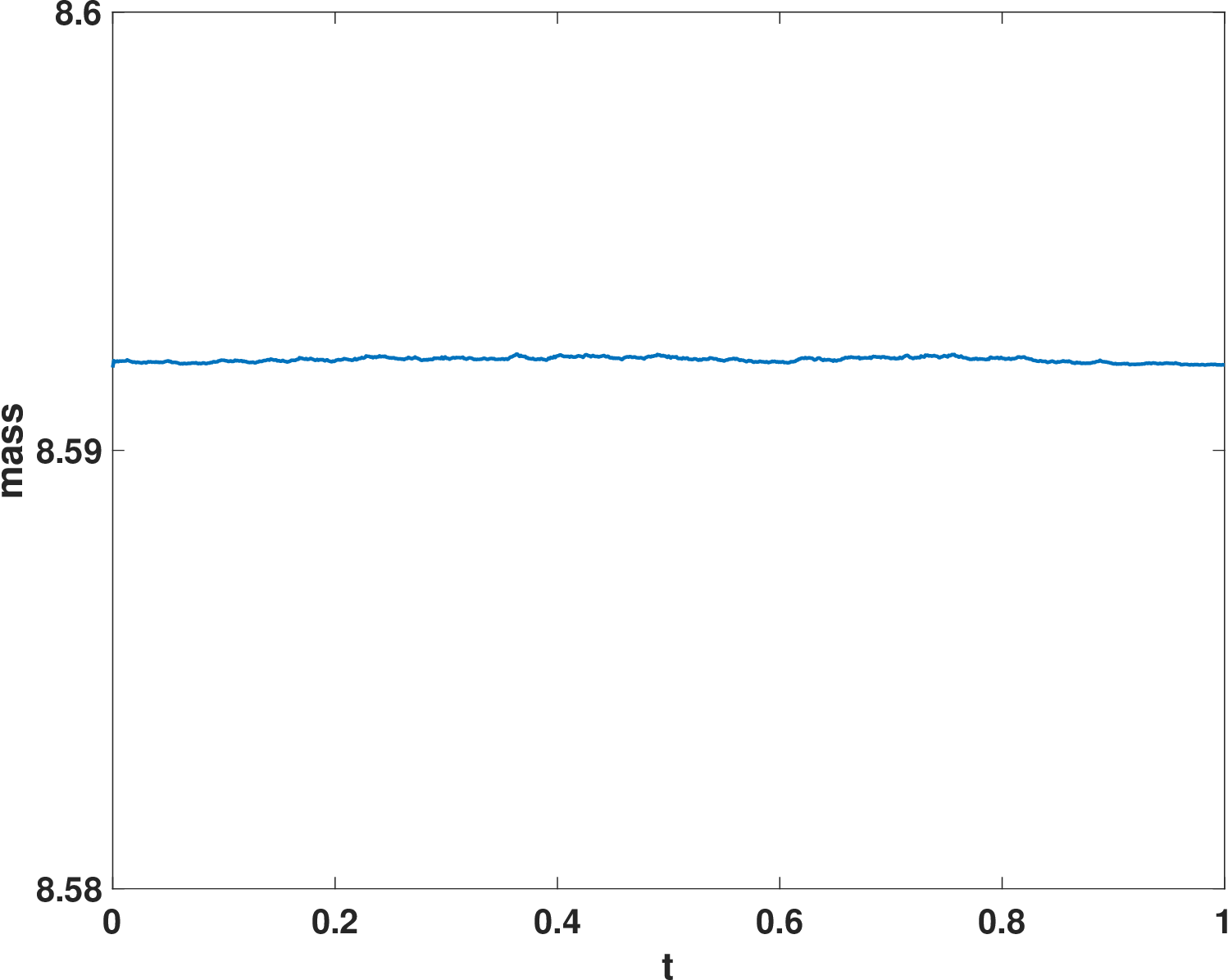}}
		\caption{Numerical results for {\bf Example 1}. (a)-(b):\, Graphs of
    $u_0(x)$ given in \eqref{ex2:u01A}  with  $s=0.8, \beta=0.51,$ 
    and the reference solution $u_{\rm ref}(x,t)$ at $t=0.4, 0.7,1,$ computed by the scheme with $\tau, N$ given  in \eqref{ref-solu-tau}. (c):\, Error ${\mathbb E}_s^m(\tau)$ against $\tau$  in log-log scale 
    for  $t=m\tau=0.4, 0.7, 1$. (d):\, Evolution of mass for $t\in [0,1].$}\label{fig:Hsu00}
\end{figure}



\medskip

\noindent{\bf Example 2: Singular initial  data $u_0\in H^s(\mathbb T).$}\,  To further validate  the fractional order convergence, we consider \eqref{eq:LogSE}  with the following singular initial data:
\begin{equation}\label{u0typeii}
    u_0(x)= |x|^{\gamma} \re^{\ri \ell x},\quad x\in \mathbb{T},  \quad \ell \in \mathbb{Z},
\end{equation}
for some $\gamma\in (0,1).$ From the Taylor expansion of $\re^{\ri \ell x}$, we know the leading singularity of $u_0(x)$ is $\phi(x):=|x|^\gamma.$ Note that $|x|^\gamma$ is  
$\gamma$-H\"older continuous in the sense that  
$\big| |x|^\gamma - |y|^\gamma \big| \le \big| |x| - |y| \big|^\gamma$ (see \cite[pp 3-4]{fiorenza_holder_2016}).
Moreover, we find from   the definition \eqref{defnGag} and direct calculation that 
\begin{equation}\label{num:ex2frac0}
\begin{aligned}
[\phi]^2_{W^{s,2}(\mathbb{T})}&= \int_{\mathbb{T}} \int_{\mathbb{T}} \frac{\big||x|^{\gamma}-|y|^{\gamma}\big|^2}{|x-y|^{1+2s}} \rd x \rd y   
\le  \int_{\mathbb{T}} \int_{\mathbb{T}} \frac{\big| |x|-|y| \big|^{2\gamma}}{|x-y|^{1+2s}} \rd x \rd y  \\
&\le \int_{\mathbb{T}} \int_{\mathbb{T}} \frac{|x-y|^{2\gamma}}{|x-y|^{1+2s}} \rd x \rd y
\;=2\iint_{\mathbb{T}^2\cap \{ x\ge y\} }  (x-y)^{2\gamma-2s -1} \rd x \rd y \\
&= \frac{ (2\pi)^{2\gamma -2s +1} }{(\gamma-s)(2\gamma -2s +1)},
\end{aligned}
\end{equation}
which is finite if $s<\gamma+1/2,$ and  implies  $|x|^{\gamma} \in W^{\gamma+1/2-\epsilon,2}(\mathbb{T})$ for sufficiently small $\epsilon>0.$ It is noteworthy that  Liu et al \cite{Liu19MC-Optimal} introduced  an optimal fractional Sobolev space characterised by the Riemann-Liouville fractional derivative and showed  the regularity index $s=\gamma+1/2$ without $\epsilon.$
In addition, it seems subtle to show $|x|^\gamma\in H^{\gamma+1/2}(\mathbb T)$ under the $2\pi$-periodic extension, though   
 this regularity can be testified to by the numerical evidences below.



Here, we  compute the reference solution $u_{\rm ref}(x)$  with  singular $u_0(x)$ given in \eqref{u0typeii} (where $\ell=2$) and 
 using the LTSFS with the discretisation parameters given in  \eqref{ref-solu-tau}. The curves of the reference solutions at different time $t$ are depicted  in Figure \ref{fig:Hsu0typeb} (a)-(b).  The error plot in  Figure \ref{fig:Hsu0typeb} (c) clearly indicates  an 
 $\mathcal{O}(\tau^{\frac s2})$ (with $s=\gamma+1/2$) convergence which  agrees well with the prediction in \reft{THM:errestfrac} and Corollary \ref{1DHs-est}.
The evolution of mass recorded in Figure \ref{fig:Hsu0typeb} (d)  shows a good  conservation of mass.

\begin{figure}[!th]
		\centering
		\subfigure[$\Re\{u_{\rm ref}(x,t)\}$]
{\includegraphics[width=0.40\textwidth,height=0.35\textwidth]{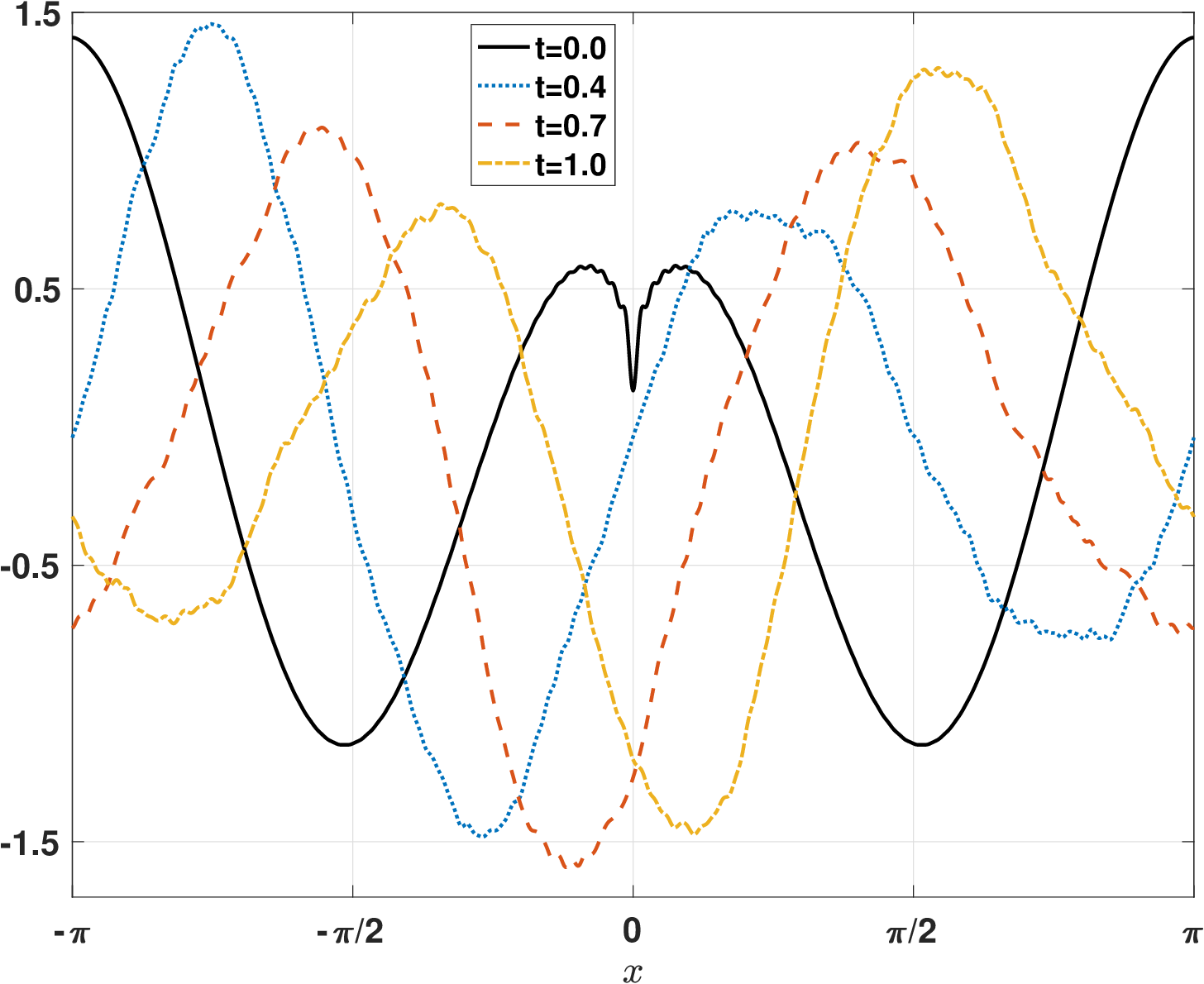}} \quad\qquad 
		\subfigure[$\Im\{u_{\rm ref}(x,t)\}$]
{\includegraphics[width=0.40\textwidth,height=0.35\textwidth]{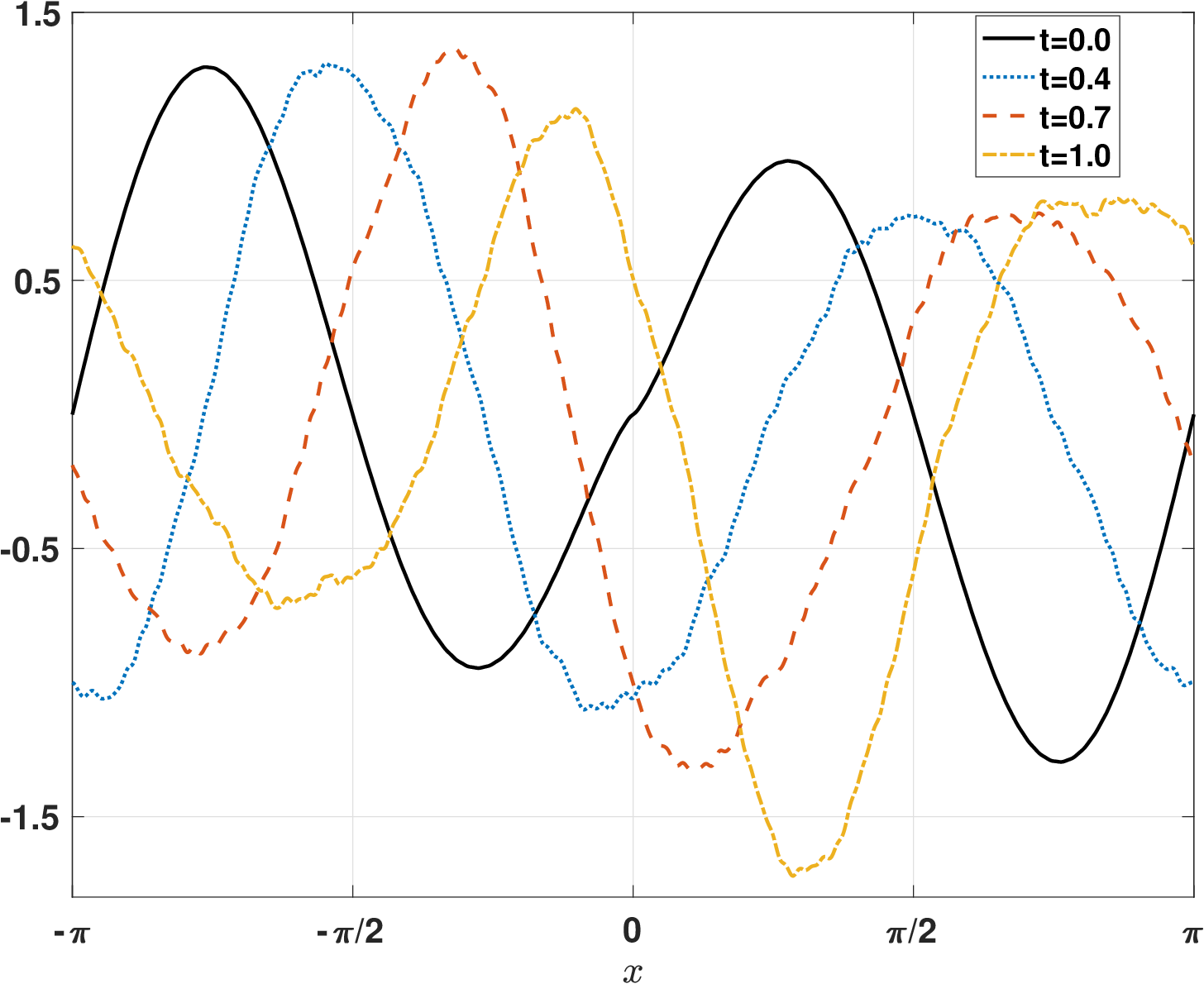}}
		\subfigure[Error ${\mathbb E}_s^m(\tau)$]
{\includegraphics[width=0.42\textwidth,height=0.36\textwidth]{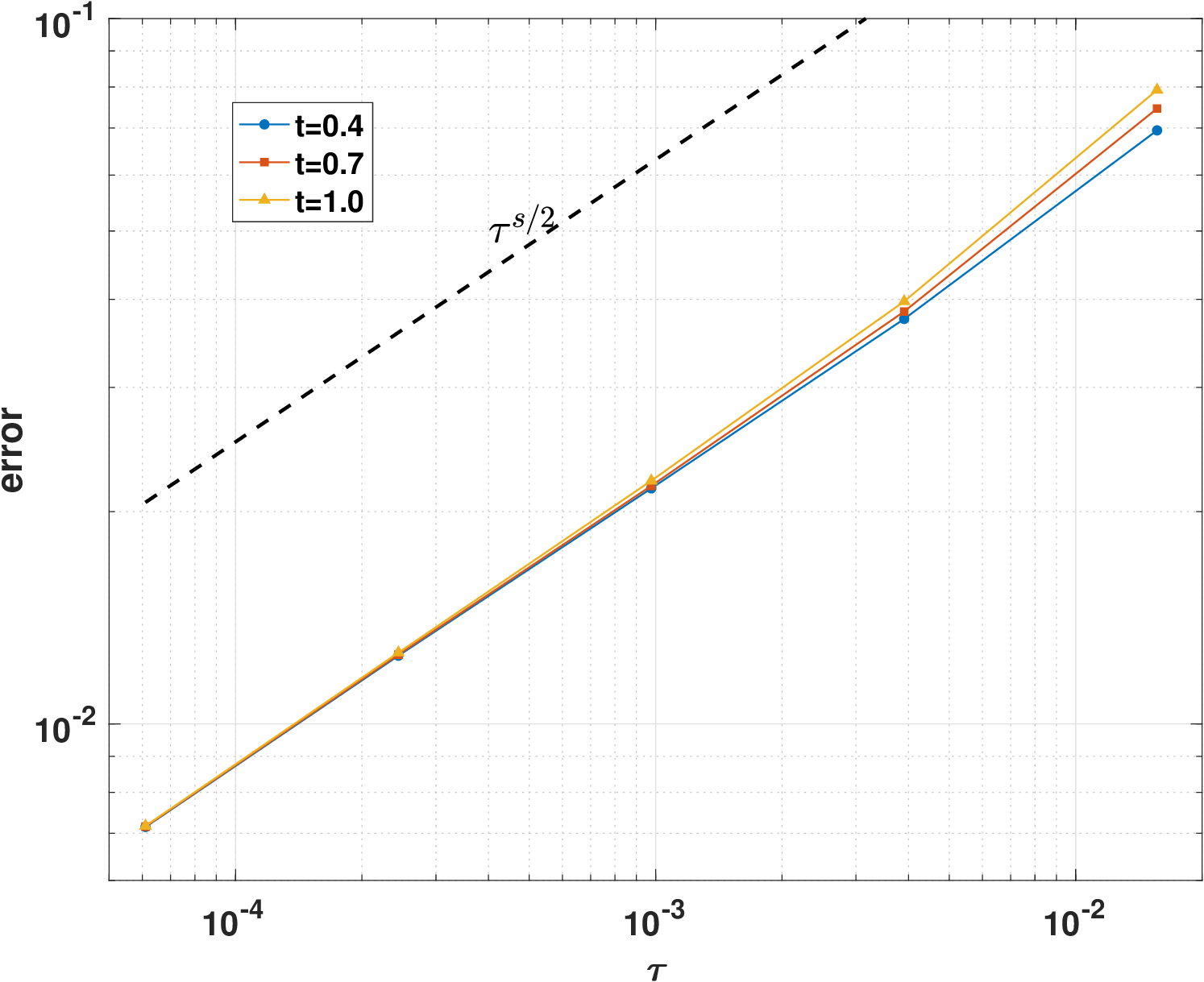}}\qquad 
		\subfigure[Evolution of mass]{
\includegraphics[width=0.42\textwidth,height=0.36\textwidth]{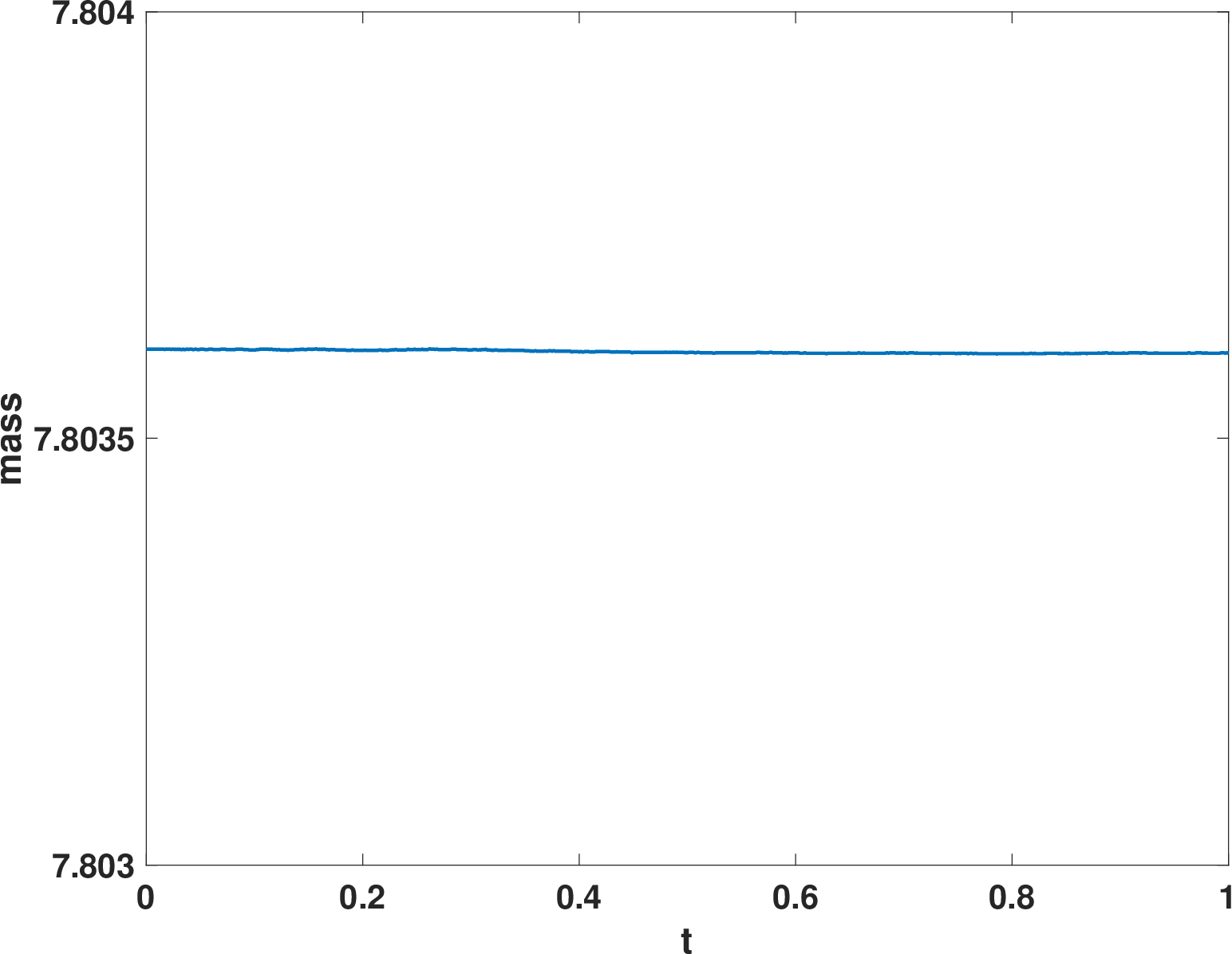}}
		\caption{Numerical results for {\bf Example 2}.  (a)-(b):\, Graphs of
    $u_0(x)$ given in \eqref{u0typeii}  with  $\gamma=0.3, \ell = 2,$ 
    and the reference solution $u_{\rm ref}(x,t)$ at $t=0.4, 0.7,1,$ computed by the scheme with $\tau, N$ given  in \eqref{ref-solu-tau}. (c):\, Error ${\mathbb E}_s^m(\tau)$ against $\tau$  in log-log scale 
    for  $t=m\tau=0.4, 0.7, 1$, and $s=0.8$. (d):\, Evolution of mass for $t\in [0,1].$}\label{fig:Hsu0typeb}
\end{figure}


\subsection{$H^1$-initial data} To demonstrate the convergence behavior  in  
\reft{THM:errestH1} and Corollary \ref{1DH1-est}, we take $u_0\in H^1(\mathbb T)$ by letting $s=1$ in {\bf Example 1} and  $\gamma=1/2$ in {\bf Example 2}. 

\smallskip
\noindent{\bf Example 3: $u_0\in H^1(\mathbb T)$ constructed by  {\bf Example 1} with $s=1$}.\,  We choose $u_0$  in \eqref{ex2:u01A} with $s=1, \beta=0.51$,  and follow the same setting as the previous subsection, but with $s=1.$ Correspondingly, the reference solutions are depicted in Figure \ref{fig:H1u0} (a)-(b), and the convergence and mass conservation are demonstrated  in Figure \ref{fig:H1u0} (c)-(d).  We reiterate that we observe a perfect agreement of convergence order ${\mathcal O}(\sqrt \tau)$ as in \reft{THM:errestH1} and Corollary \ref{1DH1-est}.
We remark that  Bao et al \cite{Bao2019Regularized} also constructed low-regularity $u_0(x)$ through  random coefficients with a specified decaying rate, where $H^1$-regularity was tested.
Here, our construction is more straightforward with a better description of the decay, so the curve of convergence has a much better fitting of the predicted order. 

\smallskip
\noindent{\bf Example 4: $u_0$ in  {\bf Example 2} with $\gamma=1/2$}.\,
Here we set $\gamma =1/2$ and $\ell=2$ in \eqref{u0typeii}. In Figure \ref{fig:H1u0typeb}, we show the curves of reference solution at $t=0.4,0.7,1$, and demonstrate convergence rate in time and conservation of mass. Again, we observe a good agreement with the prediction. 
\begin{figure}[!th]
		\centering
		\subfigure[$\Re\{u_{\rm ref}(x,t)\}$]
{\includegraphics[width=0.40\textwidth,height=0.35\textwidth]{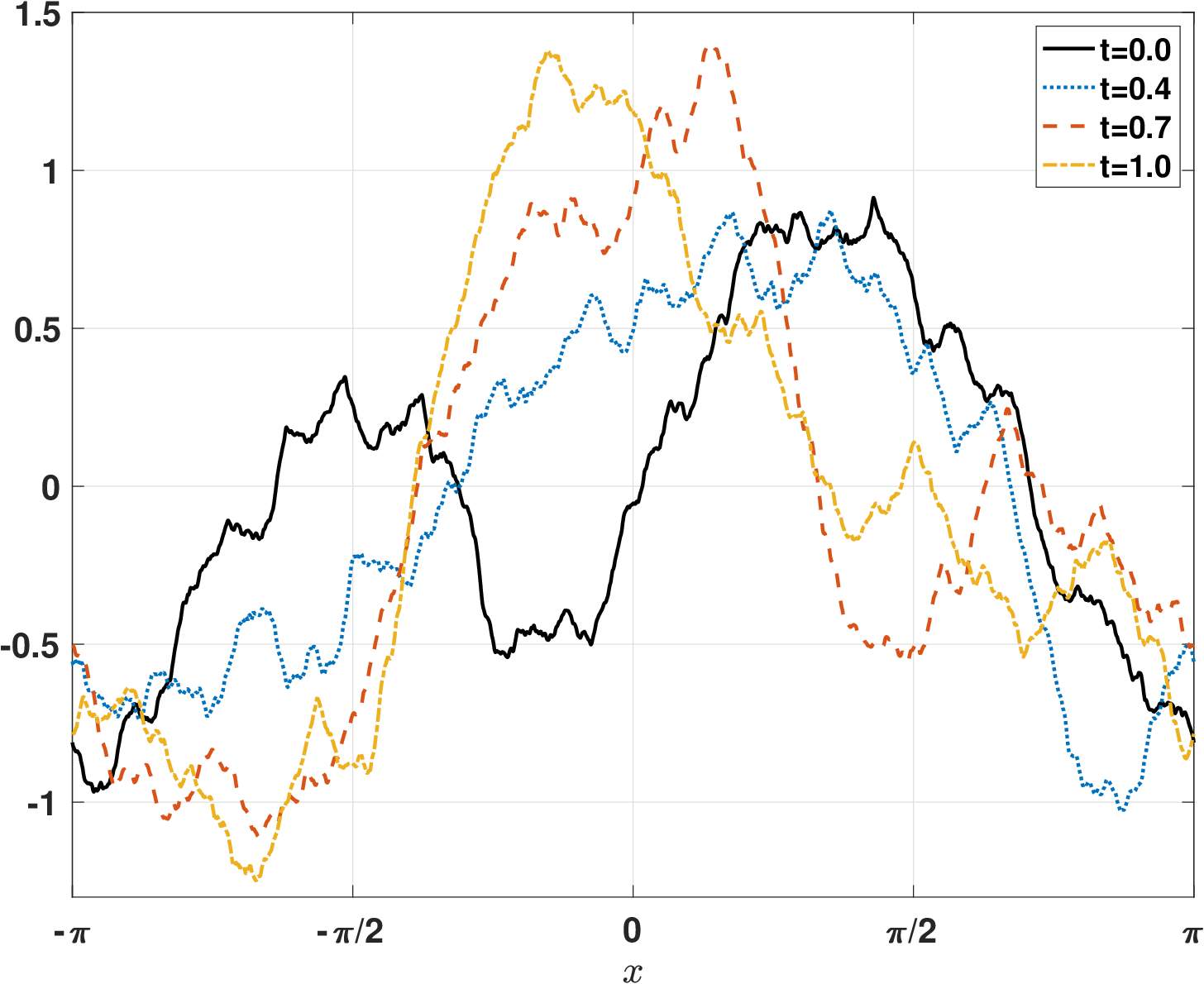}} \quad \qquad 
		\subfigure[$\Im\{u_{\rm ref}(x,t)\}$]
{\includegraphics[width=0.40\textwidth,height=0.35\textwidth]{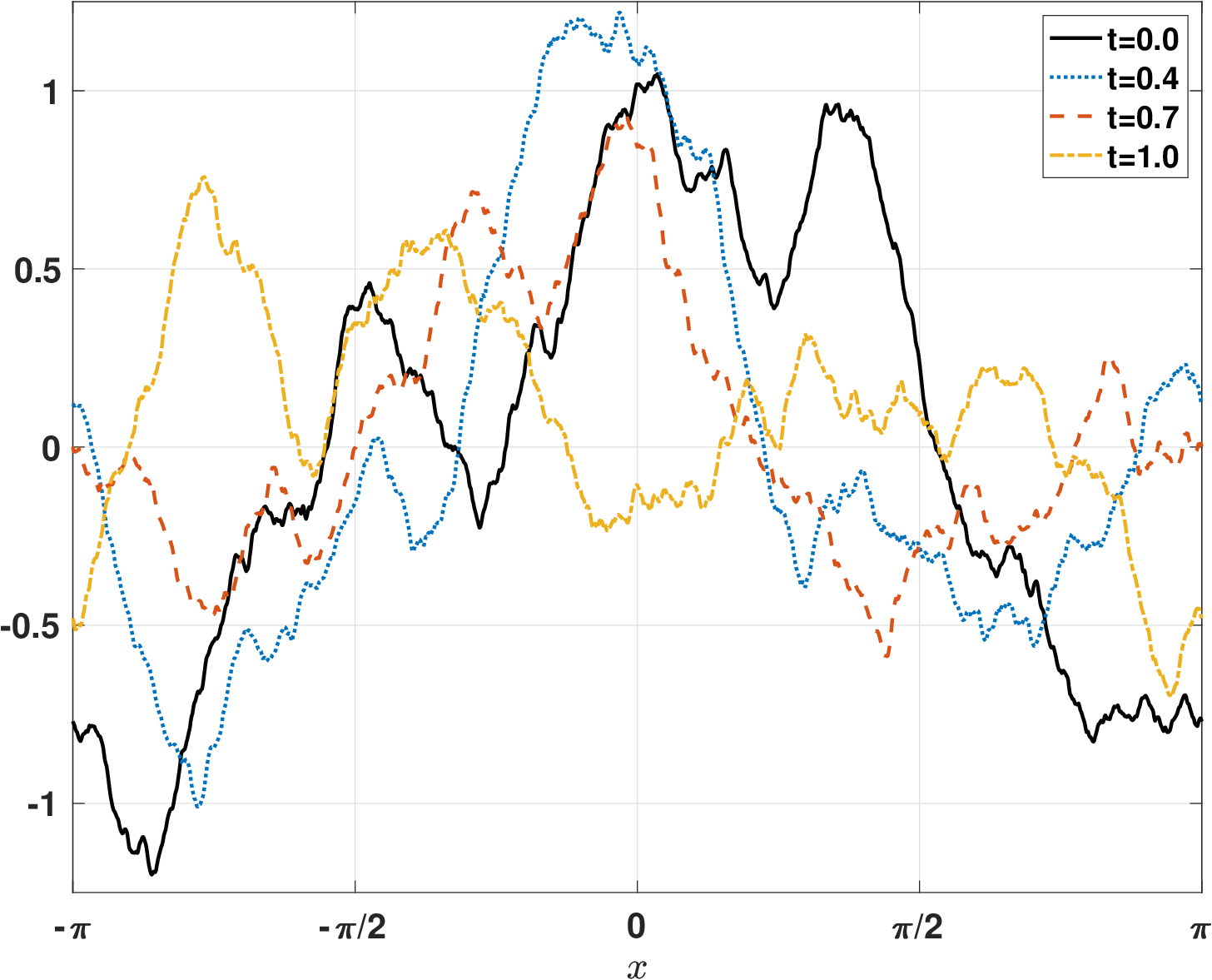}}
		\subfigure[Error ${\mathbb E}_s^m(\tau)$]
{\includegraphics[width=0.42\textwidth,height=0.36\textwidth]{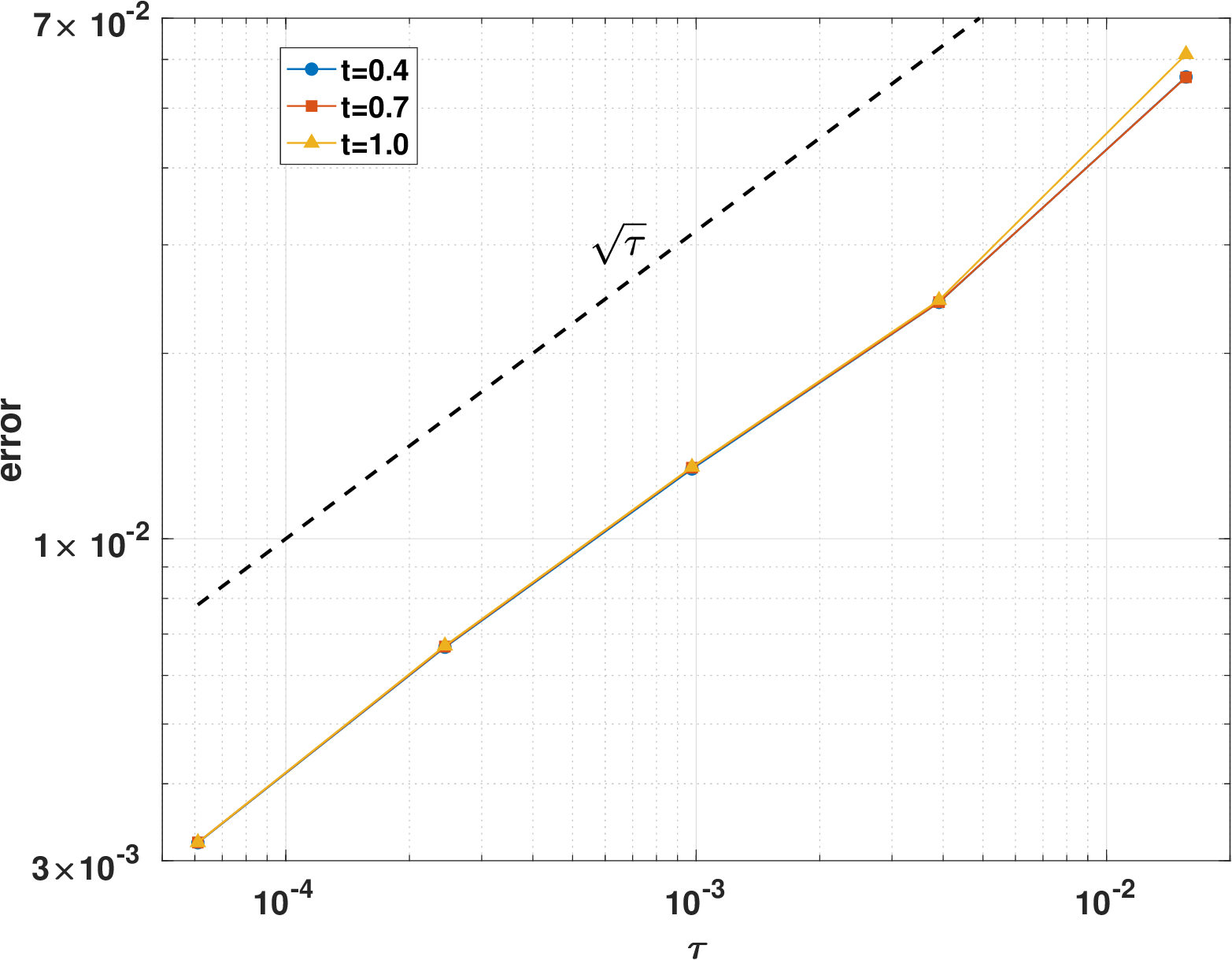}}\qquad 
		\subfigure[Evolution of mass]{
\includegraphics[width=0.42\textwidth,height=0.36\textwidth]{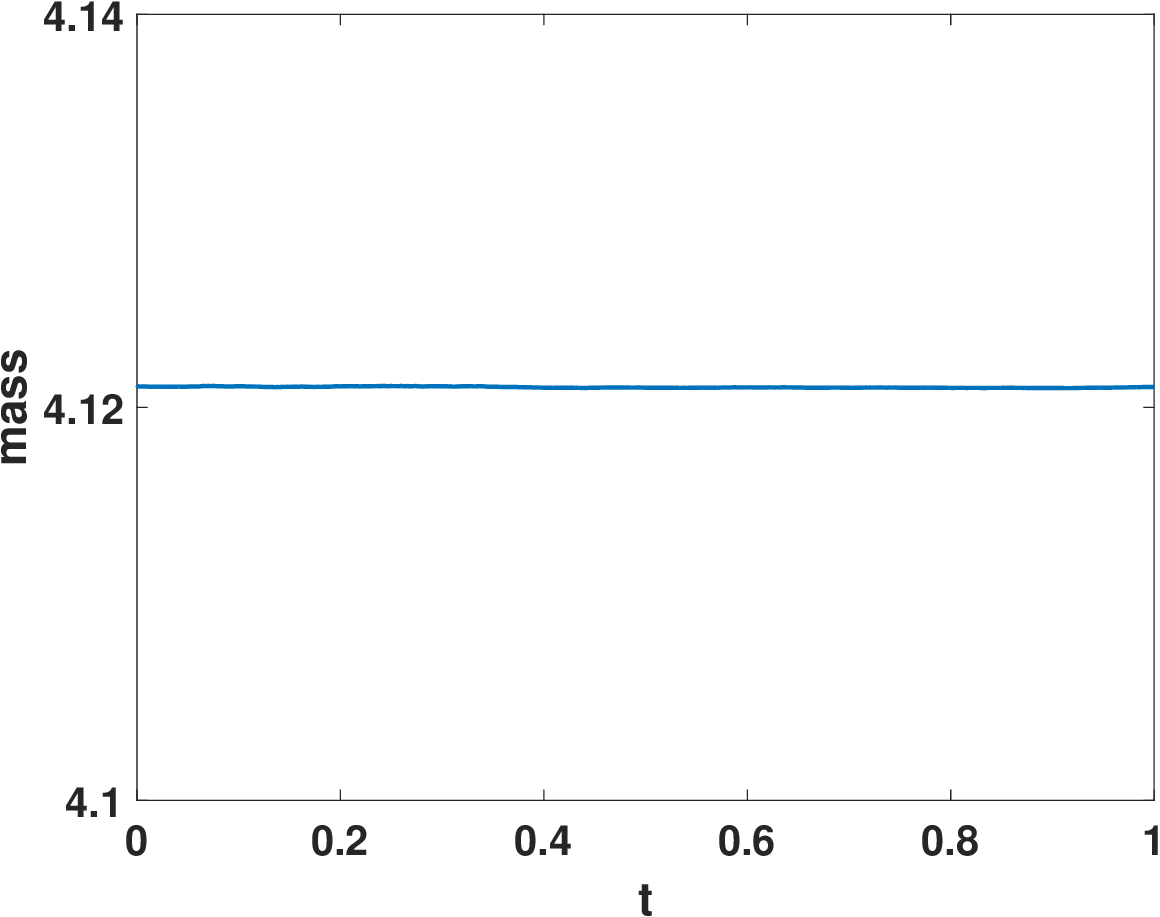}}
		\caption{Numerical results for {\bf Example 3}. (a)-(b):\, Graphs of
    $u_0(x)$ given in \eqref{ex2:u01A}  with  $s=1, \beta = 0.51,$ 
    and the reference solution $u_{\rm ref}(x,t)$ at $t=0.4, 0.7,1,$ computed by the scheme with $\tau, N$ given  in \eqref{ref-solu-tau}. (c):\, Error ${\mathbb E}_s^m(\tau)$ against $\tau$  in log-log scale 
    for  $t=m\tau=0.4, 0.7, 1$. (d):\, Evolution of mass for $t\in [0,1].$}\label{fig:H1u0}
\end{figure}



\begin{figure}[!th]
		\centering
		\subfigure[$\Re\{u_{\rm ref}(x,t)\}$]
{\includegraphics[width=0.40\textwidth,height=0.35\textwidth]{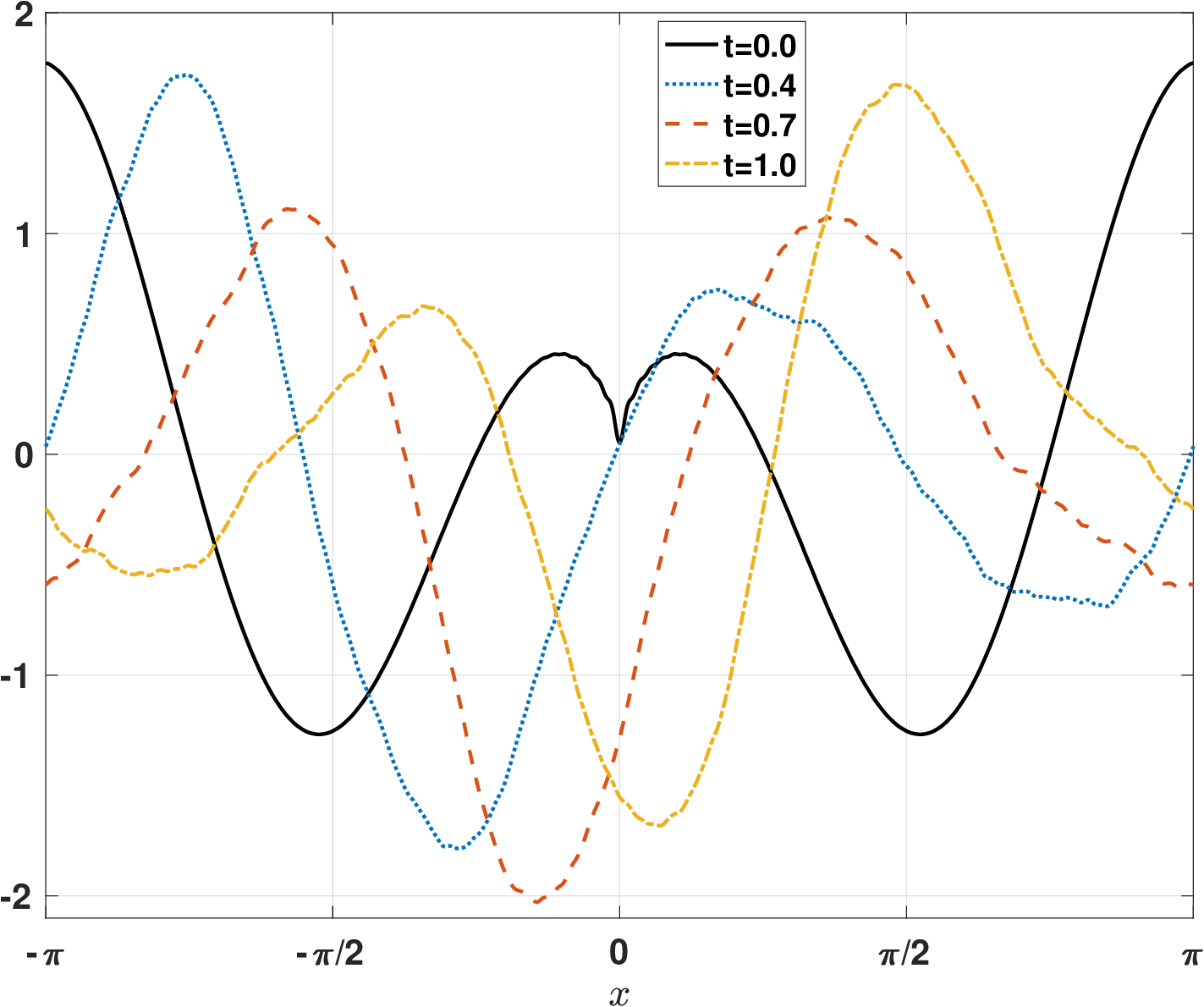}} \quad\qquad 
		\subfigure[$\Im\{u_{\rm ref}(x,t)\}$]
{\includegraphics[width=0.40\textwidth,height=0.35\textwidth]{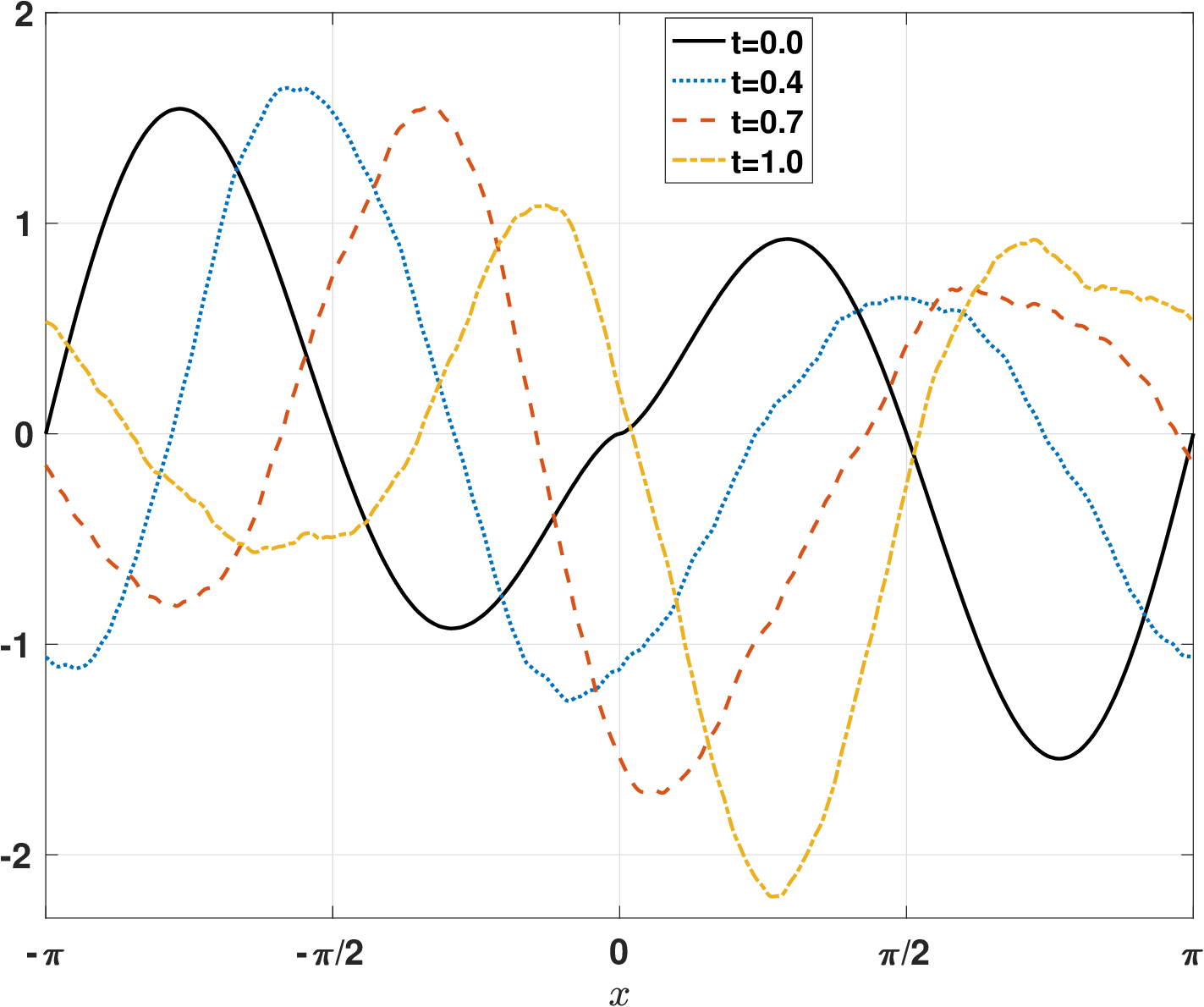}}
		\subfigure[Error ${\mathbb E}_s^m(\tau)$]
{\includegraphics[width=0.42\textwidth,height=0.36\textwidth]{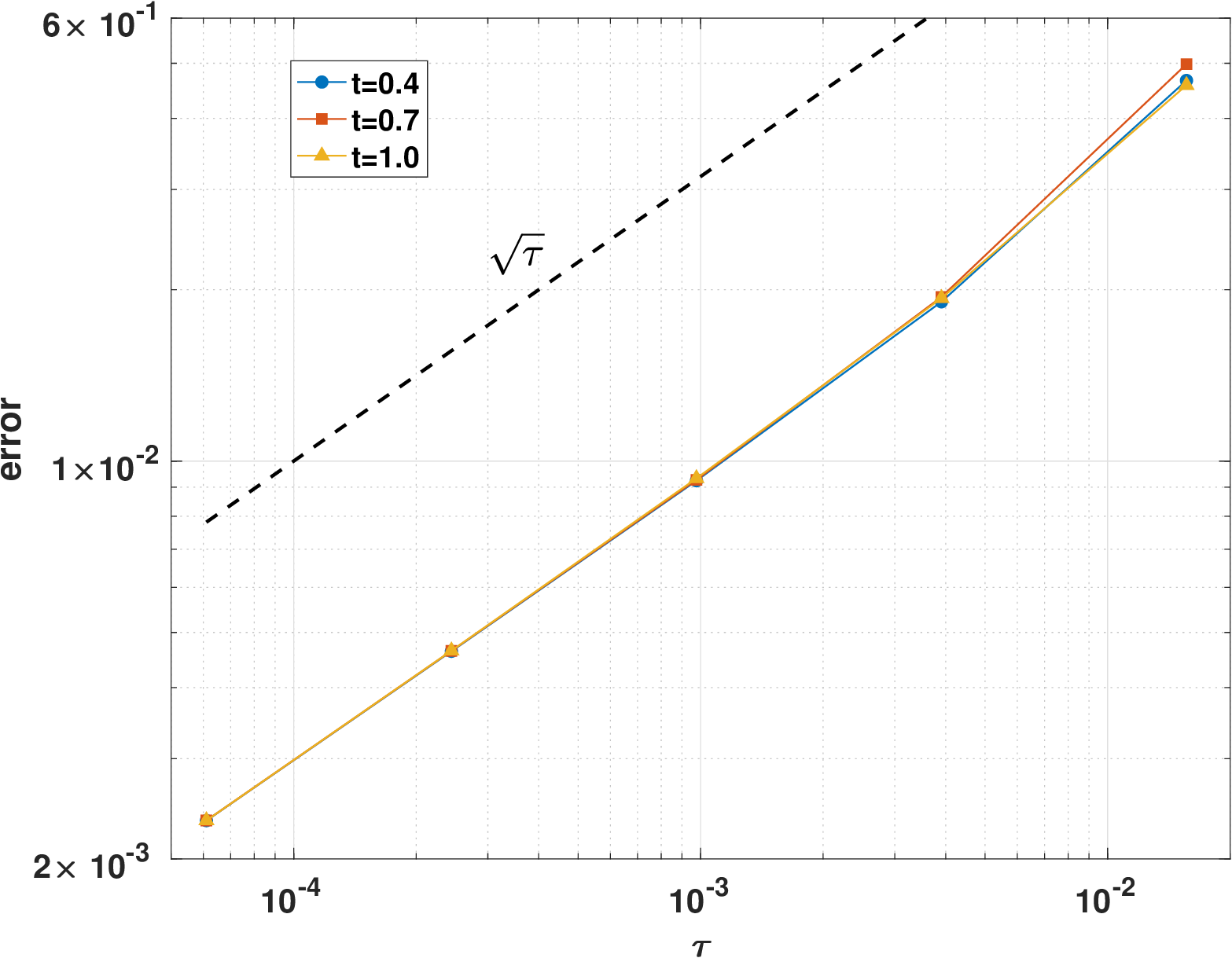}}\qquad 
		\subfigure[Evolution of mass]{
\includegraphics[width=0.42\textwidth,height=0.36\textwidth]{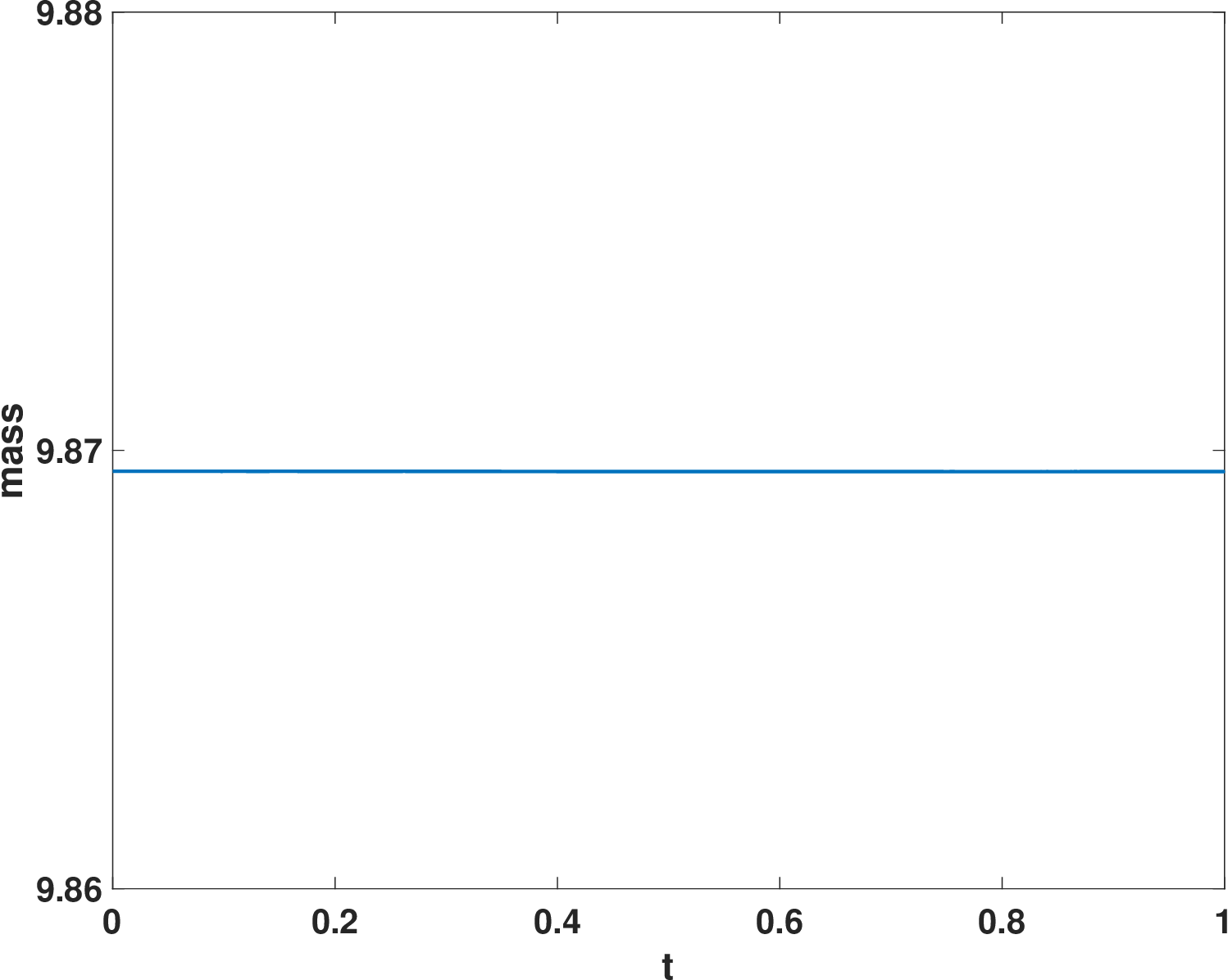}}
		\caption{Numerical results for {\bf Example 4}. (a)-(b):\, Graphs of
    $u_0(x)$ given in \eqref{u0typeii}  with  $\gamma=0.5, \ell = 2,$ 
    and the reference solution $u_{\rm ref}(x,t)$ at $t=0.4, 0.7,1,$ computed by the scheme with $\tau, N$ given  in \eqref{ref-solu-tau}. (c):\, Error ${\mathbb E}_s^m(\tau)$ against $\tau$  in log-log scale 
    for  $t=m\tau=0.4, 0.7, 1$. (d):\, Evolution of mass for $t\in [0,1].$}\label{fig:H1u0typeb}
\end{figure}

\section{Concluding remarks and discussions}\label{sect: Final}

 While we are finalizing this work, we realize  that  Carles et al \cite{carles_low_2023} showed the fractional Sobolev regularity of the LogSE on $\mathbb R^d$ and $\mathbb T^d,$ where the use of  fractional Sobolev norm in 
B$\acute{e}$nyi and Oh  \cite[Proposition 1.3]{BenyiOH2013fracSoblevTorus} (see \refl{lem:normEquiv}) was indispensable to the analysis therein.  Coincidentally, it is crucial in this context too. Here we also require the interplay between frequency and physical domain definitions, as seen in the proofs of the main results.   

There appears a marginal gap between the regularity  theory in  \cite{carles_low_2023,hayashi_cauchy_2023} and regularity requirements in  \reft{THM:errestfrac}  and \reft{THM:errestH1}, which additionally need $u\in C((0,T];L^\infty(\mathbb T^d))$  (i.e., $u$ is bounded,  though it is a sufficient condition). As a result, we could claim the convergence orders in Corollary \ref{1DHs-est}  and Corollary \ref{1DH1-est} for $d=1$  from the regularity results in \cite{carles_low_2023,hayashi_cauchy_2023} on the initial data $u_0$. However, if $u_0\in H^2(\mathbb T^d),$ then it is ensured as $u\in C((0,T];H^2(\mathbb T^d))$  (see e.g., \cite{Bao2019Error}).


 We reiterate that different from the existing works, we analyze for the first time the low regularity fractional order convergence for the non-regularised splitting scheme.  
 We also point out that  we observe from numerical evidences the first-order convergence  when $u_0\in H^2(\mathbb T),$ but the rigorous proof appears open.

\begin{appendix}
\setcounter{equation}{0}
\renewcommand{\theequation}{A.\arabic{equation}}
\section{Proof of \refl{lem:PiNapprox}}\label{AppendixA}

 For any $\phi\in X_N^d$, we write 
\begin{equation*}
\begin{aligned}
    \phi(x) &= \sum_{k\in \mathbb{K}^d_N } \hat \phi_k \re^{\ri k\cdot x} = \frac{1}{|\mathbb{T}|^d}\sum_{k\in \mathbb{K}^d_N } \Big(\int_{ \mathbb{T}^d } \phi(y) \re^{-\ri k \cdot y } \rd y\Big) \,\re^{\ri k\cdot x}\\
    & = \frac{1}{|\mathbb{T}|^d} \int_{\mathbb{T}^d} \phi(y) \sum_{k\in \mathbb{K}^d_N } \re^{\ri k\cdot (x-y)}  \rd y = \frac{2^d}{|\mathbb{T}|^d} \int_{\mathbb{T}^d} \phi(y) \, {\mathbb D}_N(x-y)\, \rd y,
    \end{aligned}
\end{equation*}
where ${\mathbb D}_N(y)$ is the $d$-dimensional Dirichlet kernel
\begin{equation*}
    {\mathbb D}_N(y) := \frac{1}{2^d}\sum_{k\in \mathbb{K}^d_N } \re^{\ri k\cdot y} = \frac{1}{2^d} \prod_{j=1}^d  \sum_{|k_j|=0}^N \re^{\ri k_j y_j} = \prod_{j=1}^d \frac{\sin((N+1/2)y_j)}{2\sin(y_j/2)}.
\end{equation*}
Then we obtain from the Cauchy-Schwarz inequality that 
\begin{equation*}
\| \phi \|_{\infty} \le \frac{2^d}{|\mathbb{T}|^d} \| \phi \|\, \| {\mathbb D}_N \|.
\end{equation*}
By the orthogonality of $\{\re^{{\rm i}k_j x_j}\}$ on $\mathbb{T}$, we have 
\begin{equation*}
\begin{aligned}
     \| {\mathbb D}_N \|^2 & = \frac{1}{4^{d}}\int_{ \mathbb{T}^d } \sum_{k \in \mathbb{K}^d_N} \re^{\ri k \cdot x} \sum_{\ell \in \mathbb{K}^d_N} \re^{-\ri \ell  \cdot x} \rd x 
     = \frac{1}{4^d} \prod_{j=0}^d \sum_{|k_j|,|\ell_j|=
     0}^N \int_{\mathbb{T}} \re^{\ri (k_j-\ell_j)x_j} \rd x_j\\
     &= \frac{ |\mathbb{T}|^d }{4^d} (2N+1)^d.
\end{aligned}
\end{equation*}
Then we obtain \eqref{XNinftyto2} directly.

\setcounter{equation}{0}
\renewcommand{\theequation}{B.\arabic{equation}}
\section{Proof of \refl{lem:errProj}}\label{AppendixB}

By the Parsval's identity, we have
\begin{equation*}
\begin{split}
| u - \Pi_N^d u |_{H^{\mu}({\mathbb T^d}) }^2  & =  |\mathbb{T}|^d\sum_{n\in  \mathbb{Z}^d \backslash \mathbb{K}^d_N }  |n|^{2\mu}  |\hat{ u }_n|^2  \le |\mathbb{T}|^d \Big(\max_{n\in \mathbb{Z}^d \backslash \mathbb{K}^d_N} \!\! {|n|^{2(\mu-s)}}\Big) \sum_{n\in \mathbb{Z}^d \backslash \mathbb{K}^d_N} |n|^{2s} |\hat{ u }_n|^2 \\
& \le N^{2(\mu-s)} | u |^2_{H^{s}(\mathbb T^d)}.  
\end{split}
\end{equation*}
The second  estimate  can be obtained straightforwardly since 
\begin{equation*}
    |\Pi_N^d u |^2_{H^{s} ( \mathbb{T}^d ) } = |\mathbb{T}|^d \sum_{k\in \mathbb{K}_N^d } |k|^{2s} |\hat u_k|^2 \le |\mathbb{T}|^d\sum_{k\in \mathbb{Z}_0^d } |k|^{2s} |\hat u_k|^2 =  | u |^2_{H^{s} ( \mathbb{T}^d ) }. 
\end{equation*}
This ends the proof.

\end{appendix}

\bibliography{References}
\end{document}